\documentclass[reqno, 12pt]{amsart}
\usepackage{amsthm,amsfonts,amssymb,mathrsfs}
\usepackage{datetime}
\usepackage{hyperref}

\usepackage[margin=1in]{geometry}

\newcommand{\bea}{\begin{eqnarray}}
\newcommand{\eea}{\end{eqnarray}}
\def\beaa{\begin{eqnarray*}}
\def\eeaa{\end{eqnarray*}}
\def\ba{\begin{array}}
\def\ea{\end{array}}
\def\be#1{\begin{equation} \label{#1}}
\def \eeq{\end{equation}}

\def\beq{\begin{equation}}

\newcommand{\nn}{\nonumber}

\def\be{{\beta}}

\def\e{\varepsilon}

\def\eps{\epsilon}

\def\si{\sigma}

\def\R{{\mathbb{R}}}
\def\C{{\mathbb{C}}}

\def\T{{\mathbb{T}}}

\def\what{\widehat}

\theoremstyle{plain}

\newtheorem{theorem}{Theorem}[section]
\newtheorem{lemma}[theorem]{Lemma}
\newtheorem{proposition}[theorem]{Proposition}

\newtheorem{remark}[theorem]{Remark}

\numberwithin{equation}{section}

\begin{document}

\title[Asymptotic stability of solitons for mKdV]{Asymptotic stability of solitons for mKdV}

\author{Pierre Germain}
\address{Pierre Germain, Courant Institute of Mathematical Sciences, 251 Mercer Street, New York 10012-1185 NY, USA}
\email{pgermain@cims.nyu.edu}

\author{Fabio Pusateri}
\address{Fabio Pusateri,  Department of Mathematics, Princeton University, Washington Road, Princeton 08540 NJ, USA}
\email{fabiop@math.princeton.edu}

\author{Fr\'ed\'eric Rousset}
\address{Fr\'ed\'eric Rousset, Laboratoire de Math\'ematiques d'Orsay (UMR 8628), Universit\'e Paris-Sud, 91405 Orsay Cedex France et Institut Universitaire de France}
  \email{frederic.rousset@math.u-psud.fr}

%\date{\today, \currenttime}

\begin{abstract}
We prove a full  asymptotic stability result for solitary  wave solutions of the mKdV equation.
We consider small perturbations of solitary waves with %very weak 
polynomial decay at infinity and prove that solutions of the Cauchy problem evolving from such data
tend uniformly, on the real line, to another solitary wave as time goes to infinity.
We describe precisely the asymptotics of the perturbation behind the solitary wave
showing that it satisfies a nonlinearly modified scattering behavior.
This latter part of our result relies on a precise study of the asymptotic behavior of small solutions of the mKdV equation.
\end{abstract}

\subjclass[2000]{}
\thanks{\noindent P. G. is partially supported by NSF grant DMS-1101269, a start-up grant from the Courant Institute, and a Sloan fellowship.
F. P. is partially supported by NSF grant  DMS-1265875.}

\maketitle

\setcounter{tocdepth}{1}
\tableofcontents
\setcounter{secnumdepth}{2}

\section{Introduction}

This paper is concerned with the Cauchy problem for the focusing modified Korteweg-de Vries (mKdV) equation
\begin{align}
\label{mKdV}
\tag{mKdV}
\left\{ \begin{array}{l} \partial_t u + \partial_x^3 u + \partial_x (u^3) = 0
\\
\\
u(t=0) = u_0 \end{array} \right.
\end{align}
for $u = u(t,x) \in \mathbb{R}$ and $(t,x) \in \mathbb{R} \times \mathbb{R}$.
%Historical reference to the equation?
This equation admits a  family of solitary wave  solutions of the form $u_{c}(t,x)= Q_{c}(x-ct)$ with
\begin{equation}
\label{Q_c}
 Q_{c}(\xi) = \sqrt{c} Q (\sqrt{c} \, \xi  ) , \qquad Q(s) := \sqrt{2}/\cosh (s) , \quad c>0.
\end{equation}
Our aim in this paper is to revisit the proof of global existence and modified scattering for \eqref{mKdV} for small and localized initial data,
and then extend it in order to obtain new asymptotic stability results for solitary wave solutions.

Important  conserved quantities\footnote{
As we will remark later, these are not needed to prove the small data result, which also  applies
%to the defocusing case, i.e. $ \partial_t u + \partial_x^3 u - \partial_x (u^3) = 0$, and 
to more general versions of \eqref{mKdV}.}
are the mass $M$, energy $H$, and momentum $P$
\begin{equation}
\label{ME}
M = \int_\mathbb{R} u^2 \, dx \qquad H = \int_\mathbb{R} \frac{1}{2} |\partial_x u|^2 - \frac{1}{4} |u|^4 \, dx,\qquad P = \int_\mathbb{R} u\,dx.
\end{equation}
Moreover, we note that solutions of \eqref{mKdV} enjoy the scaling symmetry
\begin{equation*}
u \longmapsto \lambda u (\lambda^3 t, \lambda x) ,
\end{equation*}
which is generated by the vector field $S = 1 + x\partial_x + 3 t \partial_t$.

\medskip
\subsection{Known results}

\subsubsection*{Global well-posedness and asymptotic behavior }
There is a vast body of literature dealing with  the mKdV equation,
and in particular with the local and global well-posedness of the Cauchy problem.
%as well as the related problems for the KdV, $\partial_t u + \partial_x^3 u + u \partial_x u = 0$,
%and generalized KdV $\partial_t u + \partial_x^3 u + u^k \partial_x u = 0, k \geq 3$ equations.
Without trying to be exhaustive, we mention the early works on the local and global well-posedness
by Kenig-Ponce-Vega \cite{KPV1} and Kato \cite{Kato}.
Global well-posedness in low regularity spaces, and in particular in the energy space $H^1$, was established
in the seminal work of Kenig-Ponce-Vega \cite{KPV2}.
In this latter paper the authors considered the wider class of generalized KdV (gKdV) equations
$\partial_t u + \partial_x^3 u + \partial_x u^p  = 0$, $p \geq 2$, which includes \eqref{mKdV} and the KdV equation ($p=2$).
Sharp, up to the end-point, global well-posedness in $H^s$ for $s>1/4$
was proved in the work of Colliander-Keel-Staffilani-Takaoka-Tao \cite{CKSTTKdV},
for both the focusing and defocusing mKdV equation on the line (and for $s \geq 1/2$ in the periodic case).
These results are complemented by several ill-posedness results; see for example
Christ-Colliander-Tao \cite{CCT} and references therein.\footnote{For more on the local and global well-posedness
and ill-posedness of KdV and generalized KdV equations
we refer to the books of Tao \cite{TaoBook} and Linares-Ponce \cite{LiPoBook}.} %% Can avoid this footnote

Besides global regularity, another fundamental question for dispersive PDEs concerns the asymptotic behavior for large times.
The first proof of global existence with a complete description of the asymptotic behavior of solutions of~\eqref{mKdV} in the defocusing case, is due to Deift and Zhou \cite{DZmKdV}, who used a steepest descent approach to oscillatory Riemann-Hilbert problems
and the inverse scattering transform \cite{ZakMan,AKNS}.
In \cite{DZmKdV}, thanks to the complete integrability of the defocusing mKdV equation,
the authors were able to treat suitably localized initial data with arbitrary size.
A proof of global existence and a (partial) derivation of the asymptotic behavior for small localized solutions,
without making use of complete integrability, was later given by Hayashi and Naumkin \cite{HNmKdV1,HNmKdV2},
following the ideas introduced in the context of the $1$d nonlinear Schr\"odinger (NLS) equation in \cite{HN}.
Recently, an alternative proof of the results in \cite{HNmKdV2}, with a precise derivation of asymptotics 
and a proof of asymptotic completeness, was given by Harrop-Griffiths \cite{HG}, following the approach used for the $1$d NLS equation in \cite{ITNLS}.

Our proof of global existence and asymptotic behavior - Theorem \ref{maintheo1} 
- relies on the intuition developed in \cite{KP}, where a very natural stationary phase argument is
used to understand the large time behavior of small and localized solutions and derive asymptotic corrections. 
This approach was inspired by the space-time resonance method put forward in~\cite{GMS1,GMS1a,GMS2}.
See section \ref{secideaglobal} below for a short explanation of these ideas in the present context.
A similar approach was also successfully employed in the proofs of global regularity and
modified scattering for 2d gravity \cite{IoPu1,IoPu2,IoPunote} and capillary \cite{IoPu3,IoPu4} water waves, 
and in other higher dimensional dispersive models \cite{KP,BosonStar}.

\subsubsection*{Stability of solitons}
The study of the stability of solitons also has a long  history,
but here we will only address results which are closer in spirit to the present paper.
The asymptotic stability in front of the soliton\footnote{Similarly to how it is stated in Theorem \ref{theosoliton2}.}
was first obtained by Pego-Weinstein \cite{Pego-Weinstein1} for initial perturbations of a soliton
with exponential decay as $x \rightarrow +\infty$.
This result was then refined by Mizumachi \cite{Mizumachi1}, who treated perturbations belonging to
polynomially weighted spaces of sufficiently high order.
For perturbations in the energy space $H^1$,
definitive asymptotic stability results in front of the solitary wave have been obtained for the whole class of subcritical gKdV equations
in a series of papers by Martel-Merle \cite{Martel-Merle,Martel-MerleRev,Martel-Merle2}.
 We also mention \cite{Merle-Vega} on the $L^2$ stability of KdV solitons, \cite{Buckmaster-Koch} on the $H^s$ $s \geq -1$ stability of KdV solitons,  \cite{MaMeTs} on $N$-soliton solutions
of subcritical gKdV equations, and  \cite{Mizumachi-Tzvetkov,Mizumachi2} for a different approach.
For more on the asymptotic stability of solitons and multi-solitons for subcritical gKdV equations
we refer the reader to the survey articles \cite{TaoSolitons,Martel-MerleReview} and references therein.

In \cite{Mizumachi1} the author also obtained a full stability result 
%in the sense that he was able to describe the dynamics behind the solitary wave by proving scattering of the perturbation 
for gKdV equations with a nonlinearity of degree $p \in (3,5)$.
More precisely, he showed that a solution that evolves from a small perturbation of a soliton will asymptotically resolve in a
slightly differently modulated soliton, plus a radiation which behaves like a solution of the linear flow.
Note that for the gKdV equation with quartic nonlinearity ($p=4$), there are also scattering and asymptotic stability results
in critical spaces rather than polynomially weighted ones, see \cite{TaoSolitons2,Koch}.

The results we present extend the above mentioned works by
\setlength{\leftmargini}{2em}
\begin{itemize}
\item[i)] proving the (modified) scattering result for the radiation in the (critical) case of \eqref{mKdV},
\item[ii)] allowing a wider class of small perturbations belonging to weighted Sobolev spaces with 
  weak polynomial decay at infinity.
\end{itemize}
Because of the critical dispersive nature of the equation, in the case of \eqref{mKdV}
the radiation does not behave linearly, but requires a nonlinear correction.
See Theorem \ref{maintheo2} and Remark \ref{rkcondition} for more details.
 The proof %of asymptotic stability and scattering of the radiation
that we give below combines the virial approach of Martel-Merle \cite{Martel-MerleRev}
and the weighted estimates of Pego-Weinstein \cite{Pego-Weinstein1},
in the spirit of  the recent work of Mizumachi-Tzvetkov \cite{Mizumachi-Tzvetkov} on the $L^2$ stability of solitons for KdV.

\medskip
\subsection{Main results}

%\subsubsection*{Global existence and modified scattering}
Our first main result concerns the stability of the zero solution under small perturbations.

\begin{theorem}[Global Existence and Asymptotic Behavior]\label{maintheo1}
Let an initial data $u_0$ be given such that
\begin{align}
\label{initdata}
 {\| \langle x \rangle u_0 \|}_{H^1(\R)} \leq \e_0 .
\end{align}
There exists $\overline{\e_0} > 0$, such that for all $\e_0 \in (0, \overline{\e_0}]$ the Cauchy problem \eqref{mKdV}
admits a unique global solution $u \in C(\R,H^1(\R))$.
This solution satisfies the decay estimates
\begin{align}
\label{mainthoedecay}
| u(t,x) | \lesssim \e_0 t^{-1/3} {\langle x/t^{1/3} \rangle}^{-1/4} , \qquad
  | \partial_{x} u(t,x) | \lesssim  \e_0 t^{-2/3} {\langle x/t^{1/3} \rangle}^{1/4} .
\end{align}
Moreover, for $t \geq 1$ the solution $u$ has the following asymptotics:
\setlength{\leftmargini}{1.5em}
\begin{itemize}
\item In the region $x \geq t^{1/3}$ we have the improved decay
\begin{align}
\label{asy1}
| u(t,x) | \lesssim  \frac{\e_0}{t^{1/3} (x/t^{1/3})^{3/4}} ;
\end{align}

\item In the region $|x| \leq t^{1/3+2\gamma}$, for some $\gamma>0$ sufficiently small, %$\gamma = 1/3 ( 1/6 - C\e_1^2)$,
the solution is approximately self-similar:
\begin{align}
\label{asy2}
\big| u(t,x) - \frac{1}{t^{1/3}} \varphi\big( \frac{x}{t^{1/3}} \big) \big| \lesssim  \frac{\e_0}{t^{1/3 + 3\gamma/2}} ,
\end{align}
where $\varphi$ is a bounded solution of the Painlev\'e II equation\footnote{The smallness of $\int \varphi(x)\,dx$
guarantees the existence and uniqueness of a bounded solution to the Painlev\'e II equation. Its asymptotics are as follows:
$\varphi(\xi) \sim 3^{1/9} \operatorname{Ai}(3^{1/3} \xi) \sim \frac{3^{7/36}}{2 \sqrt{\pi} \xi^{1/4}} e^{-\frac{2}{3 \sqrt 3} \xi^{3/2}}$
as $\xi \to \infty$, while
$\varphi(\xi) \sim \frac{3^{7/36}}{2 \sqrt{\pi} |\xi|^{1/4}} d \cos \left( -\frac{2}{3 \sqrt 3} |\xi|^{3/2}
+ \frac{\pi}{4} + \frac{3d^2}{4\pi} \log |y|^{3/2} + \theta \right)$ as $\xi \to -\infty$,
where $d$ and $\theta$ are constants depending on $\int \varphi(x) \,dx$.
We refer to \cite{HMPII} and \cite{DZPII} for this and much more on Painlev\'e II. Note that our proof will actually provide
 the existence of a bounded solution of the Painlev\'e equation.}
\begin{align*}
\varphi^{\prime\prime} - \frac{1}{3}\xi \varphi + \varphi^3 = 0 , \qquad  \int_\R \varphi(x) \, dx = \int_\R u_0(x) \, dx .
\end{align*}

\item In the region $x \leq - t^{1/3+2\gamma}$, the solution has a nonlinearly modified asymptotic behavior:
there exists $f_{\infty} \in L^{\infty}_\xi$ such that
\begin{align}
\label{modscatt}
\Big| u(t,x) - \frac{1}{\sqrt{3 t \xi_0}} \Re
  \exp \Big( -2 i t \xi_0^3 + \frac{i\pi}{4} + \frac{i}{6}|f_\infty (\xi_0) |^2 \log t \Big) f_\infty (\xi_0) \Big|
  \leq \frac{\e_0}{t^{1/3} (-x/t^{1/3})^{3/10}} , %% could not get sharp -3/8 power as in \cite{HG} ...
\end{align}
where $\xi_0 := \sqrt{-x/(3t)}$, and $\Re$ denotes the real part.
\end{itemize}

\end{theorem}

\begin{remark}\normalfont
In the proof of Theorem \ref{maintheo1} above, the Hamiltonian structure of the equation,
as well as the conservation of mass and energy, do not play any crucial role.
For convenience we will work with \eqref{mKdV}, but it will be clear that all our results also apply to the defocusing mKdV equation
$\partial_t u + \partial_x^3 u - \partial_x (u^3) = 0$, and to other (not necessarily Hamiltonian) versions of the equation, such as
\begin{align}
\label{mKdV'}
\partial_t u + \partial_x^3 u = a(t)\partial_x (u^3), %+ \partial_x f(u)  ,
\end{align}
where $|a(t)| \leq 1, |a^\prime(t)| \leq \langle t \rangle ^{-7/6}$. 
%and $f = O(|u|^k)$, for some $k > 3$, is a short-range nonlinear term.  %% not sure about this...maybe not so straightforward
\end{remark}

\begin{remark}%[On the class of initial data]
\label{rempartialalpha}\normalfont
In Theorem \ref{maintheo1} we have decided to state the global existence and scattering result 
for initial data satisfying \eqref{initdata}, that is $x u_0 \in H^1$.
However, in the proof we only make use of the assumption $x u_0 \in H^{\alpha}$, 
for some $\alpha$ close to, but less than, $1/2$. %This can be easily verified by inspection of our proof.
We can therefore treat a larger class of initial data with respect to \cite{HNmKdV2,HG}.
Nevertheless, we have decided to state Theorem \ref{maintheo1} assuming the stronger initial condition \eqref{initdata},
in order to make its application in the proof of Theorem \ref{maintheo2} below more convenient.
\end{remark}

\begin{remark}\normalfont
We chose to characterize the modified asymptotic behavior of $u$ \eqref{modscatt}
in $L^\infty$, but statements analogous to \eqref{asy1}-\eqref{modscatt} can be obtained for $L^2$-type norms.
\end{remark}

%\comment{Fred: but in the end there is a pointwise description in the physical space in the statement ?
%Pierre: I am not sure I see what you mean, but the theorem gives pointwise bounds}

\medskip
Our second main result is a strong asymptotic stability result for soliton solutions,
under small perturbations belonging to a weak algebraically weighted space.

\begin{theorem}[Full Asymptotic Stability of Solitons]\label{maintheo2}
Assume that
\begin{align}
u_0(x) = Q_{c_0}(x) + v_0(x) ,
\end{align}
for some $c_0 > 0$, with
\begin{align}
\label{assv_0}
{\|\langle x \rangle v_0\|}_{H^1(\R)} +  {\| \langle x_+ \rangle^m v_0\|}_{H^1(\R)} \leq \e_0 ,
\end{align}
for some $m > 3/2$.
Then, for $\e_0$ sufficiently small, there exists a unique solution $u \in C(\R,H^1(\R))$ of \eqref{mKdV}
 and a continuous function $C(\cdot)$ with $C(0)=0$, such that for some $c_+>0$ and  $x_+$ with
\begin{align}
|c_+ - c_0| + |x_+ | \lesssim  C( \e_0) ,
\end{align}
we have
\begin{align}
\label{radiation1}
| u(t,x) - Q_{c_+}(x- c_+ t -x_+) - R(t,x) | \lesssim C (\e_0) \langle t \rangle^{-1}
\end{align}
where:
\setlength{\leftmargini}{1.5em}
\begin{itemize}

\item The radiation $R$ verifies the decay estimates
\begin{align}
\label{radiation2}
%{\| R (t) \|}_{H^1(\R)} \lesssim \e_0, \quad
{\| R(t) \|}_{L^\infty(\R)} \lesssim C( \e_0) \langle t \rangle^{-1/3}.
\end{align}

\item $R$ has the same asymptotics as a small solution to \eqref{mKdV},
and in particular possesses a modified scattering behavior as $t \rightarrow \infty$ as in \eqref{modscatt}.

\end{itemize}
\end{theorem}

%\comment{(Fabio) Can write \eqref{radiation1} in a better form?
%Maybe also a few lines to justify it after the Theorems \ref{theosoliton} and \ref{theosoliton2}, or just make cross references to those.}

Note that we prove a full asymptotic stability result by describing the behavior of the perturbation behind the solitary wave and that,
because of the critical dispersive decay of the mKdV equation, the radiation has nonlinear asymptotic oscillation.

\begin{remark}\normalfont
\label{rkcondition} %Almost optimality of requirements in the perturbation.
Note that the spatial decay \eqref{assv_0} that we require in front of the solitary is only slightly more than $x^{3/2} u$ in $L^2$.
This is at the same scale as the decay property which is used in the inverse scattering theory,
where one requires $x u_{0}$ in $L^1$. 
Spatial decay conditions on the data are not explicitly stated in the work of Deift-Zhou on the defocusing mKdV equation \cite{DZmKdV},
but the condition above is used in the application of direct and inverse scattering in \cite{AblowitzBook,Melin}.
We also refer to \cite{Klein-Saut} for a recent survey.
\end{remark}

%\begin{remark}\normalfont
%\label{rkscattering}
%Strong asymptotic stability, as opposed to orbital, or stability in front of the solitary wave...
%We also improve the stability in front of the soliton in terms of requirements on the perturbation
%Scattering for radiation... Comparison with Mizumachi \cite{Mizumachi1}...
%\end{remark}

\medskip
\subsection{Ideas of the proof}

We now briefly explain the main ideas and the intuition behind our results.

\subsubsection{Global existence and modified scattering}\label{secideaglobal}
In what follows we let
\begin{align}
 \label{prof}
f(t) := e^{t\partial_x^3} u(t)
\end{align}
so that $\partial_t f = - e^{t \partial_x^3}\partial_x (u^3)$.
Then we can write \eqref{mKdV} as
\begin{equation}
\label{toucan}
\partial_t \widehat{f}(t,\xi) = - \frac{1}{2\pi} \iint e^{-it \phi(\xi,\eta,\sigma)}
 i \xi \, \widehat{f}(t,\xi-\eta-\sigma) \widehat{f}(t,\eta) \widehat{f}(t,\sigma) \,d\eta \, d\sigma.
\end{equation}
with
$$
\phi(\xi,\eta,\sigma) = \xi^3 - (\xi-\eta-\sigma)^3 - \eta^3 - \sigma^3 = 3 (\eta + \sigma)(\xi-\eta)(\xi-\sigma).
$$
We follow the approach of the space-time resonances method \cite{PG}
which is to view the above integral as an oscillatory integral,
whose large-time behavior will thus be dictated by the stationary points (in $\eta$, and in $t$ after time integration) of the phase $\phi$.
As observed in \cite{KP}, this will give a very simple means of computing the large-time correction to scattering,
due to the long-range effects of the critically dispersing nonlinear term.

%\footnote{
%In recent years several related problems concerning modified scattering for nonlinear dispersive equations
%have been solved by various techniques, such as the spatial methods used by Lindblad-Soffer for Klein-Gordon and NLS \cite{LS},
%the Fourier/stationary phase methods used in \cite{IoPu2,KP,BosonStar}, and mixed methods \cite{DelortKG1d,ADb,IT}.
%The method we use here, which is close in spirit to \cite{KP},
%allows us to get a fairly simple proof of global existence and modified scattering by using standard analytical tools.
%For example, in contrast with some of the above mentioned works, we only make a limited used of frequency decompositions,
%and avoid unnecessary spatial decompositions.
%%It is worth mentioning that frequency/spatial decompositions, originally introduced in
%%the works of Ionescu and Pausader \cite{IP1,IP2} are still fundamental in tackling a large class of quasilinear dispersive problems.
%}

Before explaining this argument, we need to describe precisely the stationary points of $\phi$. A small computation gives that
\begin{align}
\label{gradphi}
\left\{ \begin{array}{l} \partial_\eta \phi(\xi,\eta,\sigma) = 3(\xi-\sigma)(\xi-2 \eta - \sigma)
\\
\\
\partial_\sigma \phi(\xi,\eta,\sigma) = 3(\xi-\eta)(\xi- \eta - 2 \sigma) \end{array} \right.
\end{align}
and
\begin{align}
 \label{Hess}
\det \operatorname{Hess}_{\eta,\sigma} \phi = -36 (\eta^2 + \sigma^2 + \eta \sigma - \xi \eta - \xi \sigma ).
\end{align}
Notice that
\begin{align}
\label{statpoint0}
\partial_\eta \phi = \partial_\sigma \phi = 0
\quad \Longleftrightarrow \quad (\eta,\sigma) = (\eta_i,\sigma_i) \quad , \quad 1\leq i\leq 4
\end{align}
with
\begin{align}
\label{statpoint}
\left\{
\begin{array}{l} (\eta_1,\sigma_1) = (\xi,\xi) \\ (\eta_2,\sigma_2) =
 (\xi,-\xi) \\ (\eta_3,\sigma_3) = (-\xi,\xi) \\ (\eta_4,\sigma_4) = \left( \frac{\xi}{3},\frac{\xi}{3} \right),
\end{array} \right.
\end{align}
and that, for $i \in \{1,2,3\}$,
\begin{align}
\label{statphase}
\begin{split}
& \phi(\xi,\eta_i,\sigma_i) = 0 \\
& \phi(\xi,\eta_4,\sigma_4) = (8/9) \xi^3 \\
& \operatorname{det} \operatorname{Hess}_{\eta,\sigma} \phi (\xi,\eta_i,\sigma_i) = -36 \xi^2 \\
& \operatorname{det} \operatorname{Hess}_{\eta,\sigma} \phi (\xi,\eta_4,\sigma_4) = 12 \xi^2 \\
& \operatorname{sign} \operatorname{Hess}_{\eta,\sigma} \phi (\xi,\eta_i,\sigma_i) = 0 \\
& \operatorname{sign} \operatorname{Hess}_{\eta,\sigma} \phi (\xi,\eta_4,\sigma_4) = 1 - \operatorname{sign} \xi,
\end{split}
\end{align}
where $\operatorname{sign} M$ is the signature of the matrix $M$.

The above computations are the basis to derive - heuristically for the moment - the large time behavior of $f$.
By the stationary phase lemma, assuming that $\widehat{f}$ is sufficiently smooth, \eqref{toucan} implies that
\begin{align}
\label{d_tf0}
\partial_t \widehat{f}(t,\xi) = \frac{i \operatorname{sign} \xi}{6t} |\widehat{f}(\xi)|^2 \widehat{f}(\xi)
+ \frac{ic}{t} e^{-it \frac{8}{9} \xi^3} \widehat{f} \left( \frac{\xi}{3} \right)^3 + \{ \mbox{integrable terms} \}.
\end{align}
The second summand on the right-hand side should not be asymptotically relevant, due to the time-oscillating term.
Thus the above reduces to
$$
\partial_t \widehat{f}(t,\xi) \sim \frac{i \operatorname{sign} \xi}{6t} |\widehat{f}(\xi)|^2 \widehat{f}(\xi) + \{ \mbox{integrable terms} \},
$$
from which we infer that $|\widehat{f}|$ converges as $t \to \infty$ to an asymptotic profile $F$, while
$$
\widehat{f}(t,\xi) \sim F(\xi) \exp \Big(i \frac{\operatorname{sign} \xi}{6t} |F(\xi)|^2 \log t \Big), \quad \mbox{as $\,\, t \to \infty$}.
$$

\subsubsection{Asymptotic stability of solitons}\label{secideasoliton}
The key idea when proving asymptotic stability for the soliton will be the following decomposition
$$
u(t,x) = Q_{c(t)}(y) + \underbrace{v_1(t,y) + v_2(t,y)}_{v(t,y)}.
$$
We now explain precisely how the new coordinate $y$, the soliton parameter $c$, and the radiation part $v = v_1 + v_2$ are determined.
First, the new coordinates
$$y = x - \int_0^t c(s)\,ds + h(t)$$
correspond to adopting as a reference the moving frame of the soliton;
the modulation parameters $c$ and $h$ will be chosen below to ensure a certain cancellation.
The radiation part $v$ satisfies the perturbed equation
$$
\left\{ \begin{array}{l}
\partial_t v + \underbrace{ \partial_y (-c+\partial_y^2 + 3Q_c^2)}_{\mathcal{L}_c} v = \partial_y ((Q_c + v)^3 - Q_c^3 - 3 Q_c^2 v)
+ \mbox{\{less important  terms\}}
\\
v(t=0) = v_0.
\end{array} \right.
$$
The asymptotic stability of solitons follows from the decay of $v$;
this in turn is given by decay estimates for the linear group $e^{t \mathcal{L}_c}$.
However, the functions in the generalized kernel of $\mathcal{L}_c$ (of dimension 2) do not decay under this semi group.
Thus one needs to make sure that, in the spectral decomposition associated to $\mathcal{L}_c$,
the component of $v$ in the generalized kernel of $\mathcal{L}_c$ is zero: this condition completely determines $c$ and $h$.

 Following the work of Mizumachi \cite{Mizumachi0}, see also \cite{Mizumachi-Pego,Mizumachi-Tzvetkov}, 
the radiation part $v$ is then split into $v= v_1 + v_2$, where $v_1$ simply solves \eqref{mKdV} in the $y$ coordinates, with data $v_0$,
$$
\left\{ \begin{array}{l}
\partial_t v_1 - (c + \dot h) \partial_y v_1 + \partial_y (v_1^3) + \partial_y^3 v = 0
\\
v_1(t=0) = v_0,
\end{array} \right.
$$
while $v_2$, the remainder term with zero initial data, solves
$$
\left\{ \begin{array}{l}
\partial_t v_2 - \mathcal{L}_c v_2 = - \partial_y ( (Q_c + v_1 + v_2)^3 - Q_c^3 - v_1^3 - 3 Q_c^2 v_2) + \mbox{\{less important  terms\}}
\\
v_2(t=0) = 0.
\end{array} \right.
$$
The advantages of this decomposition become clear when one tries to obtain decay
for the part of the wave which is to the right of the soliton, that is the region $\{ y > 0 \}$.
In more technical terms, we want to obtain decay for $\| \langle y_+ \rangle^m v_1 \|_{L^2} + \| e^{ay} v_2 \|_{L^2}$.

\setlength{\leftmargini}{1.5em}
\begin{itemize}

\item The decay of ${\| \langle y_+ \rangle^m v_1 \|}_{L^2}$ is obtained by a virial-type argument,
  which one can apply since the equation for $v_1$ does not ``see'' the soliton, and the data are such that
  ${\| \langle y_+ \rangle^m v_0 \|}_{L^2} < \infty$.

\item The decay of ${\| e^{ay} v_2 \|}_{L^2}$ is obtained via decay estimates in exponentially weighted spaces
  (with norm of the type $\|  e^{ay} \cdot \|_{L^2}$) for the linear group $e^{t\mathcal{L}_c}$.
  This requires the data, as well as the right-hand side of the $v_2$ equation, to belong to an exponentially weighted space.
  This is easily checked for the $v_2$ equation: the data is zero, and expanding the right-hand side,
  it appears that the slowly decaying $v_1$ factors are always coupled to $v_2$ or $Q_c$, thus ensuring exponential decay.

\end{itemize}

%\subsubsection{The contraction argument}
%
%The proof of the global existence result in Theorem \ref{maintheo1} is based on global a priori estimates
%for a local solution of \eqref{mKdV}. The following is a basic local existence result:
%
%\begin{theorem}\label{theolocal}
%
%\end{theorem}
%
%The global solutions of Theorem \ref{maintheo1} will be constructed by continuing for the solutions given by Theorem \ref{theolocal} above.
%Thanks to the above local existence theory, and time reversal symmetry, we can restrict our attention to times $t \geq 1$.
%
%More precisely, we will find solutions of \eqref{mKdV} which are small in the space $X$ given by the norm
%\begin{align}
%\label{norm}
%\begin{split}
% {\|u\|}_X = \sup_{t \geq 1} \Big(  t^{-\delta} {\| u(t) \|}_{H^1}
%  + t^{-1/6} {\big\| x f(t) \big\|}_{H^1}
%  + t^{\alpha/3-1/6} {\big\| |\partial_x|^\alpha x f(t) \big\|}_{L^2}
%  + {\big\| \what{f}(t) \big\|}_{L^\infty_\xi}  \Big) ,
%\end{split}
%\end{align}
%where $\alpha$ is less than, but sufficiently close to, $\frac{1}{2}$.
%
%Notice that the above norm is almost scale invariant...
%
%Modified scattering in a neighborhood of the self similar solution.

\subsubsection{Organization of the paper}
Section \ref{seczero} contains the proof of Theorem \ref{maintheo1} about the stability of the trivial solution.
We begin by establishing some linear and simple multilinear decay estimates in \ref{seclinear}.
In \ref{secee} we prove energy estimates involving the scaling vector field.
Sections \ref{secfhat} and \ref{keysec} contain the heart of the proof of Theorem \ref{maintheo1}, that is the justification of
the asymptotic expansion \eqref{d_tf0} and the control of the remainder terms.
In section \ref{secasy} we derive the complete asymptotic description of small solutions of \eqref{mKdV} relying on
a refined linear estimate and the global bounds established before.
Section \ref{secsoliton} is devoted to the proof of Theorem \ref{maintheo2} about the stability of soliton solutions.
We first prove asymptotic stability \`a la Pego-Weinstein in section \ref{secsoliton1}, and then
give the proof of scattering for the radiation in section \ref{secsoliton2}.

\medskip
\subsection{Notations}\label{secnot}
For $x \in \mathbb{R}$,  we set
$$
\langle x \rangle = \sqrt{1 + x^2}.
$$
We denote $C$ for a constant whose value may change from one line to another.
Given two quantities $X$ and $Y$, we write

\setlength{\leftmargini}{2em}
\begin{itemize}

\item $X \lesssim Y$ if $X \leq C Y$ for a constant $C$.

\item $X \sim Y$ if $X \lesssim Y$ and $Y \lesssim X$.

\item $X \ll Y$ if $X \leq cY$ for a small constant $c$.

\end{itemize}

We define the Fourier transform by
\begin{align*}
\mathcal{F} g (\xi) = \widehat{g}(\xi) := \frac{1}{\sqrt{2\pi}} \int_{\mathbb{R}} e^{-ix\xi} g(x) \,dx
  \quad \implies \quad g(x) = \frac{1}{\sqrt{2\pi}} \int_{\mathbb{R}} e^{ix\xi} \widehat{g}(\xi) \,d\xi \, .
\end{align*}
The Fourier multiplier $m(\partial_x)$ with symbol $m$ is given by
$$
\mathcal{F}[m(\partial_x) f](\xi) = m(i \xi) \widehat{f}(\xi).
$$
and the pseudoproduct operator $T_m$ with symbol $m(\xi,\eta)$ by
$$
\mathcal{F} [T_m(f,g)](\xi) = \int m(\xi,\eta) \widehat{f}(\xi-\eta) g(\eta) \,d\eta.
$$
Let $\psi$ be smooth, supported on $[-2,-\frac{1}{2}] \cup [\frac{1}{2},2]$, and satisfying
$$
\sum_{j \in \mathbb{Z}} \psi \left( \frac{\xi}{2^j} \right) = 1 , \qquad \mbox{for $\xi \neq 0$}.
$$
Define
$$
\chi = \sum_{j<0} \psi \left( \frac{\xi}{2^j} \right)
$$
and the Littlewood-Paley operators
\begin{align*}
& P_j = \psi \left( \frac{\partial_x}{i 2^j} \right) \, , \quad P_{<j} = \chi \left( \frac{\partial_x}{i 2^j} \right)
  , \quad  P_{>j} = 1-\chi \left( \frac{\partial_x}{i 2^j} \right) ,
\\
& P_{\sim j} = \sum_{2^k \sim 2^j} P_k ,\quad P_{\ll j} = \sum_{2^k \ll 2^j} P_k ,\quad P_{\lesssim j} = \sum_{2^k \leq 2^{j+C}} P_k .
\end{align*}
We will often denote $f_j := P_j f$,  $f_{\sim j} := P_{\sim j} f$, and so on.

\bigskip
\section{Stability of the zero solution}\label{seczero}
We assume that the following $X$-norm is a priori small:
\begin{align}
\label{apriori}
{\|u\|}_X = \sup_{t\geq 1} \left( t^{-\delta} {\| u \|}_{H^1}
  + t^{-1/6} {\big\| x f \big\|}_{H^1}
  + t^{\alpha/3 - 1/6} {\big\||\partial_x|^\alpha xf \big\|}_{L^2}
  + {\big\| \widehat{f}(\xi) \big\|}_{L^\infty}  \right) \leq \e_1 \, .
\end{align}
We will then show
\begin{align}
\label{conc}
 {\| u \|}_X \leq \e_0 + C \e_1^3
\end{align}
for some absolute constant $C$. This a priori estimate, combined  with a bootstrap argument,
and the choice $\e_1 = \e_0^{2/3}$,
gives global existence for sufficiently small $\varepsilon_0$.
%\comment{(Pierre) We should explain at some point what we do for $0<t<1$. (Fabio) Something like what follows?}
For simplicity, and without loss of generality, we only consider $t \geq 1$,
assuming that a local solutions has been already constructed on the time interval $[0,1]$ by standard methods, such as those in \cite{KPV1}.
Using also time reversibility we obtain a solution for all times.

\bigskip
\subsection{Linear and multilinear estimates}\label{seclinear}

We begin by proving a refined linear estimate which also gives useful $L^p$ bounds.

\begin{lemma}[Linear Estimates]\label{lemlinear2}
For any $t \geq 1$, $x \in \R$, and $u(t,x) = e^{-t\partial_x^3} f(t,x)$ with $f$ such that
\begin{equation}
\label{lemlinear2ass}
{\| \what{f}(t) \|}_{L^\infty} + t^{-1/6} {\| x f(t) \|}_{L^2} \leq 1,
\end{equation}
the following estimate holds true:
\begin{align}
 \label{lemlinear21}
& \big| |\partial_x |^\beta u(x,t) \big| \lesssim t^{-1/3 - \beta/3} \big( 1 +  |x/t^{1/3}| \big)^{-1/4 + \beta/2} ,
  \qquad 0 \leq \beta \leq 1 ,
\\
 \label{lemlinear21b}
& \big| \partial_x u(x,t) \big| \lesssim t^{-2/3} \big( 1 +  |x/t^{1/3}| \big)^{1/4} .
\end{align}
In particular, whenever $u =  e^{-t\partial_x^3} f$ satisfies the a priori bounds \eqref{apriori}, one has
for any $\beta \in [0,\frac{1}{2})$, and all $p$ with $p(1/4 - \beta/2) > 1$,
\begin{align}
\label{lemlinear22}
{\| |\partial_x |^\beta u(t) \|}_{L^p} \lesssim t^{-1/3 - \beta/3 + 1/(3p)} .
\end{align}
\end{lemma}

%\comment{PG $\to$ FP: In the above lemma, it seems to me we can remove the hypothesis $t \geq 1$ and replace $\beta<1$ by $\beta\leq 1$.
%It does not really matter but it seems more clear to me this way.}

\begin{remark}
\normalfont
The refined linear estimate \eqref{lemlinear21} in the case $\beta = 0$, and the estimate \eqref{lemlinear21b}
coincide with the estimates obtained in \cite{HNmKdV2}; see also \cite{CV} for related work.
The improvement for $\beta > 0$ is needed to obtain \eqref{lemlinear22}, which will be in turn used to prove
the trilinear estimate \eqref{oriole} below.
\eqref{oriole} allows us to simplify our subsequent analysis (especially the key estimate
of section \ref{keysec}) and give a sharper global existence result, see remark \ref{rempartialalpha}.
\end{remark}

\begin{proof}
Denote $\Lambda(\xi) = \xi^3$, and write
\begin{align*}
u(t,x) = e^{-t\partial_x^3} f(t,x) = \frac{1}{\sqrt{2\pi}} \int_{\R} e^{it \phi(\xi)} \what{f}(\xi) \,d\xi ,
\qquad \phi(\xi) := \xi (x/t) + \Lambda(\xi).
\end{align*}
For $x \leq 0$, let
\begin{align}
\xi_0^{\pm} := \pm \sqrt{-x/(3t)}
\end{align}
denote the stationary points of the phase $\phi(\xi)$, $\phi^\prime (\xi_0^\pm) = 0$.
In the case $x > 0$ there are no stationary points, and the estimate in this case is easier and
follows from the same arguments that we present below.

We now restrict our attention to $x \leq 0$.
Up to taking complex conjugates, we notice that in order to obtain \eqref{lemlinear21}-\eqref{lemlinear21b} it suffices to show that
\begin{equation*}
\Big| %\Re
  \int_0^\infty e^{i t \phi(\xi)}  |\xi|^\beta \what{f}(\xi,t)\,d\xi \Big|
  \lesssim t^{-1/3-\beta/3} \big( 1 +  |x/t^{1/3}| \big)^{-1/4 + \beta/2} ,
\end{equation*}
for all $\beta \in (0,1)$.
Let us denote $\xi_0 := \sqrt{-x/(3t)}$ the only stationary point in the above integral.
We see that since $|x/t|^{1/3} = 3 (\xi_0 t^{1/3})^2$,  it is then enough to show
\begin{equation}
\label{lemlinear2conc}
\Big| %\Re
  \int_0^\infty e^{i t \phi(\xi)} |\xi|^\beta \what{f}(\xi,t) \,d\xi \Big|
  \lesssim t^{-1/3-\beta/3} \max \big( t^{1/3}\xi_0, 1\big)^{-1/2+\beta} ,
\end{equation}
for any $t \geq 1$, $x \leq 0$, and any function $f$ satisfying~\eqref{lemlinear2ass}.
We distinguish two cases depending on the size of $\xi_0$.

%\comment{PG $\to$ FP: it seems to me we should remove the real part $\Re$ in the two above inequalities if we want to prove \eqref{lemlinear21b}.}

\bigskip

\noindent
{\it Case 1: $\xi_0 \leq t^{-1/3}$}.
In this case we only need to obtain a bound of $t^{-1/3-\beta/3}$ for the term in \eqref{lemlinear2conc}.
%This was already proven in Lemma \ref{lemlinear}, see the estimate \eqref{linear1}.
%\comment{We do not need the proof below in view of \eqref{linear1}, but I wrote it in case we need it
%when we clean up all the stationary phase type lemmas}
We split the integral in \eqref{lemlinear2conc} as follows:
\begin{align*}
& \int_0^\infty e^{i t \phi(\xi)}  |\xi|^\beta \what{f}(\xi)\,d\xi = A + B ,
\\
& A = \int_0^\infty e^{i t \phi(\xi)}  |\xi|^\beta \what{f}(\xi) \chi(2^{-10}t^{1/3}\xi) \,d\xi ,
\\
& B = \int_0^\infty e^{i t \phi(\xi)}  |\xi|^\beta \what{f}(\xi) (1-\chi(2^{-10}t^{1/3}\xi)) \,d\xi .
\end{align*}
The first term can be very easily estimated using the first bound in \eqref{lemlinear2ass}.
%\begin{align*}
%| A | \lesssim t^{-1/3(1+\beta)} {\| \what{f} \|}_{L^\infty}
%\end{align*}
For the second term we notice that $|\xi| \gg \xi_0$ on the support of the integral,
so that  $| \partial_\xi \phi | = 3|\xi_0^2 - \xi^2| \gtrsim |\xi|^2 \gtrsim t^{-2/3}$.
An integration by parts then gives:
\begin{align*}
| B | & \lesssim B_1 + B_2 ,
\\
B_1 & = \frac{1}{t} \int_0^\infty
  \Big| \partial_\xi \Big( \frac{1}{\partial_\xi \phi} |\xi|^\beta (1-\chi(2^{-10}t^{1/3}\xi) \Big) \Big| |\what{f}(\xi)| \,d\xi ,
\\
B_2 & = \frac{1}{t} \int_0^\infty
  \Big| \frac{1}{\partial_\xi \phi} |\xi|^\beta (1-\chi(2^{-10}t^{1/3}\xi)\Big| | \partial_\xi \what{f}(\xi) | \,d\xi .
\end{align*}
We can then estimate
\begin{align*}
B_1 \lesssim \frac{1}{t} {\| \what{f} \|}_{L^\infty}
  \int_0^\infty \frac{1}{|\xi|^{3-\beta}} (1-\chi(2^{-10}t^{1/3}\xi)) + \frac{1}{|\xi|^{2-\beta}} | \chi^\prime (2^{-10}t^{1/3}\xi) | t^{1/3} \, d\xi
  \lesssim t^{-1/3 -\beta/3} .
\end{align*}
Similarly, we can use the second bound provided by \eqref{lemlinear2ass} to obtain:
\begin{align*}
B_2 \lesssim \frac{1}{t} {\| \partial_\xi \what{f} \|}_{L^2}
  \Big( \int_0^\infty \frac{1}{|\xi|^{4-2\beta}} (1-\chi(2^{-10}t^{1/3}\xi))  \, d\xi \Big)^{1/2}
  \lesssim t^{-1/3 -\beta/3} .
\end{align*}

\bigskip

\noindent
{\it Case 2: $\xi_0 \geq t^{-1/3}$}.
In this case we aim to prove a bound of $t^{-1/2} \xi_0^{-1/2+\beta}$ for the left-hand side of \eqref{lemlinear2conc}.
To separate the non-stationary and stationary cases we split the integral as follows (see \ref{secnot}):
\begin{align}
\label{lemlinear2split}
 \begin{split}
 & \int_0^\infty e^{i t \phi(\xi)}  |\xi|^\beta \what{f}(\xi)\,d\xi = C + D ,
\\
& C = \int_0^\infty e^{i t \phi(\xi)}  |\xi|^\beta \what{f}(\xi) (1 - \psi(\xi/\xi_0)) \,d\xi ,
\\
& D = \int_0^\infty e^{i t \phi(\xi)}  |\xi|^\beta \what{f}(\xi) \psi(\xi/\xi_0) \,d\xi .
\end{split}
\end{align}
%The treatment of $C$ is similar the the one for $B$ above.
Integrating by parts we get
\begin{align*}
| C | & \lesssim C_1 + C_2 ,
\\
C_1 & = \frac{1}{t} \int_0^\infty
  \Big| \partial_\xi \Big( \frac{1}{\partial_\xi \phi} |\xi|^\beta (1 - \psi(\xi/\xi_0)) \Big) \Big| |\what{f}(\xi)| \,d\xi ,
\\
C_2 & = \frac{1}{t} \int_0^\infty
  \frac{1}{|\partial_\xi \phi|} |\xi|^\beta (1 - \psi(\xi/\xi_0)) | \partial_\xi \what{f}(\xi) | \,d\xi .
\end{align*}
Using the fact that on the support of the above integrals $| \partial_\xi \phi | \gtrsim \max(\xi,\xi_0)^2$ we obtain
\begin{align*}
C_1 \lesssim \frac{1}{t} {\| \what{f} \|}_{L^\infty}
  \int_0^\infty \left( \frac{|\xi|^{\beta-1}}{\max(\xi,\xi_0)^{2}} (1-\psi(\xi/\xi_0)) + \frac{1}{|\xi|^{2-\beta}} | \psi^\prime(\xi/\xi_0)|  \xi_0^{-1} \right)\, d\xi
  \lesssim t^{-1} \xi_0^{-2+\beta} ,
\end{align*}
%\comment{PG $\to$ FP: small modification in the above line... OK?
%
%FP: Wrote $\xi^{\beta-1}$ instead of $1/\xi^{\beta-1}$.}
which is stronger than the desired bound since $\xi_0 \geq t^{-1/3}$. Similarly
\begin{align*}
C_2 \lesssim \frac{1}{t} {\| \partial_\xi \what{f} \|}_{L^2}
  \Big( \int_0^\infty \frac{1}{\max(\xi,\xi_0)^{4-2\beta}} (1-\psi(\xi/\xi_0)) \, d\xi \Big)^{1/2}
  \lesssim t^{-5/6} \xi_0^{-3/2+\beta}
\end{align*}
which suffices since $\xi_0 \geq t^{-1/3}$.

To estimate the resonant contributions $\xi \approx \xi_0$
we let $l_0$ be the smallest integer such that $2^{l_0} \geq 1/\sqrt{t\xi_0}$ and bound the term $D$ in \eqref{lemlinear2split}
as follows:
\begin{align*}
\begin{split}
& | D | \leq \sum_{l=l_0}^{\log \xi_0 + 10} |D_l| ,
\\
& D_{l_0} := \int_{\R} e^{it\phi(\xi)} |\xi|^\beta \what{f}(\xi) \psi(\xi/\xi_0) \chi\big( 2^{-l_0} (\xi-\xi_0) \big) \, d\xi ,
\\
& D_{l} := \int_{\R} e^{it\phi(\xi)} |\xi|^\beta \what{f}(\xi) \psi(\xi/\xi_0) \psi ( 2^{-l} (\xi-\xi_0) ) \, d\xi \, , \qquad l \geq l_0 + 1 .
\end{split}
 \end{align*}
The choice of $l_0$ and the first bound in \eqref{lemlinear2ass} immediately give us
$|D_{l_0}| \lesssim \xi_0^\beta 2^{l_0} \lesssim t^{-1/2} \xi_0^{-1/2+\beta}$.
To control the terms $D_l$ we integrate by parts and see that:
\begin{align*}
| D_l |& \lesssim D_{l,1} + D_{l,2} ,
\\
D_{l,1} & = \frac{1}{t} \int_0^\infty
  \Big| \partial_\xi \Big( \frac{1}{\partial_\xi \phi} |\xi|^\beta \psi(\xi/\xi_0) \psi(2^{-l}(\xi-\xi_0)) \Big) \Big| |\what{f}(\xi)| \,d\xi ,
\\
D_{l,2} & = \frac{1}{t} \int_0^\infty
  \frac{1}{|\partial_\xi \phi|} |\xi|^\beta  \psi(\xi/\xi_0) \psi(2^{-l}(\xi-\xi_0)) | \partial_\xi \what{f}(\xi) | \,d\xi .
\end{align*}
Using the fact that on the support of the integrals $|\partial_\xi \phi| = 3|\xi^2 - \xi_0^2| \approx 2^l \xi_0$,
we can estimate
\begin{align*}
D_{l,1} \lesssim t^{-1} {\| \what{f} \|}_{L^\infty} 2^{-l} \xi_0^{-1+\beta} ,
\end{align*}
which gives the desired bound upon summation in $l$.
Similarly, we have
\begin{align*}
D_{l,2} \lesssim t^{-1} {\| \partial_\xi \what{f} \|}_{L^2} 2^{-l/2} \xi_0^{-1+\beta} .
\end{align*}
Using the second bound in \eqref{lemlinear2ass} and summing in $l$ we obtain a bound of
$t^{-5/6} \xi_0^{-1+\beta} 2^{-l_0/2}$, which is better than our desired bound. This concludes the proof of \eqref{lemlinear21}.
The estimate \eqref{lemlinear22} follows by integrating in $L^p$ the inequality \eqref{lemlinear21}.
\end{proof}

\begin{lemma}[Multilinear Estimates]\label{lembilinear}
Let $u = e^{t\partial_x^3} f$ be a function satisfying the a priori assumptions
\begin{align}
\label{apriorilembilinear}
\sup_{t \geq 1} \big(
  {t}^{-1/6} {\big\| x f(t) \big\|}_{L^2} + {\big\| \widehat{f}(t) \big\|}_{L^\infty}  \big) \leq \e_1 .
\end{align}
Then the following bilinear estimate holds:
\begin{align}
 \label{bilinear}
\sup_{t \geq 1} \, t \, {\| u (t)  u_x(t) \|}_{L^\infty_x} \lesssim \e_1^2 .
\end{align}
Moreover, for all $0 \leq \alpha < \frac{1}{2}$ we have
\begin{equation}
\label{oriole}
{\| |\partial_x|^\alpha u^3(t) \|}_{L^2} \lesssim \e_1^3 |t|^{-5/6 - \alpha/3} .
\end{equation}
\end{lemma}

\begin{proof}
To obtain \eqref{bilinear} it suffices to multiply the bounds provided by \eqref{lemlinear21} in the case $\beta=0$ and \eqref{lemlinear22}.

To show \eqref{oriole} we
start by choosing $\beta \in (\alpha,1/2)$, $p,q,p_1,p_2 \in (2,\infty)$, such that
\begin{align}
\label{orioleind}
1/p + 1/q = 1/2 , \,\quad 1/p = \theta/p_1 + (1-\theta)/p_2 \quad \mbox{with} \quad \theta = \alpha/\beta ,
\end{align}
and
$$
\qquad p_1(1/4-\beta/2) > 1, \qquad p_2>4
$$
(to see that it is possible to choose parameters satisfying the above requirements, first fix $\beta \in (\alpha,1/2)$;
the other indexes are then fully determined by $p$ and $p_1$; one checks that the above inequalities are satisfied if $p,p_1 \to \infty$).
%\comment{PG $\to$ FP. Since it is a bit tricky, I added the above small explanation.
%I removed the condition $p_1>p$ which is not present in Bahouri-Chemin-Danchin,
%and uncommented the ref to them below, since the proof of this GN inequality is also quite tricky! Are you OK with all that?}
We use the fractional Leibniz rule, followed by the Gagliardo-Nirenberg inequality (Theorem 2.44 in\cite{BCD}) to obtain
\begin{align*}
{\| |\partial_x|^\alpha u^3 \|}_{L^2} & \lesssim {\| |\partial_x|^\alpha u \|}_{L^p} {\| u^2 \|}_{L^q}
  \lesssim {\| |\partial_x|^\beta u \|}_{L^{p_1}}^\theta  {\| u \|}_{L^{p_2}}^{1-\theta}  {\| u \|}_{L^{2q}}^2
\end{align*}
Using the linear estimate \eqref{lemlinear22} we have
\begin{align*}
{\| |\partial_x|^\beta u \|}_{L^{p_1}} \lesssim t^{-1/3 - \beta/3 + 1/(3p_1)} ,
\qquad
{\| u \|}_{L^{p_2}} \lesssim t^{-1/3 + 1/(3p_2)} ,
\qquad
{\| u \|}_{L^{2q}} \lesssim t^{-1/3 + 1/(6q)} .
\end{align*}
Using these three inequalities we see that
\begin{align*}
{\| |\partial_x|^\alpha u^3 \|}_{L^2} & \lesssim t^{\gamma}
\end{align*}
where, using \eqref{orioleind}, we have
\begin{align*}
\gamma & = (-1/3 -\beta/3 + 1/(3p_1))\theta + (-1/3 + 1/(3p_2))(1-\theta) + 2(-1/3 + 1/(6q))
  \\ & = -1 + (1/3) \big( \theta(-\beta + 1/p_1) + (1-\theta)/p_2) + 1/(3q) = -5/6 - \alpha/3
\end{align*}
as desired.
\end{proof}

\vskip10pt
\subsection{Energy estimates}\label{secee}

We now prove energy and weighted energy estimates.
\begin{lemma}\label{lemee}
Let $u \in C([0,T); H^1)$ be a solution of \eqref{mKdV} satisfying the apriori bounds \eqref{apriori}.
Then
\begin{align*}
%\label{lemee1}
{\| u(t) \|}_{H^1} \leq \e_0  + C \e_1^2 .%\langle t \rangle^{\delta} .
\end{align*}
%\begin{align}
%\label{lemee1}
%{\| u(t) \|}_{H^1} \lesssim \e_0 .
%\end{align}
Moreover, if $u = e^{-t\partial_x^3} f$,
\begin{align}
\label{lemee2}
{\| x f(t) \|}_{H^1} \leq C(\e_0 + \e_1^3) \langle t \rangle^{1/6} .
\end{align}
Finally, for $0 \leq \alpha < \frac{1}{2}$,
\begin{align}
\label{lemee3}
{\||\partial_x|^\alpha x f(t) \|}_{L^2} \leq C(\e_0 + \e_1^3) \langle t \rangle^{1/6 - \alpha/3} \, .
\end{align}
\end{lemma}

\begin{proof}
The first estimate follows from the conservation of Mass and Energy \eqref{ME}.
The estimates \eqref{lemee2} and \eqref{lemee3} will be obtained by energy estimates performed on the \eqref{mKdV} equation itself,
or on the equation obtained after commuting the scaling vector field $S := 1 + x\partial_x + 3t \partial_t$, i.e.
\begin{equation}
\label{SmKdV}
\partial_t S u + \partial_x^3 S u + 3 \partial_x (u^2 S u) = 0.
\end{equation}

\bigskip

\noindent
{\it Proof of \eqref{lemee2}}.
In the following, we denote, given a function $a$, $I a$ for the antiderivative of $a$ vanishing at $-\infty$:
$[I a](x) = \int_{-\infty}^x a $.
Applying $I$ to~\eqref{SmKdV} gives
$$
\partial_t I S u + \partial_x^2 S u + 3 u^2 S u = 0.
$$
Multiplying by $I S u$ and integrating in space yields
$$
\frac{1}{2} \frac{d}{dt} {\big\| I S u \big\|}_{L^2}^2 + \int \partial_x^2 S u \,  I S u  \,dx + 3 \int u^2 S u \, I S u \,dx = 0,
$$
or, after integrating by parts, and taking \eqref{bilinear} into account,
$$
\frac{1}{2} \frac{d}{dt} {\big\| I S u \big\|}_{L^2}^2 = \frac{3}{2} \int \partial_x (u^2) \big( I S u \big)^2 \,dx
  \lesssim \frac{\e_1^2}{t} {\big\| I S u \big\|}_{L^2}^2.
$$
By Gronwall's lemma
\begin{align}
\label{boundISu}
{\big\| I S u \big\|}_{L^2}^2 \lesssim \e_0 t^{C \e_1^2}.
\end{align}
Observe now that $xf = I S f - 3t\partial_t I f = e^{t\partial_x^3} I S u + 3 e^{t\partial_x^3}t u^3$,
from which the above inequality, combined with \eqref{oriole}, gives the desired result:
$$
{\| xf \|}_{L^2} \leq  {\| I S u \|}_{L^2} + 3t {\| u^3 \|}_{L^2} \lesssim \e_0 t^{C\e_1^2} + \e_1^3 t^{1/6}.
$$
A similar argument applies to give a bound for ${\| \partial_x xf \|}_{L^2}$ and completes the proof of \eqref{lemee2}.

\bigskip

\noindent
{\it Proof of \eqref{lemee3}}.
%\comment{Fred: I think that we should discuss again here the interest of this proof in terms of the regularity of the initial data}
As explained in Remark \ref{rempartialalpha}, for the proof here, we do not really need to assume that $xf$ is in $H^1$.
 We shall give here a proof of   \eqref{lemee3} that do not require higher regularity.  By using  the $H^1$ regularity, a shorter proof is possible
 ( we shall use this argument in the proof of Lemma \ref{lemfinal} below to handle the solitary wave stability).
Starting from \eqref{SmKdV}, a simple energy estimate leads to
\begin{align*}
\frac{1}{2} \frac{d}{dt} {\| |\partial_x|^{\alpha-1} S u \|}_{L^2}^2 &
  = -3 \int |\partial_x|^{\alpha -1} \partial_x (u^2 S u) |\partial_x|^{\alpha-1} S u \,dx
\\
& = -3  \int |\partial_x|^{\alpha-1} \partial_x (wv) |\partial_x|^{\alpha-1} v \,dx
\end{align*}
where
\begin{align*}
w = u^2 \quad \mbox{and} \quad v = S u \, .
\end{align*}
Recalling that $w_j = P_j w$ and $v_j = P_j v$, a paraproduct decomposition of the above right-hand side gives
\begin{align*}
\frac{d}{dt} {\| |\partial_x|^{\alpha-1} v \|}_{L^2}^2 & =
  \int |\partial_x|^{\alpha-1} \partial_x \Big( \sum_{2^j \gg 2^k} w_j v_k + \sum_{2^j \sim 2^k} w_j v_k + \sum_{2^k \gg 2^j} w_j v_k \Big)
  |\partial_x|^{\alpha-1} v\,dx
\\
& =: I + II + III.
\end{align*}

\noindent
To estimate $I$, we use the dispersive estimate \eqref{bilinear} and 
the standard properties of the Littlewood-Paley decomposition (see for example \cite{BCD}) to obtain
\begin{align*}
I & \lesssim \Big\| |\partial_x|^{\alpha-1} \sum_{2^k \ll 2^j} \partial_x ( w_j v_k ) \Big\|_{L^2} {\| |\partial_x|^{\alpha-1} v \|}_{L^2}
\\
& \lesssim \left[ \sum_j \Big( 2^{\alpha j} \Big\| \sum_{2^k \ll 2^j} w_j v_k \Big\|_{L^2} \Big)^2 \right]^{1/2} \| |\partial_x|^{\alpha-1} v\|_{L^2} \\
& \lesssim \left[ \sum_j {\Big( \sum_{2^k \ll 2^j }
   2^{(\alpha - 1)j}
  2^j {\| w_j \|}_{L^\infty} 2^{k(1-\alpha)} {\|  |\partial_x|^{\alpha-1} v_k \|}_{L^2} \Big)}^2 \right]^{1/2}
  {\| |\partial_x|^{\alpha-1} v \|}_{L^2}
\\
& \lesssim \frac{\varepsilon_1^2}{\langle t \rangle} \left[ \sum_{j} \left( \sum_{2^k \ll 2^j}  2^{(j - k)(\alpha - 1)}
  {\|  |\partial_x|^{\alpha-1} v_k \|}_{L^2} \right)^2 \right]^{1/2} {\| |\partial_x|^{\alpha-1} v \|}_{L^2}
  \lesssim \frac{\varepsilon_1^2}{\langle t \rangle} {\| |\partial_x|^{\alpha-1} v \|}_{L^2}^2 .
\end{align*}

\noindent
The estimate of $II$ also relies on \eqref{bilinear}:
\begin{align*}
II & \lesssim  {\Big\| |\partial_x|^{\alpha-1} \partial_x \sum_{2^k \sim 2^j} (w_j v_k) \Big\|}_{L^2} {\| |\partial_x|^{\alpha-1} v \|}_{L^2}
\\
& \lesssim \left[ \sum_\ell \Big( 2^{\alpha \ell} \Big\| P_\ell \Big( \sum_{2^k \sim 2^j} w_j v_k \Big) \Big\|_{L^2} \Big)^2 \right]^{1/2} \|  |\partial_x|^{\alpha-1} v \|_{L^2} \\
& \lesssim \left[ \sum_{\ell} {\Big(  \sum_{2^k \sim 2^j \gtrsim 2^\ell} 2^{\alpha \ell}  2^{-\alpha k}
  2^j {\| w_j \|}_{L^\infty}  2^{k(\alpha-1)} {\|v_k \|}_{L^2} \Big)}^2 \right]^{1/2} {\| |\partial_x|^{\alpha-1} v \|}_{L^2}
\\
& \lesssim \frac{\varepsilon_1^2}{\langle t \rangle} \left[ \sum_{\ell} \Big( \sum_{2^k \gtrsim 2^\ell} 2^{\alpha(\ell-k)}
  {\|  |\partial_x|^{\alpha-1} v_k \|}_{L^2} \Big)^2 \right]^{1/2} {\| |\partial_x|^{\alpha-1} v \|}_{L^2}
  \lesssim \frac{\varepsilon_1^2}{\langle t \rangle} {\| |\partial_x|^{\alpha-1} v \|}_{L^2}^2 .
\end{align*}

\noindent
To estimate $III$, we need the classical commutator estimate
\begin{align}
 \label{comm}
{\left\| \left[ |\partial_x|^{\alpha-1} \partial_x , P_{\ll j} w \right] P_j f \right\|}_{L^2}
  \lesssim 2^{(\alpha-1)j} {\|\partial_x w\|}_{L^\infty} {\| P_j f\|}_{L^2} ,
\end{align}
proved in the appendix in Lemma \ref{lemmacomm}.
Commuting $|\partial_x|^{\alpha-1} \partial_x$ with $w_{\ll k}$ in $III$ gives
\begin{align*}
III & = \int \sum_k \left[ |\partial_x|^{\alpha-1} \partial_x , w_{\ll k} \right] v_k \, |\partial_x|^{\alpha-1} v \,dx
  + \int \sum_k w_{\ll k} \,  \partial_ x |\partial_x|^{\alpha-1} \, v_k \, |\partial_x|^{\alpha-1} v \,dx
\\
& =: III_a + III_b
\end{align*}
It follows easily from the commutator estimate \eqref{comm} and \eqref{bilinear} that
\begin{align*}
|III_a| \lesssim {\| \partial_x w \|}_{L^\infty}  {\| |\partial_x|^{\alpha-1} v \|}_{L^2}^2
  \lesssim \frac{\varepsilon_1^2}{\langle t \rangle} {\| |\partial_x|^{\alpha-1} v \|}_{L^2}^2 .
\end{align*}
To estimate $III_b$, we integrate by parts to obtain
\begin{align*}
III_b & = \int \sum_{2^j \sim 2^k} w_{\ll j } \partial_x |\partial_x|^{\alpha-1}  v_j |\partial_x|^{\alpha-1} v_k\,dx \\
& =  \int \sum_{2^j \sim 2^k} \partial_x  w_{\ll j} |\partial_x|^{\alpha-1} v_j |\partial_x|^{\alpha-1} v_k\,dx + \mbox{\{ remainder \}}.
\end{align*}
The remainder is easy to treat, thus we skip it, and the main term is not much harder:
\begin{align*}
|III_b| \lesssim {\| \partial_x w \|}_{L^\infty} {\| |\partial_x|^{\alpha-1} v \|}_{L^2}^2
  \lesssim \frac{\varepsilon_1^2}{t} {\| |\partial_x|^{\alpha-1} v \|}_{L^2}^2 .
\end{align*}

Gathering the estimates on $I$, $II$ and $III$, we obtain
\begin{align*}
\frac{d}{dt} {\| |\partial_x|^{\alpha-1} v \|}_{L^2}^2 \lesssim \frac{\varepsilon_1^2}{t} {\| |\partial_x|^{\alpha-1} v \|}_{L^2}^2,
\end{align*}
which implies, by Gronwall's lemma,
\begin{align*}
{\| |\partial_x|^{\alpha-1} S u \|}_{L^2} = {\| |\partial_x|^{\alpha-1} v \|}_{L^2} \lesssim \varepsilon_0 t^{C \varepsilon_1^2}.
\end{align*}

Finally, observe that
\begin{align*}
{\| |\partial_x|^\alpha xf \|}_{L^2} = {\| |\partial_x|^{\alpha-1}(x\partial_x + 1) f \|}_{L^2}
  & \leq {\| |\partial_x|^{\alpha-1} S f \|}_{L^2} +3 t {\| |\partial_x|^{\alpha-1} \partial_t f \|}_{L^2}
\\
  & = {\| |\partial_x|^{\alpha-1} S u \|}_{L^2} +3 t {\| |\partial_x|^{\alpha} u^3 \|}_{L^2} .
\end{align*}
The desired estimate follows by combining the bound on ${\| |\partial_x|^{\alpha-1} S u \|}_{L^2}$ and \eqref{oriole}.
\end{proof}

\vskip10pt
\subsection{Control of $\sup_t \| \widehat{f}(t) \|_\infty$}\label{secfhat}
This section is dedicated to proving the following key proposition:
\begin{proposition}\label{prohatf}
Under the a priori assumptions \eqref{apriori}, the following estimate holds for a solution $u$ of \eqref{mKdV}:
\begin{align}\label{esthatf}
\sup_t {\big\| \widehat{f}(t) \big\|}_\infty \leq  {\big\| \widehat{u_0} \big\|}_\infty + C\e_1^3 \, .
\end{align}
\end{proposition}

\begin{proof}
We will show the following key identity: for $t>1$,
\begin{align}
\label{d_tf=}
\partial_t \what{f}(t,\xi) = \frac{i \operatorname{sign} \xi}{6t} |\what{f}(t,\xi)|^2 \widehat{f}(t,\xi)
  + \frac{c}{t} e^{it \frac{8}{9} \xi^3} \Big[ \mathbf{1}_{|\xi|>t^{-1/3}} \what{f}(t,\xi/3) \Big]^3 %P_{> -\frac{1}{3} \log t} f
  + R(t,\xi) \, ,
\end{align}
where $c \in \C$ is a constant, and $R$ satisfies the bound
\begin{align}
\label{estR}
\int_{-\infty}^{\infty} |R(t,\xi)| \,dt \lesssim \e_1^3 \, .
\end{align}
The proofs of \eqref{d_tf=} and \eqref{estR} above will be given in Section \ref{keysec} below;
let us first show how \eqref{d_tf=}-\eqref{estR} imply the desired conclusion \eqref{esthatf}.
Define the modified profile $\what{w}$ as follows:
\begin{align}
\label{modprof}
\what{w}(t,\xi) := e^{- i B(t,\xi)} \what{f}(t,\xi)  \qquad \, \quad B(t,\xi)
:= \frac{1}{6} \operatorname{sign} \xi \int_1^t \big|\what{f}(s,\xi) \big|^2 \frac{ds}{s} \, .
\end{align}
Then we have
\begin{align*}
\partial_t \what{w}(t,\xi) & = e^{- i B(t,\xi)} \left[ \partial_t \what{f}(t,\xi) - i\partial_t B(t,\xi)\what{f}(t,\xi) \right]
\\
& = e^{- i B(t,\xi)} \left\{ \frac{c}{t} e^{it \frac{8}{9} \xi^3}
  \Big[ \mathbf{1}_{|\xi|>t^{-1/3}} \what{f}(t,\xi/3) \Big]^3 \, + R(t,\xi) \right\}
\end{align*}
Integrating in time the above identity, using the fact that $B$ is real, $|\what{w}(t,\xi)| = |\what{f}(t,\xi)|$,
and the remainder estimate \eqref{estR}, we obtain
\begin{align*}
\big| \what{f}(t,\xi) \big| & \leq
\left| \widehat{u_0}(\xi) \right|
  + c \left| \int_{|\xi|^{-3}}^t e^{is \frac{8}{9} \xi^3} e^{-iB(s,\xi)} \mathbf{1}_{|\xi|>s^{-1/3}} \widehat{f} (s,\xi/3)^3 \frac{ds}{s} \right| + \e_1^3 \, .
\end{align*}
The desired conclusion will then follow once we show that, for $t>|\xi|^{-3}$,
\begin{align}
\label{estextraR}
\Big| \int_{|\xi|^{-3}}^t e^{is \frac{8}{9} \xi^3} e^{-iB(s,\xi)} \widehat{f} (s,\xi/3)^3 \frac{ds}{s} \Big| \lesssim \e_1^3 \, .
\end{align}

\begin{proof}[Proof of \eqref{estextraR}]
Integrating by parts in $s$ using the identity $e^{is\frac{8}{9}\xi^3} = \frac{9}{8i\xi^3} \partial_s e^{is\frac{8}{9}\xi^3}$, we see that
\begin{align}
 \label{estextraRIBP}
\begin{split}
 \Big| & \int_{|\xi|^{-3}}^t e^{is \frac{8}{9} \xi^3} e^{-iB(s,\xi)} \widehat{f} (s,\xi/3)^3 \frac{ds}{s} \Big|
  \lesssim J + K + L + M
\\
& J = \frac{1}{|\xi|^3} |\widehat{f} (s,\xi/3) |^3 \frac{1}{s} \,\, \Big|^{s=t}_{s=|\xi|^{-3}} \, ,
\\
& K = \int_{|\xi|^{-3}}^t \frac{1}{|\xi|^3} |\partial_s \widehat{f} (s,\xi/3) |  |\widehat{f} (s,\xi/3)|^2 \frac{ds}{s} \, ,
\\
& L = \int_{|\xi|^{-3}}^t \frac{1}{|\xi|^3} | \partial_s B(s,\xi)|  |\widehat{f} (s,\xi/3)|^3 \frac{ds}{s} \, ,
\\
& M = \int_{|\xi|^{-3}}^t \frac{1}{|\xi|^3} |\widehat{f} (s,\xi/3)|^3 \frac{ds}{s^2} \, .
\end{split}
\end{align}
Since $t \geq |\xi|^{-3}$, the a priori assumption ${\|\what{f}(t)\|}_{L^\infty} \leq \e_1$ gives immediately
that $J \lesssim \e_1^3$.
Using again ${\|\what{f}(t)\|}_{L^\infty}  \leq \e_1$, and \eqref{d_tf=}-\eqref{estR} we can estimate
\begin{align*}
& K \lesssim \int_{|\xi|^{-3}}^t \frac{1}{|\xi|^3} \Big[ \frac{\e_1^3}{s} +  R(s,\xi) \Big] \e_1^2 \frac{ds}{s}  \lesssim \e_1^5 \, .
\end{align*}
From the definition of $B$ in \eqref{modprof} we see that
\begin{align*}
& L \lesssim \int_{|\xi|^{-3}}^t \frac{1}{|\xi|^3} \frac{\e_1^2}{s} \e_1^3 \frac{ds}{s}  \lesssim \e_1^5 \, .
\end{align*}
The last term, $M$, is easily bounded by $|\xi|^{-3} \epsilon_1^3 \int_{|\xi|^{-3}}^t \frac{ds}{s^2} \lesssim \epsilon_1^3$, which completes the proof of \eqref{estextraR}.
\end{proof}

\vskip10pt
\subsection{Proof of \eqref{d_tf=}-\eqref{estR}}\label{keysec}
Recall that we assume $t>1$.

\vskip10pt
\noindent
\underline{\it Some elementary estimates on $f$.}
Recall that $f_j = P_j f$; we start by stating a few estimates on $f_j$ that follow immediately from the a priori assumption \eqref{apriori}:
\begin{equation}
\label{bullfinch}
\begin{split}
& {\| \widehat{f}_j \|}_{L^\infty} \lesssim \e_1 ,
\\
& {\| xf_j \|}_{L^2} \lesssim
  \big[ 2^{-\frac{j}{2}} + \min \big(t^{\frac{1}{6}}, 2^{-\alpha j} t^{\frac{1}{6} - \frac{\alpha}{3}} \big) \big] \e_1,
%& {\| xf_j \|}_{L^2} \lesssim \min \big( 2^{-\frac{j}{2}} + t^{\frac{1}{6}} \big) \e_1
%\\
%& {\| xf_j \|}_{L^2} \lesssim \big( 2^{-\frac{j}{2}} + 2^{-\alpha j} t^{\frac{1}{6} - \frac{\alpha}{3}} \big) \e_1
\\
& {\| f_j \|}_{L^1} \lesssim {\|f_j\|}_{L^2}^{1/2} {\|xf_j\|}_{L^2}^{1/2} \lesssim \big( 1 + 2^{\frac{j}{4}} t^{\frac{1}{12}} \big) \e_1
\\
& {\|f_j\|}_{L^1} \lesssim \big( 1 + 2^{j(\frac{1}{4} - \frac{\alpha}{2})} t^{\frac{1}{12} - \frac{\alpha}{6}} \big) \e_1,
\\
& {\| |x|^{2\rho} f_j \|}_{L^1} \lesssim {\| f_j \|}_{L^2}^{\frac{1}{2} - 2 \rho} {\| xf_j \|}_{L^2}^{\frac{1}{2} + 2\rho}
  \lesssim \left( 2^{-2 \rho j} + 2^{j ( \frac{1}{4} - \rho )} t^{\frac{1}{12} + \frac{\rho}{3}} \right) \e_1,
  \quad \mbox{for} \quad 0 \leq \rho < \frac{1}{4} .
\end{split}
\end{equation}
Moreover, if $2^j \geq t^{-\frac{1}{3}}$, then $f_{<j}$ satisfies
\begin{align}
\begin{split}
 \label{bullfinch2}
& {\| f_{<j} \|}_{L^1} \lesssim 2^{\frac{j}{4}} t^{\frac{1}{12}} \e_1
\\
& {\| f_{<j} \|}_{L^1} \lesssim 2^{j(\frac{1}{4} - \frac{\alpha}{2})} t^{\frac{1}{12} - \frac{\alpha}{6}} \e_1
\\
& {\| |x|^{2\rho} f_{<j} \|}_{L^1} \lesssim 2^{j(\frac{1}{4} - \rho)} t^{\frac{1}{12} + \frac{\rho}{3}} \e_1,
  \quad \mbox{for} \quad 0 \leq \rho < \frac{1}{4} .
\end{split}
\end{align}
Let us prove the first inequality, the proofs of the other ones being similar. Observe that, by the a priori assumption \eqref{apriori},
\begin{align*}
& {\| f_{<-\frac{1}{3} \log_2 t} \|}_{L^2} \lesssim \e_1 t^{-1/6} ,
\\
& {\| x f_{<-\frac{1}{3} \log_2 t} \|}_{L^2} =
  {\| \partial_\xi [ \chi(t^{1/3} \xi) \widehat{f}(\xi)] \|}_{L^2} \leq t^{1/3}
  {\| \chi'(t^{1/3} \xi) \widehat{f}(\xi) \|}_{L^2} + {\| \chi(t^{1/3} \xi) \partial_\xi \widehat{f}(\xi) \|}_{L^2} \lesssim \e_1 t^{1/6} ,
\end{align*}
and estimate
\begin{align*}
{\| f_{<j} \|}_{L^1} & \leq {\| f_{< -\frac{1}{3} \log_2 t} \|}_{L^1} + \sum_{t^{-1/3} \leq 2^k \leq 2^j } {\| f_k \|}_{L^1}
\\
& \lesssim {\| f_{< -\frac{1}{3} \log_2 t} \|}_{L^2}^{1/2} {\| x f_{<-\frac{1}{3} \log_2 t} \|}_{L^2}^{1/2}
  + \sum_{t^{-1/3} \leq 2^k \leq 2^j } 2^{\frac{k}{4}} t^{\frac{1}{12}} \e_1
\\
& \lesssim \big( 1 + 2^{\frac{j}{4}} t^{\frac{1}{12}} \big) \e_1 \lesssim 2^{\frac{j}{4}} t^{\frac{1}{12}} \e_1.
\end{align*}

\bigskip

\noindent \underline{\it Decomposition of $\partial_t \widehat{f}$.}
Assume that $|\xi| \in (2^j,2^{j+1})$ and split
\begin{equation}
 \label{d_tf=0}
\begin{split}
&\partial_t \widehat{f}(t,\xi)  = - \frac{i}{2\pi} \iint e^{-it\phi(\xi,\eta,\sigma)} \xi \widehat{f}(\xi-\eta-\sigma) \widehat{f}(\eta) \widehat{f}(\sigma)\,d\eta\,d\sigma \\
& = - \sum_{k,l} \underbrace{\frac{i}{2\pi} \iint e^{-it\phi(\xi,\eta,\sigma)} \xi \widehat{f}(\xi-\eta-\sigma)
\widehat{f}(\eta) \widehat{f}(\sigma) \psi\left( \frac{\eta}{2^k} \right) \psi \left( \frac{\sigma}{2^\ell}
\right) \,d\eta\,d\sigma}_{A_{k \ell}}
\\
& = - \Big[ \underbrace{\sum_{\substack{2^k, 2^\ell \lesssim 2^j \\ 2^j > t^{-1/3}}} A_{k \ell}}_{I}
+  \underbrace{\sum_{\substack{2^k \sim 2^\ell \gg 2^j \\ 2^k > t^{-1/3}}} A_{k \ell}}_{II}
+ \underbrace{\sum_{\substack{2^k \gg  2^\ell, 2^j\\ 2^k > t^{-1/3}}} A_{k \ell}}_{III}
+  \underbrace{\sum_{2^j, 2^k, 2^\ell \lesssim t^{-1/3}} A_{k \ell}}_{IV} \Big] + \left\{ \begin{array}{l} \mbox{symmetric} \\ \mbox{ terms} \end{array} \right\}.
\end{split}
\end{equation}

\bigskip

\noindent \underline{\it Contribution of $I$.}
It can be written
\begin{equation*}
\begin{split}
I & = \frac{i}{2\pi} \iint e^{-it\phi(\xi,\eta,\sigma)} \xi \widehat{f_{\lesssim j}}(\xi-\eta-\sigma) \widehat{f_{\lesssim j}}(\eta)
\widehat{f_{\lesssim j}}(\sigma) \chi \left( \frac{\eta}{C2^j} \right) \chi \left( \frac{\sigma}{C2^j} \right) \,d\eta\,d\sigma
\\
& = \frac{2^{3j} i}{2\pi} \iint e^{-it2^{3j} \phi(\xi',\eta',\sigma')} \xi' \widehat{f_{\lesssim j}}(2^j(\xi'-\eta'-\sigma'))
\widehat{f_{\lesssim j}}(2^j \eta') \widehat{f_{ \lesssim j}}(2^j \sigma') \chi  \left( \frac{\eta'}{C} \right) \chi  \left( \frac{\sigma'}{C} \right)  \,d\eta' \,d\sigma'
\end{split}
\end{equation*}
where we changed variables by setting $(\xi',\eta',\sigma') = 2^{-j}(\xi,\eta,\sigma)$. This can also be written
$$
I = \frac{2^{3j} i}{2\pi} \iint e^{-it2^{3j} \phi(\xi',\eta',\sigma')} \xi' F(\eta',\sigma') \chi  \left( \frac{\eta'}{C} \right)  \chi  \left( \frac{\sigma'}{C} \right) \,d\eta'\,d\sigma'
$$
with
$$
F(\eta',\sigma') = \widehat{f_{\lesssim j}}(2^j(\xi'-\eta'-\sigma')) \widehat{f_{\lesssim j}}(2^j \eta') \widehat{f_{\lesssim j}}(2^j \sigma') \, ,
$$
and where $|\xi'| \sim 1$.
Applying Lemma \ref{pinguin2}, in light of \eqref{statpoint0}--\eqref{statphase}, we get, for $|\xi| \geq t^{-1/3}$,
\begin{equation}
\label{eider1}
I = \frac{i \operatorname{sign} \xi}{6t} |\widehat{f}(\xi)|^2 \widehat{f}(\xi)
+ \frac{ic}{t} e^{-it \frac{8}{9} \xi^3} \widehat{f} \left( \frac{\xi}{3} \right)^3
+ 2^{3j} O \left( \frac{ {\| \langle (x,y) \rangle^{2\rho} \widehat{F} \|}_{L^1}}{(2^{3j} t)^{1+\rho}} \right),
\end{equation}
where $c$ is a constant whose exact value will not matter.
Now observe that
$$
\widehat{F}(x,y) = \frac{2^{-3j}}{2\pi} \int e^{-iz\xi} f_{\lesssim j}(2^{-j}(z-x)) f_{\lesssim j}(2^{-j}z) f_{\lesssim j}(2^{-j}(y-z))\,dz
$$
so that
\begin{equation}
\label{eider2}
{\|\widehat{F}\|}_{L^1} \lesssim {\|f_{\lesssim j}\|}_{L^1}^3
  \qquad \mbox{and} \qquad
  {\| |(x,y)|^{2\rho} \widehat{F} \|}_{L^1} \lesssim 2^{2\rho j} {\| f_{\lesssim j} \|}_{L^1}^2 {\| |x|^{2\rho} f_{\lesssim j} \|}_{L^1} .
\end{equation}
Combining \eqref{eider1} and \eqref{eider2} above, and using \eqref{bullfinch2} gives
\begin{align}
\label{bound1}
\begin{split}
& \Big| I - \frac{i\operatorname{sign} \xi}{6t} |\widehat{f}(\xi)|^2 \widehat{f}(\xi)
- \frac{ic}{t} e^{-it \frac{8}{9} \xi^3} \widehat{f} \left( \frac{\xi}{3} \right)^3 \mathbf{1}_{|\xi|>t^{-1/3}} \Big|
\\
& \qquad \lesssim 2^{-3 \rho j} t^{-1-\rho} ( \|f_{\lesssim j} \|_{L^1}^3 + 2^{2\rho j} \| f_{\lesssim j} \|_{L^1}^2 {\| | x |^{2\rho} f_{\lesssim j} \|}_{L^1})
\\
& \qquad \lesssim 2^{j(\frac{3}{4} - 2\rho - \alpha)} t^{-\frac{3}{4} - \frac{2}{3}\rho - \frac{\alpha}{3}} \e_1^3 .
\end{split}
\end{align}
Recall that $\alpha$ is close to, but less than, $\frac{1}{2}$. Choosing $\rho$ close to, but less than, $\frac{1}{4}$, we get that
$\frac{3}{4} - 2\rho - \alpha =: - \kappa < 0$, and  $-\frac{3}{4} - \frac{2}{3}\rho - \frac{\alpha}{3} = -1 - \frac{\kappa}{3}$.
It follows that
\begin{align*}
\int_{2^{-3j}}^\infty \Big| I - \frac{i\operatorname{sign} \xi}{6t} |\widehat{f}(\xi)|^2 \widehat{f}(\xi)
- \frac{ic}{t} e^{-it \frac{8}{9} \xi^3} \widehat{f} \Big( \frac{\xi}{3} \Big)^3 \mathbf{1}_{|\xi|>t^{-1/3}} \Big| \, ds
\\
  \lesssim \e_1^3 \int_{2^{-3j}}^\infty  2^{- \kappa j}
  s^{-1 - \frac{\kappa}{3}} \,ds \lesssim \e_1^3 .
\end{align*}

\bigskip

\noindent \underline{\it Contribution of $II$.}
We essentially follow the same approach as for $I$, and keep in particular the same value for $\rho$ and $\alpha$. A change of variables gives
$$
II = \sum_{\substack{2^k \sim 2^\ell \gg  2^j \\ 2^k > t^{-1/3}}}
  \frac{2^{2k}i}{2\pi} \iint e^{-it2^{3k} \phi((2^{-k} \xi,\eta,\sigma)} \xi \widehat{f_{\lesssim k}}(\xi-2^k(\eta+\sigma))
  \widehat{f_{\lesssim k}}(2^k \eta) \widehat{f_{\lesssim k}}(2^k \sigma) \psi
  \left( \eta \right) \psi \left(\sigma\right) \,d\eta\, d\sigma.
$$
Due to the absence of stationary points, Lemma \ref{pinguin2} implies
$$
|II| \lesssim \sum_{\substack{2^k  \gg  2^j \\ 2^k > t^{-1/3}}}
    2^{j}2^{2k} \frac{ {\| \langle (x,y) \rangle^{\rho} \widehat{F_k} \|}_{L^1}}{(2^{3k} t)^{1+\rho}}
$$
where
$$
F_k(\eta,\sigma) = \widehat{f_{\lesssim k}}(\xi-2^k(\eta+\sigma)) \widehat{f_{\lesssim k}}(2^k \eta) \widehat{f_{\lesssim k}}(2^k \sigma)
$$
and, as above,
$$
{\| |(x,y)|^{\rho} \widehat{F_k} \|}_{L^1} \lesssim 2^{\rho k} {\|f_{\lesssim k}\|}_{L^1}^2 {\||x|^{\rho}f_{\lesssim k}\|}_{L^1}.
$$
As before, using \eqref{bullfinch2}, this leads to
\begin{align}
\label{bound2}
\begin{split}
|II| & \lesssim  \sum_{\substack{2^k \gg  2^j \\ 2^k > t^{-1/3}}} 2^j 2^{2k} \frac{1}{(2^{3k} t)^{1+\rho}} ( 2^{\rho k}
  {\|f_{\lesssim k}\|}_{L^1}^2 {\| | x |^{\rho} f_{\lesssim k} \|}_{L^1} +  {\|f_{\lesssim k}\|}_{L^1}^3 )
\\
& \lesssim  \sum_{\substack{2^k \gg  2^j \\ 2^k > t^{-1/3}}}
2^j 2^{k(-\frac{1}{4} - \frac{5}{2}\rho - \alpha)}  t^{-\frac{3}{4} - \frac{5\rho}6 - \frac{\alpha}{3}} \e_1^3 ,
\end{split}
\end{align}
and thus since $\frac{3}{4} + \frac{5 \rho}{6} + \frac{\alpha}{3} > 1$,
$$
\int_0^\infty |II|\,ds \lesssim \e_1^3
  \int_0^\infty 2^j s^{-\frac{3}{4} - \frac{5\rho}6 - \frac{\alpha}{3}} \max(2^{j},s^{-1/3})^{-\frac{1}{4} - \frac{5}2\rho - \alpha} \, ds \lesssim \e_1^3.
$$

\bigskip

\noindent \underline{\it Contribution of $III$.} For the summands in $III$, $2^k \gg 2^\ell , 2^j$, thus $|\eta|$
is the largest variable and we can write $III = \sum_k A_k$, with
\begin{align*}
A_{k}(\xi) = \frac{i}{2\pi} \iint e^{-it\phi(\xi,\eta,\sigma)} \xi \widehat{f_{\sim k}}(\xi-\eta-\sigma)
\widehat{f_{\sim k}}(\eta) \widehat{f_{\ll k}}(\sigma) \psi\left( \frac{\eta}{2^k} \right) \chi \left( \frac{\sigma}{2^k} \right) \,d\eta\,d\sigma \, .
\end{align*}
On the support of the integrand, $|\partial_\sigma \phi| \sim 2^{2k}$ and
\begin{align}
\label{1/d_siphi}
\left| \partial_\eta^{m_1} \partial_\sigma^{m_2} \frac{1}{\partial_\sigma \phi(\xi,\eta,\sigma)} \right|
  \lesssim 2^{-2k} 2^{-(m_1+m_2)k}
\end{align}
for all integers $m_1,m_2 \in \{ 0, \dots, 10\}$.
We then integrate by parts in $\sigma$ to get
\begin{align*}
\nn
III & = \sum_{2^k \gg  2^j,t^{-1/3}} III_{k}^{(1)} +  III_{k}^{(2)} +  III_{k}^{(3)}
\\
III_{k}^{(1)} & := \frac{i}{2\pi} \iint e^{-it\phi(\xi,\eta,\sigma)} \frac{\xi}{it\partial_\sigma \phi} \, \partial_\si \widehat{f_{\sim k}}(\xi-\eta-\sigma)
\widehat{f_{\sim k}}(\eta) \widehat{f_{\ll k}}(\sigma) \psi\left( \frac{\eta}{2^k} \right) \chi\left( \frac{\sigma}{2^k} \right) \,d\eta\,d\sigma \, ,
\\
III_{k}^{(2)} & := \frac{i}{2\pi} \iint e^{-it\phi(\xi,\eta,\sigma)} \frac{\xi}{it\partial_\sigma \phi} \widehat{f_{\sim k}}(\xi-\eta-\sigma)
\widehat{f_{\sim k}}(\eta) \, \partial_\si \widehat{f_{\ll k}}(\sigma) \psi\left( \frac{\eta}{2^k} \right) \chi \left( \frac{\sigma}{2^k} \right) \,d\eta\,d\sigma \, ,
\\
III_{k}^{(3)} & := \frac{i}{2\pi} \iint e^{-it\phi(\xi,\eta,\sigma)} \xi \,
  \partial_\si \left[ \frac{1}{it\partial_\sigma \phi} \psi\left( \frac{\eta}{2^k} \right) \chi \left( \frac{\sigma}{2^k} \right) \right]
  \widehat{f_{\sim k}}(\xi-\eta-\sigma) \widehat{f_{\sim k}}(\eta) \widehat{f_{\ll k}}(\sigma)
  \,d\eta\,d\sigma \, .
\end{align*}
From \eqref{1/d_siphi} and Lemma \ref{penguin3} it follows that
\begin{align}
\label{bound31}
\begin{split}
\left| III_{k \ell}^{(1)} \right| + \left| III_{k \ell}^{(2)} \right| & \lesssim
  \frac{2^j}{t 2^{2k}} \left[ {\| \partial \widehat{f_{\sim k}} \|}_{L^2}  {\| \widehat{f_{\ll k}} \|}_{L^2}
  + {\| \widehat{f_{\sim k}} \|}_{L^2}  {\| \partial \widehat{f_{\ll k}} \|}_{L^2} \right] {\| u_{\sim k} \|}_{L^\infty}
\lesssim \frac{2^j}{t^{7/6} 2^{3k/2}}\epsilon_1^3.
\end{split}
\end{align}
This gives the desired estimate after summing and integrating in time:
\begin{align*}
\int_0^\infty \sum_{\substack{2^k \gg 2^j  \\ 2^k > t^{-1/3}}} |III_{k}^{(1)}| + |III_{k}^{(2)}| \,ds  \lesssim \e_1^3 \int_{2^{-3k}}^\infty
  \frac{2^j}{s^{7/6}\max(2^j,s^{-1/3})^{3/2}} \, ds
  \lesssim \e_1^3 .
\end{align*}
The remaining term can be estimated similarly using again \eqref{1/d_siphi} and Lemma \ref{penguin3}:
\begin{align}
 \label{bound32}
\left| III_{k}^{(3)} \right|  \lesssim \frac{2^j}{t 2^{3k}} {\| \widehat{f_{\sim k}} \|}_{L^2} {\| \widehat{f_{\ll k}} \|}_{L^2} {\| u_{\sim k} \|}_{L^\infty}
\lesssim\frac{2^j}{t^{4/3} 2^{2k}} \e_1^3,
\end{align}
which gives the desired bound upon summation and time integration since
$$
\int_{2^{-3k}}^\infty \frac{2^j}{s^{4/3} \max(2^j,s^{-1/3})^2} \,ds \lesssim 1.
$$

\bigskip

\noindent
\underline{\it Contribution of $IV$.} Using simply $\|\widehat{f}\|_\infty \leq \varepsilon_1$, the term $IV$ can be estimated by
$$
|IV| \lesssim \sum_{2^j,2^k,2^\ell < t^{-1/3}} 2^{j+k+l} \varepsilon_1^3 \lesssim 2^j t^{-2/3} \mathbf{1}_{t<2^{-3j}} \varepsilon_1^3,
$$
which gives the desired result after time integration.
\end{proof}

\vskip10pt
\subsection{Asymptotics}\label{secasy}

In this section we derive the asymptotic behavior of solutions of \eqref{mKdV} as time goes to infinity.
%
%An important role is played by the self-similar solution of \eqref{mKdV}
%
%We will recall and use some of the properties of $\mathcal{S}$ shortly below.
%
%Also, as it is customary, we let $\mathrm{Ai}$ denote the linear solution to the Airy equation:
%\begin{align}\label{Ai}
%\mathrm{Ai}(y) = \int_\R e^{i (\eta y + s \eta^3)} \, d\eta .
%\end{align}
We are going to show the following:

\begin{proposition}[Asymptotics for small solutions]
\label{proasy}
Let $u$ be a solution of \eqref{mKdV} satisfying the global bounds \eqref{apriori}-\eqref{conc}.
Then, for any $t\geq 2$, the following holds:

\setlength{\leftmargini}{1.5em}
\begin{itemize}

\item In the region $x \geq t^{1/3}$ we have the decay estimate
\begin{align}
\label{prouasy1}
| u(t,x) | \lesssim  \frac{\e_0}{t^{1/3} (x/t^{1/3})^{3/4}} ;
\end{align}

\item In the region $|x| \leq t^{1/3+2\gamma}$, with $\gamma = 1/3 ( 1/6 - C\e_1^2)$, the solution is approximately self-similar:
\begin{align}
\label{prouasy2}
\big| u(t,x) - \frac{1}{t^{1/3}} \varphi\big( \frac{x}{t^{1/3}} \big) \big| \lesssim  \frac{\e_0}{t^{1/3 + 3\gamma/2}} ,
\end{align}
where $\varphi$ is a bounded solution of the Painlev\'e II equation
%\footnote{The existence, uniqueness, and asymptotic properties of solutions to \eqref{painleve} can be found in \cite{HMPII} and \cite{DZPII}.}
\begin{align}
 \label{painleve}
\varphi^{\prime\prime} - 3\xi \varphi + \varphi^3 = 0 , \qquad  \mathrm{p.v.} \int_\R \varphi(x) \, dx = \int_\R u_0(x) \, dx .
\end{align}

\item In the region $x \leq - t^{1/3+2\gamma}$, the solution has a nonlinearly modified asymptotic behavior:
there exists $f_{\infty} \in L^{\infty}_\xi$ such that
\begin{align}
\label{prouasy3}
\Big| u(t,x) - \frac{1}{\sqrt{3 t \xi_0}} \Re
  \exp \Big( -2 i t \xi_0^3 + \frac{i\pi}{4} + \frac{i}{6}|f_\infty (\xi_0) |^2 \log t \Big) f_\infty (\xi_0) \Big|
  \leq \frac{\e_0}{t^{1/3} (-x/t^{1/3})^{3/10}} , %% could not get sharp -3/8 power as in \cite{HG} ...
\end{align}
where $\xi_0 := \sqrt{-x/(3t)}$, and $\Re$ denotes the real part.
\end{itemize}

\end{proposition}

The proof of Proposition \ref{proasy} is given in the remaining of this section.

\medskip
\subsubsection*{Decaying region: Proof of \eqref{prouasy1}}
The proof of \eqref{prouasy1} follows from similar argument to those used in the proof of Lemma \ref{lemlinear2}.
%so we only sketch it below.
As before we denote $\Lambda(\xi) = \xi^3$ and write
\begin{align*}
u(t,x) = e^{-t\partial_x^3} f(t,x) = \frac{1}{\sqrt{2\pi}} \int_{\R} e^{it \phi(\xi)} \what{f}(\xi) \,d\xi ,
\qquad \phi(\xi) := \xi (x/t) + \Lambda(\xi).
\end{align*}
Since for any $x>0$ we have $\partial_\xi \phi = x/t + 3 \xi^2 \geq \max(x/t,\xi^2)$, we integrate by parts in the above formula and bound:
\begin{align*}
& | u(t,x) | \lesssim I + II ,
\\
& I = \int_\R \left| \frac{1}{t \partial_\xi \phi(\xi)} \partial_\xi \what{f}(\xi) \right| \, d\xi ,
\\
& II = \int_\R \left| \frac{1}{t {[\partial_\xi \phi(\xi)]}^2} \partial_{\xi}^2 \phi(\xi) \what{f}(\xi) \right| \, d\xi .
\end{align*}
Using the weighted $L^2$ bound in \eqref{apriori}-\eqref{conc} we can estimate
\begin{align*}
| I | \lesssim \frac{1}{t} \Big[ \int_\R {(x/t + 3\xi^2)}^{-2} \, d\xi \Big]^{1/2} {\| xf \|}_{L^2}
  \lesssim \frac{1}{t} (x/t)^{-3/4} \e_0 t^{1/6} ,
\end{align*}
which is the desired bound. Similarly, we can use the bound on $\what{f}$ to obtain
\begin{align*}
| II | \lesssim \frac{1}{t} \int_\R {(x/t + 3\xi^2)}^{-2} |\xi| \, d\xi {\| \what{f} \|}_{L^\infty}
  \lesssim \frac{1}{t} (x/t)^{-1} \e_0 ,
\end{align*}
which is a stronger bound than what we need since $x \geq t^{1/3}$.

\medskip
\subsubsection*{Self-similar region: Proof of \eqref{prouasy2}}
We now look at the self-similar region $|x| \leq t^{1/3 + 2\gamma}$.
Define $v$ through the identity
\begin{align}
\label{vss}
u(t,x) = \frac{1}{t^{1/3}} v \big(t,\frac{x}{t^{1/3}} \big) , \quad  v(t,x) = t^{1/3} u(t,t^{1/3}x) .
\end{align}
Recall the definition of the scaling vectorfield $S = 1 + x\partial_x + 3t\partial_t$.
A simple computation shows that
\begin{align}
 \label{v_t1}
\partial_t v (t,x) = \frac{1}{3 t^{2/3}} \big( Su \big) (t, t^{1/3}x) .
\end{align}
Moreover, since $u$ is a solution of \eqref{mKdV}, one can verify that
\begin{align}
\label{v_t2}
\partial_t v(t,x) = \frac{1}{t} \partial_x \Big( \frac{1}{3} x v  - v_{xx} - v^3 \Big)(t,x) .
\end{align}

Our aim is to show that $v(t,x)$ is a Cauchy sequence in time with values in $L^{\infty}_x$.
For this we first show that, for all $|x| \leq  t^{2\gamma}$, one has
\begin{align}
\label{estv1}
& \big| P_{\geq 2^{20} t^{\gamma} } v(t,x) \big| \leq \e_0 t^{-3\gamma/2} ,
\\
\label{estv2}
& \big| \partial_t P_{\leq 2^{20} t^{\gamma}}  v(t,x) \big| \leq \e_0 t^{-7/6+3\gamma/2+C\e_1^2} .
\end{align}
%Notice that the above two bounds match given our choice of $\gamma = 1/3 ( 1/6 - C\e_1^2)$.

For \eqref{estv1}, we recall that $f = e^{t\partial_x^3} u$, and write
\begin{align*}
P_{\geq 2^{20} t^{\gamma}} v(t,x) %& = \int_{\R} e^{ ix\xi + i\xi^3} \chi(\xi t^{-\gamma-20}) \what{f}(t,\xi t^{-1/3}) d\xi
  %\\ &
= t^{1/3} \int_{\R} e^{ i \phi(\xi;x,t) } \chi(\xi t^{1/3-\gamma} 2^{-20}) \what{f}(t,\xi) d\xi ,
  \qquad  \phi(\xi;x,t) := x \xi t^{1/3} + t \xi^3 .
\end{align*}
Since for any $|x| \leq  t^{2\gamma}$, we have $|\partial_\xi \phi | \gtrsim \xi^2 t \gtrsim t^{1/3+2\gamma}$ on the support of the above integral,
an integration by parts argument similar to those in the proof of Lemma \ref{lemlinear2}, %and ${\| x f \|}_{L^2} \lesssim \e_1 t^{1/6}$,
shows the validity of \eqref{estv1}.
Notice that a similar bound also holds for $P_{\sim 2^{20} t^{\gamma}}  v(t,x)$.
Because of this, in order to obtain \eqref{estv2}, it suffices to prove the estimate for $P_{\leq 2^{20} t^{\gamma}} \partial_t  v(t,x)$.
Observe that from \eqref{v_t1} one has
$\partial_x^{-1} \partial_t v = 1/(3t) (ISu) (t,t^{1/3}x)$. Therefore, using Bernstein's inequality,
and the bound \eqref{boundISu}, we get
\begin{align*}
\big| P_{\leq 2^{20} t^{\gamma}} \partial_t v(t,x) \big| \lesssim t^{3\gamma/2} {\| \partial_x^{-1} \partial_t v \|}_{L^2}
  \lesssim t^{3\gamma/2 - 1} {\| (ISu) \|}_{L^2} t^{-1/6}
  \lesssim \e_0 t^{-7/6 + 3\gamma/2 + C\e_1^2} ,
\end{align*}
as desired.

We then write
\begin{align*}
v(t,v) = v(t,x) [1-\psi(x/t^{2\gamma})] + P_{\geq 2^{20} t^{\gamma}}v(t,x) \psi(x/t^{2\gamma})
  + P_{\leq 2^{20} t^{\gamma}}v(t,x) \psi(x/t^{2\gamma}) .
\end{align*}
Combining the decay estimate \eqref{lemlinear21} which gives $|v(t,x) [1-\psi(x/t^{2\gamma})]| \lesssim \e_0 t^{-\gamma/2}$,
with \eqref{estv1}-\eqref{estv2}, we see that there exists
$\varphi := \lim_{t \rightarrow \infty} v(t)$ with $|v(t) - \varphi| \lesssim \e_0 t^{-\gamma/2}$.
It also follows that, uniformly for $|x| \leq t^{2\gamma}$,
$$|v(t,x) - \varphi(x) | \lesssim \e_0 t^{-3\gamma/2} + \e_0 \int_t^{\infty} t^{-7/6 + 3\gamma/2 + C_1\e_1^2} \lesssim \e_0 t^{-3\gamma/2}$$
where we recall our choice of $\gamma = 1/3 ( 1/6 - C\e_1^2)$.

To verify that $\varphi$ satisfies the first identity in \eqref{painleve} it suffices to notice that from \eqref{v_t1} and \eqref{v_t2} one has
\begin{align*}
{\| x v  - 3 v_{xx} - 3v^3 \|}_{L^2} = {\| I  Su \|}_{L^2} t^{-1/6} \lesssim \e_0 t^{-1/6 + c\e_1^2} .
\end{align*}
To prove the second identity in \eqref{painleve} we let $0 < a < \gamma/2$ and use $|v(t) - \varphi| \lesssim t^{-\gamma/2}$ to write
\begin{align*}
%\label{asymom1}
\begin{split}
\int \varphi(x) \psi(x/t^a) \, dx %= \int ( \varphi(x) - v(t,x) ) \psi(x/t^a) \, dx + \int v(t,x) \psi(x/t^a) \, dx \\
  =  \int v(t,x) \psi(x/t^a) \, dx + O(t^{a-\gamma/2}) .
\end{split}
\end{align*}
Using Plancherel, and the moment conservation for $u$, we have
\begin{align*}
\int v(t,x) \psi(x/t^a) \, dx = \int ( \what{u}(t,\xi/t^{1/3}) - \what{u}(t,0) ) \what{\psi}(\xi t^a) t^a \, dx + \int u_0(x) \, dx .
\end{align*}
Using the bounds \eqref{apriori}-\eqref{conc} we see that for all $|\xi| \leq 1$
\begin{align*}
| \what{u}(t,\xi/t^{1/3}) - \what{u}(t,0) | \leq | \what{f}(t,\xi/t^{1/3}) (e^{i\xi^3} - 1) |
  + | \what{f}(t,\xi/t^{1/3}) - \what{f}(t,0) | \lesssim |\xi|^{1/2}.
\end{align*}
This shows that
\begin{align*}
\Big| \int \varphi(x) \psi(x/t^a) \, dx - \int u_0(x) \, dx \Big| \lesssim  t^{a-\gamma/2} + t^{-a/2} ,
\end{align*}
which implies \eqref{painleve}.

\medskip
\subsubsection*{Modified scattering: Proof of \eqref{prouasy3}}
The next Lemma gives a refined version of the linear estimate \eqref{lemlinear21}.

\begin{lemma}[Refined linear estimate]\label{lemlinearref}
Let $u = e^{-t\partial_x^3} f$, for $f \in L^2$ satisfying
\begin{align}
 \label{linearrefhyp}
\sup_{t \geq 2} \big( {t}^{-1/6} {\| \langle x \rangle f(t) \|}_{L^2} + {\| \what{f}(t) \|}_{L^\infty_\xi} \big) \leq 1 .
\end{align}
Then, for all $t \geq 2$ and $x \leq -t^{1/3}$,
\begin{align}
\label{linearrefconc}
\Big| u(t,x) - \frac{1}{\sqrt{3t \xi_0}} \,\Re \big( e^{- 2i t\xi_0^3 + i\frac{\pi}{4}} \what{f}(\xi_0) \big) \Big|
  \lesssim \frac{1}{t^{1/3} |x/t^{1/3}|^{3/10} } ,
\end{align}
where $\xi_0 := \sqrt{-x/(3t)}$, and $\Re$ denotes the real part.
\end{lemma}

This result can be proven by similar arguments to those in the proof of Lemma \ref{lemlinear2}, and those of Lemma 3.2 in \cite{IoPunote}.
For completeness we give the main ideas of proof below.

\begin{proof}[Proof of Lemma \ref{lemlinearref}]
We write
\begin{align}
\label{ufhat}
 u(t,x) = \sqrt{\frac{2}{\pi}} \Re \int_0^\infty e^{i t \phi(\xi)} \what{f}(\xi) \, d\xi ,
  \qquad \phi(\xi) := \frac{x}{t} \xi + \xi^3 .
\end{align}
As before we let $\xi_0 := \sqrt{-x/3t} \approx t^{-1/3}(-x/t^{1/3})^{1/2}\gtrsim t^{-1/3}$ be the only stationary point of the phase $\phi$ in \eqref{ufhat}.

We first look at the frequency region with $|\xi - \xi_0| \geq \xi_0/2$. There we have $|\partial_\xi \phi(\xi)| \gtrsim \max(\xi^2,\xi_0^2)$.
Then, an integration by parts like the one in the proof of Lemma \ref{lemlinear2} (cfr. the terms $C_1$ and $C_2$ there) gives us a bound
of the form $t^{-1} \xi_0^{-2} + t^{-5/6} \xi_0^{-3/2} \lesssim t^{-1/3} (-x/t^{1/3})^{-3/4}$,
which is smaller than the right-hand side of \eqref{linearrefconc}.

We then analyze the case with $|\xi - \xi_0| \leq \xi_0/2$.
If $|\xi-\xi_0| \approx 2^\ell$, for $\ell \geq \ell_0$ with
$$2^{\ell_0} \approx t^{-1/3} (-x/t^{1/3})^{-1/5},$$
%\comment{this is the optimal cut I got, but $1/4$ instead of $1/5$ should be the optimal value,
% which would give $3/8$ instead of $3/10$ in \eqref{linearrefconc}. But I could not see an argument giving $3/8$.
%It is still better than $1/4$ which is sufficient.}
we integrate by parts in frequency.
Using $|\partial_\xi \phi(\xi)| \gtrsim 2^\ell \xi_0$, we bound these contributions by
\begin{align*}
t^{-1} \Big( {\| \partial_\xi \what{f} \|}_{L^2} \xi_0^{-1} 2^{-\ell/2}
  + {\|\what{f} \|}_{L^\infty} \xi_0^{-1} 2^{-\ell} \Big) .
\end{align*}
Using \eqref{linearrefhyp}, and the definitions of $\xi_0$ and $\ell_0$,
we see that the contribution from the region $|\xi-\xi_0| \geq 2^{\ell_0}$
is of the order of $t^{-1/3} (-x/t^{1/3})^{-3/10}$, which is an acceptable remainder.

We are then left with estimating the contribution to the integral \eqref{ufhat} coming from the region $|\xi - \xi_0| \leq 2^{\ell_0}$.
We write this contribution as
\begin{align*}
& \sqrt{\frac{2}{\pi}} \Re \int_0^\infty e^{i t \phi(\xi)} \chi( (\xi-\xi_0) 2^{-\ell_0})\what{f}(\xi) \, d\xi
  = A + B + C
\\
& A = \sqrt{\frac{2}{\pi}} \Re \big( e^{i t \phi(\xi_0)} \what{f}(\xi_0) \int_0^\infty e^{it3\xi_0 \xi^2/2} \chi(\xi/2^{\ell_0})\, d\xi \Big)
\\
& B = \sqrt{\frac{2}{\pi}} \Re \int_0^\infty
  \Big( e^{i t \phi(\xi)} - e^{i t \phi(\xi_0) + it \phi^{\prime\prime}(\xi_0) (\xi-\xi_0)^2/2 } \Big) \chi( (\xi-\xi_0) 2^{-\ell_0})\what{f}(\xi) \, d\xi
\\
& C = \sqrt{\frac{2}{\pi}} \Re e^{i t \phi(\xi_0)} \int_0^\infty
  e^{it \phi^{\prime\prime}(\xi_0) (\xi-\xi_0)^2/2 } \chi( (\xi-\xi_0) 2^{-\ell_0}) \big( \what{f}(\xi) - \what{f}(\xi_0) \big)\, d\xi .
\end{align*}

Using the hypotheses we immediately see that
\begin{align*}
& |B| \lesssim t 2^{4 \ell_0} \lesssim t^{-1/3} (-x/t^{1/3})^{-4/5} ,
\\
& |C| \lesssim t^{1/6} 2^{3 \ell_0/2} \lesssim t^{-1/3} (-x/t^{1/3})^{-3/10} ,
\end{align*}
so that these terms are acceptable remainders.

Using the formula
\begin{align*}
\int_\R e^{-ax^2} \, dx = \sqrt{\frac{\pi}{a}} , \qquad a \in \C , \quad \Re a > 0 ,
\end{align*}
we see that
\begin{align*}
\int_0^\infty e^{it3\xi_0 \xi^2/2} e^{-\xi^2/2^{\ell_0}} \, d\xi = \frac{1}{2} \sqrt{\frac{2\pi}{-i3t\xi_0}}
  + O \big( 2^{2\ell_0} + 2^{-\ell_0} (t\xi_0)^{-3/2} \big) .
\end{align*}

Finally, it follows that
\begin{align*}
A = \Re \sqrt{\frac{i}{3t\xi_0}} e^{it\phi(\xi_0)} \what{f}(\xi_0) + O \big( t^{-2/3} |x/t^{1/3}|^{-3/10} \big) ,
\end{align*}
and this completes the proof of the Lemma.
\end{proof}

Notice that in the region $x \leq -t^{1/3+2\gamma}$ one has $\xi_0 = \sqrt{x/(-3t)} \gtrsim t^{-1/3+\gamma} \gg t^{-1/3}$.
Our next goal is then to identify an asymptotic profile for $\what{f}(\xi)$, where $f = e^{t\partial_x^3} u$ and $u$ solves \eqref{mKdV},
whenever $|\xi| \gg t^{-1/3+\gamma}$.
%is close to $\xi_0 = \sqrt{-x/(3t)}$, and $x \leq - t^{1/3+2\gamma}$.
This will then determine the leading order asymptotic term for $u$ in this region via \eqref{linearrefconc}.

\begin{lemma}\label{lemasyprof}
Let $f = e^{t\partial_x^3} u$ with $u$ satisfying the bounds \eqref{apriori}-\eqref{conc},
and let us define the modified profile as in \eqref{modprof}:
\begin{align}
\label{asyprof0}
\what{w}(t,\xi) := e^{- i B(t,\xi)} \what{f}(t,\xi)  , \, \quad
  B(t,\xi) := \frac{1}{6} \operatorname{sign} \xi \int_1^t \big|\what{f}(s,\xi) \big|^2 \frac{ds}{s} \, .
\end{align}
Then there exists $w_\infty \in L^\infty$ such that, for all $t \geq 2$, and $|\xi| \geq t^{-1/3 + \gamma}$
\begin{align}
\label{asyprof1}
| \what{w}(t,\xi) - w_\infty(\xi) | \lesssim \e_0 (|\xi| t^{1/3})^{- \kappa} ,  %expect \kappa = 1/4-$  %t^{-\gamma/3}
\end{align}
for any $\kappa \in (0,1/4)$.
Moreover, there exists $f_\infty \in L^\infty$ such that, for $|\xi| \geq t^{-1/3 + \gamma}$, %for $\xi_0 = \sqrt{-x/t}$,
we have
\begin{align}
\label{asyprof2}
%\sup_{|\xi| \geq t^{-1/3 + 2\gamma}}
  \Big| \what{f}(t,\xi) - \exp \Big(\frac{i}{6}
  \operatorname{sign} \xi |f_\infty (\xi) |^2 \log t \Big) f_\infty (\xi) \Big| \lesssim \e_0 (|\xi| t^{1/3})^{-\kappa} . %t^{-\gamma/3}\log t .
\end{align}
\end{lemma}

\begin{proof}
To prove \eqref{asyprof1} it suffices to show that for all times $t_2 \geq t_1 \geq 2$, one has
\begin{align}
\label{asyprof11}
| \what{w}(t_1,\xi) - \what{w}(t_2,\xi)| \leq \e_1^3 { \big( 2^j t_1^{1/3} \big) }^{-\kappa} .
\end{align}
for every $|\xi| \approx 2^j$, with  $j\in\mathbb{Z}$ and $2^j \geq t_1^{-1/3 + \gamma}$.
The starting point to prove \eqref{asyprof11} is the formula \eqref{d_tf=0} which, for $|\xi| \geq t^{-1/3 + \gamma} \gg t^{-1/3}$, reads
\begin{align*}
\partial_t \what{f}(t,\xi) =  I + II + III ,
\end{align*}
where all the terms on the right-hand side are defined in \eqref{d_tf=0}.
From \eqref{bound1} and the definition of the modified profile $\what{w}$ in \eqref{asyprof0}, we see that,
for $t_1 \leq t \leq t_2$,
\begin{align}
\label{asyprof12}
\begin{split}
\Big| \partial_t \what{w}(t,\xi)
  - e^{- i B(t,\xi)} \frac{ic}{t} e^{it \frac{8}{9} \xi^3} \widehat{f} (\xi/3)^3 \Big|  %% \mathbf{1}_{|\xi|>t^{-1/3}}
  \lesssim  2^{-j\kappa} t^{-1-\kappa/3} \e_1^3 + |II(t,\xi)| + |III(t,\xi)|, %+ |IV(t,\xi)|,
\end{split}
\end{align}
where we recall that we have previously defined $\kappa = -\frac{3}{4} + 2\rho + \alpha$,
and we can choose $0<\alpha<\frac{1}{2}$ and $0<\rho<\frac{1}{4}$ so that $\kappa = 1/4 - \beta$, for any small $\beta > 0$.
To prove \eqref{asyprof1} it will then suffice to show
\begin{align}
\label{asyprof20}
\Big| \int_{t_1}^{t_2} e^{- i B(t,\xi)} \frac{ic}{t} e^{-it \frac{8}{9} \xi^3} \what{f}(\xi/3)^3 \, dt \Big|
  \lesssim \e_1^3 { \big( 2^j t_1^{1/3} \big) }^{-\kappa} ,
\end{align}
for all $|\xi| \geq t_1^{-1/3 + \gamma}$, and
\begin{align}
 \label{asyprof21}
|II(t,\xi)| + |III(t,\xi)| \lesssim \e_1^3 t^{-1-\kappa/3} 2^{-\kappa j} ,
\end{align}
for $t_1 \leq t \leq t_2$, and $|\xi| \geq t^{-1/3 + \gamma}$.
Here we have used the fact that the first term on the right-hand side of \eqref{asyprof12} matches the right-hand side of \eqref{asyprof21},
which, upon integration between $t_1$ and $t_2$, gives the desired bound.

To prove \eqref{asyprof20} we use an integration by parts argument similar to the one that gave us \eqref{estextraR}.
%In particular, letting $t_1^\prime := \max(t_1, 2^{-3j})$, and
Proceeding as in \eqref{estextraRIBP}, we see that
\begin{align*}
\begin{split}
 \Big| & \int_{t_1}^{t_2} e^{it \frac{8}{9} \xi^3} e^{-iB(t,\xi)} \widehat{f} (t,\xi/3)^3 \frac{dt}{t} \Big|
  \lesssim J^\prime + K^\prime + L^\prime + M^\prime
\\
& J^\prime = \frac{1}{|\xi|^3} |\widehat{f} (t,\xi/3) |^3 \frac{1}{t} \,\, \Big|^{t_2}_{t=t_1} \, ,
\\
& K^\prime = \int_{t_1}^{t_2} \frac{1}{|\xi|^3} |\partial_t \widehat{f} (t,\xi/3) |  |\widehat{f} (t,\xi/3)|^2 \frac{dt}{t} \, ,
\\
& L^\prime = \int_{t_1}^{t_2} \frac{1}{|\xi|^3} | \partial_t B(t,\xi)|  |\widehat{f}(t,\xi/3)|^3 \frac{dt}{t} \, ,
\\
& M^\prime = \int_{t_1}^{t_2} \frac{1}{|\xi|^3} |\widehat{f}(t,\xi/3)|^3 \frac{dt}{t^2} \, .
\end{split}
\end{align*}
Using ${\|\what{f}(t)\|}_{L^\infty} \leq \e_1$ we immediately see that $J^\prime \lesssim \e_1^3 2^{-3j} t_1^{-1}$,
which is more than sufficient, since $2^j t_1^{1/3} \gg 1$.
Using \eqref{d_tf=} and \eqref{estR} we see that
\begin{align*}
K^\prime \lesssim \int_{t_1}^{t_2} \frac{1}{|\xi|^3} \Big[ \frac{\e_1^3}{t} +  R(t,\xi) \Big] \e_1^2 \frac{dt}{t}
  \lesssim \e_1^2 2^{-3j} t_1^{-1} \Big[ \e_1^3 + \int_{t_1}^{t_2} R(t,\xi) \,dt  \Big]
  \lesssim \e_1^5 2^{-3j} t_1^{-1} \, .
\end{align*}
$L^\prime$ and $M^\prime$ can be bounded similarly, using also $|\partial_t B(t,\xi)| \leq \e_1^2 t^{-1}$.
%From the definition of $B$ in \eqref{modprof} we see that
%\begin{align*}
%& L \lesssim \int_{|\xi|^{-3}}^t \frac{1}{|\xi|^3} \frac{\e_1^2}{s} \e_1^3 \frac{ds}{s}  \lesssim \e_1^5 \, .
%\end{align*}
%It is easy to see that the same bound also holds for the last term $M$, so that the proof of \eqref{estextraR} is completed.

We now prove \eqref{asyprof21}.
To bound $II$ we look back at the estimate \eqref{bound2}, recall that $\kappa = -\frac{3}{4} + 2\rho + \alpha$,
and see that
\begin{align*}
|II(t,\xi)| \lesssim \e_1^3 2^j \sum_{2^k \gg 2^j} 2^{(-1-\kappa) k} t^{-1 - \kappa/3} \lesssim \e_1^3 2^{-\kappa j} t^{-1 - \kappa/3} .
  %\lesssim \e_1^3 t^{-1-\gamma\kappa}.
\end{align*}
%having used once again the restriction on $|\xi| \approx 2^j \geq t^{-1/3+\gamma}$.
To estimate $III$ we recall \eqref{bound31} and \eqref{bound32}, and, in the case $2^j \geq t^{-1/3+\gamma}$, deduce the following:
\begin{align*}
|III(t,\xi)| \lesssim \e_1^3  2^j \sum_{2^k \gg 2^\ell \gtrsim 2^j}
  \big( t^{-7/6} 2^{-3k/2} + t^{-4/3} 2^{-2k} \big)
  \\
  \lesssim \e_1^3 ( 2^{-j/2} t^{-7/6} + 2^{-j} t^{-4/3} )
  \lesssim \e_1^3 t^{-1} (t^{1/3} 2^{j} )^{-1/2}.
\end{align*}
%
%Similarly, to estimate the last term on the left-hand side of \eqref{asyprof21},
%we look at the bound obtained in \eqref{bound4}, and proceed exactly as before
%to see that $|IV(t,\xi)| \lesssim \e_1^3 t^{-1-2\gamma}$.
This completes the proof of \eqref{asyprof21} and gives us \eqref{asyprof11}.
We also deduce that $\what{w}(t)$ is a Cauchy sequence and obtain the existence of a limit profile $w_\infty$ as in \eqref{asyprof1}.

To prove \eqref{asyprof2} we begin by observing that \eqref{asyprof1} implies that for $t\geq 2$
\begin{align}
\label{asyprof30}
\big| |\what{f}(t,\xi)| - |w_\infty(\xi)| \big| \lesssim \e_1^3 (|\xi| t^{1/3})^{- \kappa} .
\end{align}
Next, for $B$ as in \eqref{asyprof0}, we define
\begin{align}
\label{asyprof31}
A(t,\xi) := B(t,\xi) - \frac{1}{6} \operatorname{sign}\xi \, {|\what{f}(t,\xi)|}^2 \log t .
\end{align}
Omitting the variable $\xi$, we calculate for $2\leq t_1 \leq t_2$
\begin{align*}
A(t_2) - A(t_1) = \frac{1}{6} \operatorname{sign}\xi
  \int_{t_1}^{t_2} \big( {|\what{f}(s)|}^2 - {|\what{f}(t_2)|}^2 \big) \frac{ds}{s}
  + \frac{1}{6} \operatorname{sign}\xi \big( {|\what{f}(t_1)|}^2 - {|\what{f}(t_2)|}^2 \big) \log t_1 .
\end{align*}
From this and \eqref{asyprof30} we deduce that $A(t,\xi)$ is a Cauchy sequence in time,
and there exists $A_\infty \in L^\infty_\xi$ such that
\begin{align*}
|A(t,\xi) - A_\infty(\xi)| \lesssim \e_1^3 (|\xi| t^{1/3})^{- \kappa} \log t.
\end{align*}
Thanks to \eqref{asyprof30} and \eqref{asyprof31} we see that
\begin{align*}
\big |B(t,\xi) - \big( A_\infty(\xi) +
  \frac{1}{6} \operatorname{sign}\xi \, {|w_\infty(\xi)|}^2 \log t \big) \big| \lesssim \e_1^3 (|\xi| t^{1/3})^{- \kappa} \log t ,
\end{align*}
and, in view of \eqref{asyprof0} and \eqref{asyprof1}, we obtain
\begin{align*}
\big | \what{f}(t,\xi) - w_\infty(\xi) \exp \big( i A_\infty(\xi) +
  \frac{i}{6} \operatorname{sign}\xi \, {|w_\infty(\xi)|}^2 \log t \big) \big| \lesssim \e_1^3 (|\xi| t^{1/3})^{- \kappa} \log t .
\end{align*}
The desired conclusion \eqref{asyprof2} follows by defining $f_\infty(\xi) := w_\infty(\xi) \exp(i A_\infty(\xi))$.
\end{proof}

\medskip
Finally, we observe that in the space-time region $x/t^{1/3} \leq - t^{2\gamma}$ we have
$\xi_0 = \sqrt{-x/3t} \approx (-x/t^{1/3})^{1/2} t^{-1/3} \geq t^{-1/3 + \gamma}$,
and we can then combine the refined linear estimate \eqref{linearrefconc} in Lemma \ref{lemlinearref},
and the modified asymptotic estimate \eqref{asyprof2} in Lemma \ref{lemasyprof} to obtain:
\begin{align*}
%\label{uasy3}
\Big| u(t,x) - \frac{1}{\sqrt{3t\xi_0}} \Re \Big\{ \exp \Big( -2i t \xi_0^3 + i\frac{\pi}{4} +
  \frac{i}{6} |f_\infty(\xi_0)|^2 \log t \Big) f_\infty (\xi_0)  \Big\} \Big|
  \\
\lesssim \e_0 {(t \xi_0)}^{-1/2} (t^{1/3}\xi_0)^{- \kappa}  \log t + \e_0 t^{-1/3} |x/t^{1/3}|^{-3/10}
\end{align*}
for $f_\infty \in L^\infty$, and whenever $x/t^{1/3} \leq - t^{2\gamma}$.
Since $\kappa$ can be chosen arbitrarily close to $1/4$, this gives \eqref{prouasy3} and concludes the proof of Proposition \ref{proasy}.
$\hfill \Box$

\bigskip

\section{Stability of solitons}\label{secsoliton}

In this section, shall study the
the asymptotic stability of the solitons
\begin{align*}
 Q_{c}(x-ct) = \sqrt{c} Q (\sqrt{c} ( x - ct) ) , \qquad Q(s) := \sqrt{2}/\cosh (s) , \quad c>0,
\end{align*}
for the focusing mKdV equation
\beq
\label{mKdVbis}
\partial_{t} u +\partial_{x}^3 u  +  \partial_{x} u^3 = 0.
\eeq
 The aim is to prove Theorem \ref{maintheo2}.
%We are going to prove the following (denoting $x_{+} := \max(x,0)$):
%\begin{theorem}
%\label{theosoliton}
%Assume that $u_{0}= Q_{c_{0}} + v_{0}$ with
%\beq
%\label{hypv0sol0}
%\|  v_{0} \|_{H^{1}} + \|  x v_{0} \|_{H^{1}}  + \| \langle x_{+} \rangle^m  v_{0}\|_{H^1} \leq \eps_{0} ,
%\eeq
%for some $m > 3/2$.
%Then, for sufficiently small $\eps_{0}$, there exist modulation parameters $c(t)$ (the scaling) and $x(t)$ (the shift),
%such that the solution of \eqref{mKdVbis} with $u(t=0,x) = u_0(x)$ has the form
%\begin{align*}
%u(t,x)  =  Q_{c(t)}(x - x(t)) +  R(t,x)
%\end{align*}
%where the remainder $R$ is of size $\e_0$ and tends to zero in $L^\infty$,
%and $c(t)$ and $\dot{x}(t)$ tend to a fixed value $c_{+}$.
%
%
%\end{theorem}
%
%\comment{Should state a more precise result for the convergence using the more precise description of the previous part,
%and the estimates obtained in section \ref{secsoliton2}.}
%
%
%\begin{remark}\normalfont
%Note that the decay assumption that we require in front of the solitary wave is roughly $ x^{3/2} u$ is in $L^2$.
%This is almost at the same scale as the decay property which is assumed to use the inverse scattering theory,
%where one requires $x u_{0}$ is in $L^1$. We refer to \cite{DZmKdV}, and to \cite{Klein-Saut} for a recent survey.
%\end{remark}
%
%
This will be obtained by combining the modified scattering result of the previous section and an asymptotic
stability result in a  weighted space for the soliton (Theorem \ref{theosoliton2}).

 For a smooth non-negative weight $w$, we shall use the following notation for the  weighted norms:
 $$
\|u\|_{L^2_{w}} = \| w \, u \|_{L^2}, \quad \|u\|_{H^s_{w}}^2 = \sum_{k \leq s}  \| w \, \partial_x^k u \|_{L^2}^2.
$$
 In the following, we shall use as weights,  $w(x)= (1 + \tanh (\delta x))^{1/2}$, with $\delta$ sufficiently small, %depending on $c_0$...
and $w'$.
We shall first prove:

\begin{theorem}
\label{theosoliton2}
%Let $c_0 > 0$.
For every $\eps_{1} > 0$ there exists $\eps_0$ such that the following holds true:
if  %consider an initial data for \eqref{mKdVbis} of the form $u_{0} = Q_{c_{0}} + v_{0}$ with
$v_0$ satisfies
\beq
\label{hypv0sol}
\|v_{0} \|_{H^1} + \| \langle x_{+} \rangle^m  v_{0}\|_{H^1} \leq \eps_{0}
\eeq
for some fixed $m>1/2$, then there exists a shift $h(t)$ and a modulation speed $c(t)$
such that the solution of \eqref{mKdVbis} with $u(t=0) = Q_{c_{0}} + v_{0}$ satisfies
\beq
\label{usplit} u(t,x)  =   Q_{c(t)}( y) + v(t,y), \quad y= y(x,t)= x - \int_{0}^t c(s) \, ds + h(t) ,
\eeq
with
\beq
\label{estsoliton}
\|\langle y_{+} \rangle^m v(t) \|_{H^1} +  \langle t \rangle^m  \|v(t) \|_{H^1_{w}} +  \langle t \rangle^{2 m } ( |c'(t)|  +  | h'(t)|)
  \lesssim \eps_{1}, \quad \forall t \geq 0 .\eeq
Moreover, one has the bound
\beq
\label{estsoliton2}
  \int_{0}^{\infty} \|v \|_{H^2_{w'}}^2 \, dt \lesssim \eps_{1}^2 .
\eeq
%and there exists $c_{+}$ and $h_{+}$ such that
%$$ |c(t) - c_{+}| + |h(t) - h_{+}| \lesssim { \eps_{1} \over \langle t \rangle^{2m-1}}.$$
\end{theorem}

Note that this Theorem gives in particular that perturbations of a solitary wave decay to its right.
This kind of result was already obtained in \cite{Pego-Weinstein1,Mizumachi1,Martel-Merle}. Nevertheless, we establish
 here a form of the result which is appropriate for the proof of Theorem \ref{maintheo2}. In particular, we prove
  rates of decay that will be useful in order to describe the radiation behind the solitary wave,
  following the approach of the previous section in a second step.

\bigskip
\subsection{Proof of Theorem \ref{theosoliton2}}\label{secsoliton1}
 We shall split the proof in several steps.

\subsubsection*{Step 1: Linear estimates in exponentially  weighted spaces}
In this first step, we shall recall the properties of the equation  \eqref{mKdVbis} linearized about the solitary wave $Q_{c}$.
By changing variables from $x$ to $ y= x-ct$, we obtain the linearized equation
\beq
\label{KdVlinK}
\partial_{t} v - c \partial_{y} v + \partial_{y}^3 v + 3 \partial_{y}\left(  Q_{c}^2 v\right) = 0.
\eeq
Let us denote by $S_{c}(t)$ the linear group associated to this linear equation, so that the solution
of \eqref{KdVlinK} with initial value $v_{0}$ can be written as $v(t)= S_{c}(t) v_{0}$.
We shall  recall  the decay results  for $S_{c}$  obtained by  Pego-Weinstein \cite{Pego-Weinstein1}  by using the weighted norms
$$
\|f \|_{L^2_{a}} := \|e^{a y} f \|_{L^2}, \quad  \| f \|_{H^1_{a}}^2 :=  \|e^{a y} f \|_{L^2}^2 + \|e^{a y}\partial_y  f \|_{L^2}^2
$$
where $a$ is chosen  so that
\beq
\label{achoix} 0 < a < \sqrt{c/3}.
\eeq
%We set $H^0_{a}= L^2_{a}$.

Let us define
\begin{align}
 \label{L_c}
\mathcal{L}_{c} :=  \partial_{y}\big( - c v + \partial_{y}^2 v  + 3 Q_{c}^2 v).
\end{align}
and  $\xi_{c}^1(y)=  \partial_{y} Q_{c}$, $  \xi_{c}^2 (y) = \partial_{c} Q_{c}$, that describe the generalized kernel of $\mathcal{L}_{c}$:
\begin{align*}
\mathcal{L}_{c} \xi_{c}^1= 0, \quad \mathcal{L}_{c}  \xi_{c}^2 = \xi_{c}^1.
\end{align*}

To define a projection on this generalized kernel, we use the generalized kernel of the adjoint (for the $L^2$ scalar
product) $\mathcal{L}_{c}^*$.  Let us set
\begin{equation}
\label{zetadef}
\zeta_{c}^1 (y)= - \alpha_{1} \left(  \int_{- \infty}^y \partial_{c} Q_{c} \right)+ \alpha_{2} Q_{c}(y), \quad  \zeta_{c}^2 (y)= \alpha_{1} Q_{c},
\end{equation}
where the normalization factors $\alpha_{1}$ and $\alpha_{2}$ are chosen so that\footnote{Note that 
these integrals are well defined thanks to the  fast decay of the $\xi_{c}^i$.}
$$ \int \xi_{c}^i \zeta_{c}^j= \delta_{ij}, \quad 1 \leq i,\, j \leq 2.$$
Note that
$$
\mathcal{L}_c^* \zeta_1 = \zeta_2, \qquad \mathcal{L}_c^* \zeta_2 = 0 .
$$
Define the projections
\begin{align}
\label{pcqc}
\mathcal{P}_{c} v = (v, \zeta_{c}^1)_{L^2} \xi_{c}^1 + (v, \zeta_{c}^2)_{L^2} \xi_{c}^2, \qquad \mathcal{Q}_{c}= I - \mathcal{P}_{c}
\end{align}
Note that these projections are well defined on $L^2_{a}$ and commute with $\mathcal{L}_c$ as well as $S_c(t)$ for all $t$. 
From the linear stability of the solitary wave, one has:

\begin{theorem}[Pego-Weinstein \cite{Pego-Weinstein1}, Theorem 4.2]\label{theoPego}
We have the following decay and smoothing estimates:
\begin{align*}
&  \| S_{c}(t) \mathcal{Q}_{c} v \|_{L^2_{a}} \lesssim  e^{- bt } \|v \|_{L^2_{a}},
\\
&    \| S_{c}(t) \mathcal{Q}_{c} v \|_{H^1_{a}} + \|  S_{c}(t) \mathcal{Q}_{c}  \partial_{y }v \|_{L^2_{a}}  \lesssim
     e^{- bt }  \max \big( 1, t^{-1/2} \big) \|v \|_{L^2_{a}}.
\end{align*}
for some $b>0$.
\end{theorem}

%\comment{PG $\to$ FR: sorry it might be stupid, but can you say a word about how you bound $\|  S_{c}(t) \mathcal{Q}_{c}  \partial_{y }v \|_{L^2_{a}}$?
% FR: we get the estimate by  writing Duhamel formula with $e^{\t \partial_{x}^3}$...
% }

 By induction, we can deduce from the above estimates and the Duhamel formula  that
 \begin{align*}
&  \| S_{c}(t) \mathcal{Q}_{c} v \|_{H^k_{a}} \lesssim  e^{- bt } \|v \|_{H^k_{a}},
\\
&    \| S_{c}(t) \mathcal{Q}_{c} v \|_{H^{k+1}_{a}} + \|  S_{c}(t) \mathcal{Q}_{c}  \partial_{y }v \|_{H^k_{a}}  \lesssim
     e^{- bt }  \max \big( 1, t^{-1/2} \big) \|v \|_{H^k_{a}}
\end{align*}
for every $k \geq 0$.
\subsubsection*{Step 2: Decomposition of the perturbation}
The perturbation of the solitary wave $v(t,y)$  defined in \eqref{usplit} evolves according to
\beq
\label{eqperturb}
\partial_{t} v - \tilde{c} \partial_{y} v + 3\partial_{y} (Q_{c(t)}^2 v) + \partial_{y}^3 v = \partial_{y}\mathcal{F} (v) + e_{Q}
  , \quad v_{/t=0}= v_{0}(x)
\eeq
where
\begin{align}
& \nonumber\tilde c(t) = c(t)- \dot h(t) \\
& \label{eQ} e_{Q}(t,y) =  \dot c \partial_{c} Q_{c(t)}(y) + \dot h  \partial_{y} Q_{c(t)}(y) = \dot c \xi_{c(t)}^2(y) + \dot h \xi_{c(t)}^1(y) \\
& \nonumber  \mathcal{F}(v)= - \left( ( Q_{c}+ v)^3 - Q_{c}^3 - 3 Q_{c}^2 v \right).
\end{align}
The  modulation parameters $(h(t), c(t))$ will be chosen to ensure the constraint
\beq
\label{contrainte2}
(v, \zeta_{c}^1)_{L^2}= (v, \zeta_{c}^2)_{L^2}= 0.
\eeq
Note that these constraints are always well defined (even the first one)
when $v$ is such that $\langle y_+ \rangle^m v \in L^2$ for $m>1/2$.

We shall use Mizumachi's  \cite{Mizumachi2} approach that consists in splitting  $v(t,y)$ defined in \eqref{usplit} into
\beq
\label{mizudec}
v (t,y)= v_{1}(t,y)+ v_{2}(t,y)
\eeq
where $v_1$ will be estimated in $H^1_w$, and $v_2$ in $H^1_a$.
We choose $v_{1} (t,y)$ as the solution of the free  nonlinear  equation
\beq
\label{freemKdV}
\partial_{t}  v_{1} -  \tilde{c} \partial_{y} v_{1} + \partial_{y}^3 v_{1} + \partial_{y} v^3_{1} = 0 , \quad  v_1(0)= v_{0},
\eeq
and $v_{2}$ as the solution of
\beq
  \label{KdVexp}
  \partial_{t} v_{2} - \tilde{c} \partial_{y} v_{2}  +  3  \partial_{y} (Q_{c(t)}^2 v_{2})  + \partial_{y}^3 v_2 =
  \partial_{y}\mathcal{N} (v) + e_{Q} , \quad v_{2}(0)= 0 ,
\eeq
with
\begin{align}
\label{defNdev}
\mathcal{N}(v)=  -  ( Q_{c(t)} + v_{1}+ v_{2})^3  + Q_{c(t)}^3  + v_{1}^3  +  3 Q_{c(t)}^2 v_{2}.
\end{align}

We shall solve this equation for $v_{2}$ in the weighted space $H^{1}_{a}$ by using estimates for the linear semigroup $S_c$.
Note that the choice of the equation for $v_{2}$ is made
in order to ensure that the source term $\mathcal{N}(v)$ that involves  $v_{1}$ lies in the weighted space $L^2_{a}$.

Let us define the norm:
\begin{align}
\label{N(t)}
\begin{split}
N(t) :=   %N(v(t)) =
\langle t \rangle^m ( {\|v_{1}(t) \|}_{H^1_{w}} + {\| v_{2}(t) \|}_{H^1_{a}}) & + {\|\langle y_{+} \rangle^m v_{1}(t) \|}_{H^1} + {\|v_{2}\|}_{H^1}
\\
& + {| c(t) - c_{0}|} +  {| h(t) - h_{0}|} ,
\end{split}
\end{align}
with the parameters $\delta$ in the definition of $w$, and $a$ in the exponential weights,
chosen so that the following relations hold:
\beq
\label{lienpoids}
Q_{c  \pm  5 \eps_{0}}^{1/3} e^{\kappa |x|} \lesssim w + w' \lesssim e^{a x} , \quad   \, \forall x \in \mathbb{R},
\eeq
for a small constant $\kappa>0$.

\bigskip
\noindent
\underline{The bootstrap argument.} We assume that
\beq
\label{hypbootstrapsol}
N(t) \lesssim \tilde{\epsilon}_1, \quad \forall \,\, t \in [0,T]
\eeq
 and we will prove that, for all $t \in [0, T]$
$$
N(t) \lesssim \epsilon_0^{1\over 2}.
$$
It will be convenient to use also the  quantity
$$ M(t) =  \sup_{s\in[0, t]} \Big(  \langle s \rangle^m ( {\|v_{1}(s) \|}_{H^1_{w}} + {\| v_{2}(s) \|}_{H^1_{a}} \Big).$$
Note that by the bootstrap assumption, we also have that $M(t) \leq \tilde{\eps}_{1}$ on $[0, T]$.

\subsubsection*{Step 3: $H^1$ estimate}
In this step we shall prove that
\begin{proposition}
\label{vH1}
For $t \in [0,T]$ we have the estimate
$$ \|v_{1}(t)\|_{H^1} \lesssim  \eps_{0},
  \quad  \|v_{2}(t)\|_{H^1}^2  \lesssim \eps_{0} + ( 1 +\tilde \eps_{1}) (\|v_{1}(t)\|_{H^1_{w}} + \|v_{2}(t) \|_{H^1_{a}} + |c(t) - c_{0}|) .$$
\end{proposition}
 Note that the last estimate does not seem appropriate for the bootstrap. Nevertheless, we shall prove below
 that   the  estimates for  $\|v_{1}(t)\|_{H^1_{w}},$ $  \|v_{2}(t) \|_{H^1_{a}}$ and  $|c(t) - c_{0}|$  are much better  behaved
  in the sense that these quantities  can be  estimated in terms of  $\eps_{0}$ if $\tilde{\eps}_{1}$ is sufficiently small. We could
   use the orbital stability of the solitary wave to get better estimates at this stage.

\begin{proof}[Proof of Proposition \ref{vH1}]
For the KdV type equation \eqref{freemKdV} we have the conservation of the quantities
$$ \int_{\mathbb{R}} |v_{1}|^2 \, dx, \quad  \int_{\mathbb{R}} \left( \frac{1}{2} | \partial_{x} v_{1}|^2 - \frac{v_{1}^4}{4} \right) \, dx .$$
Using these and Sobolev inequalities we easily get
\beq
\label{v1H1}
\|v_{1}(t) \|_{H^1} \lesssim \eps_{0}, \quad \forall \,\, t \in [0, T].
\eeq

To estimate $v_{2}$ we use the conserved quantities for \eqref{mKdVbis}.
The mass conservation
$$ \int_{\mathbb{R}} |u(t,x)|^2 \, dx = \int_{\mathbb{R}} | Q_{c_{0}} (x) + v_{0}(x) |^2 \, dx$$
implies, after expanding $u$ as in \eqref{usplit} and \eqref{mizudec}, that
$$
\int_{\mathbb{R}} |v_1 + v_2 + Q_{c(t)}|^2 \,dx = \int Q_{c_0}^2 \,dx + O(\epsilon_0),
$$
and thus
$$
\int_{\mathbb{R}} | v_{2}|^2\, dy = \int_{\mathbb{R}} (Q_{c_{0}}^2 - Q_{c(t)}^2) \, dy
  - \int_{\mathbb{R}} v_{1}^2\, dy  - 2\int_{\mathbb{R}} Q_{c(t)} v_{1}\, dy
  - 2 \int_{\mathbb{R}} Q_{c(t)} v_{2} \, dy  - 2 \int_{\mathbb{R}} v_1 v_2 \, dy+  O(\eps_{0}).
$$
This yields
$$ \| v_{2}(t) \|_{L^2}^2 \lesssim \eps_{0} + |c(t) - c_{0}| +  \|v_{1}(t) \|_{H^1_{w}} + \|v_{2}\|_{H^1_{a}} , $$
if $\tilde \eps_{1}$ is chosen small enough.

To estimate $\|  \partial_{x} v_{2}(t) \|_{L^2}^2$ one can proceed in a similar way, by using the conservation of the Hamiltonian
for \eqref{mKdVbis}.
\end{proof}

\subsubsection*{Step 4: Estimates of the modulation parameters}

The existence of the modulation parameters is based on the following:
\begin{lemma}\label{lemmodpar}
Let $c_{0} > 0$, $h_{0} \geq 0$.
There exists $\delta >0$ such that for every $w$ satisfying
\begin{align*}
w(t) - Q_{c_{0}}(\cdot - c_{0}t + h_{0}) \in \mathcal{C}^1([0, T_{0}], H^1_{\langle x \rangle_{+}^m})
\end{align*}
for some $m>1/2$, with
\begin{align*}
\sup_{[0,T_{0}]} \|\langle (\cdot + h_{0})_{+}^m \rangle\left( w(t) - Q_{c_{0}}(\cdot - c_{0}t + h_{0}) \right) \|_{H^1}  < \delta ,
\end{align*}
there exists $(h(t), c(t)) \in \mathcal{C}^1([0, T_{0}])$ such that
$$ \int_{\mathbb{R}} \left( w(t,x) - Q_{c(t)}(y) \right) \zeta_{c(t)}^k(y) \,dx= 0, \quad k= 1, \, 2$$
where $ y = x - \int_{0}^t c(s) \, ds + h(t)$.
\end{lemma}
The proof of this lemma is now very  classical and relies on the use of the implicit function Theorem.
We refer to \cite[Proposition 5.1]{Pego-Weinstein1} or \cite[Proposition 3.1]{Mizumachi1} for the proof.

By using Lemma \ref{lemmodpar} for $ w= u$, we get the existence of $c(t)$ and $h(t)$ such that the decompositions \eqref{usplit}, and \eqref{mizudec} with \eqref{contrainte2} hold.
\begin{proposition}\label{propshift}
On $[0, T]$ we have the following  estimates for the modulation parameters:
$$ |\dot h (t) | + |\dot c(t) | \lesssim \langle t \rangle^{-2m} M(t)^2. $$
\end{proposition}

Note that by integrating in time the above estimate, we get that
\beq
\label{estshift}
  |c(t) - c_{0}| + |h(t) - h_{0}| \leq  M(t)^2, \quad \forall \,\, t  \in [0, T] .
  \eeq

\begin{proof}[Proof of Proposition \ref{propshift}]
By using the equation \eqref{eqperturb}, we get by taking the time derivatives of the constraints \eqref{contrainte2} that
the vector $\Gamma(t)= \left(h (t), c(t) \right)^t $ verifies the ODE
\beq
\label{moduleq} A(t) \dot \Gamma(t) =
  -\left( \begin{array}{cc}   (\mathcal{F}(v), \partial_{y} \zeta_{c}^1) \\ (\mathcal{F}(v), \partial_{y} \zeta_{c}^2) \end{array} \right),
\eeq
(using once again $(v,\zeta_c^2) = 0$) with
$$ A(t) =  Id -  \left(
\begin{array}{cc} ( v , \partial_{y} \zeta_{c}^1) & (v ,\partial_{c} \zeta_{c}^1)
\\
(v, \partial_{y} \zeta_{c}^2) & (v, \partial_{c} \zeta_{c}^2) \end{array}
\right) := Id - B(t).
$$
Since
$| B(t)| \lesssim \| v_{2}(t)\|_{L^2_{w}}$, we have that $A(t)$ is invertible for $\tilde \eps_{1}$ sufficiently small,
 with the norm of its inverse smaller than 2.
Moreover, we can  estimate  the right hand side of \eqref{moduleq} by using the localization provided by $\partial_{y} \zeta_{c}^i$.
In particular, we obtain that
$$
| (\mathcal{F}(v), \partial_{y} \zeta_{c}^i)|
  \lesssim ( 1 + \|v\|_{H^1}) \big( \| v_{1}\|_{L^2_{w}}^2 + \| v_{2}\|_{L^2_{a}}^2) \lesssim  \langle t \rangle^{-2m}  M(t)^2,
$$
for $t \in [0, T]$, which gives the desired estimate.
\end{proof}

\subsubsection*{Step 5: Estimates of $v_{1}$}
We shall now use localized virial type estimates in order to  estimate the weighted norms of $v_1$.
\begin{proposition}
\label{propv1}
For every $t \in [0, T]$,  we have the estimates:
\begin{align}
\label{propv1conc1}
\| v_{1}(t) \|_{H^1_{w}} \lesssim  \eps_{0} \langle t \rangle^{-m},
  \qquad
\| \langle y_{+} \rangle^m v_{1} \|_{H^1} \lesssim \eps_{0}.
\end{align}
Moreover, we also have
\begin{align}
\label{propv1conc2}
\int_{0}^t \| v_{1}(s) \|_{H^2_{w'}}^2 \, ds \lesssim \eps_{0}^2.
\end{align}
\end{proposition}

\begin{proof}
We first notice that on $[0, T]$, we have by assumption that $ |c(t) - c_{0}| \leq \tilde \eps_{1}$ and also by using Proposition
\ref{propshift} that $ |\dot h| \lesssim \tilde  \eps_{1}$,
consequently, by  assuming that  $\tilde \eps_{1}$ is sufficiently small, we can always ensure that
\beq
\label{tildecbelow}
c_{0}/2 \leq  \tilde c (t) \leq  2  c_{0}, \quad \forall t \in [0, T].
\eeq

We shall use weights
\beq
\label{defpoids}
\phi_{k}(t,y) %:= \phi_{k}(t,y;\sigma,x_0,\delta)
  := \chi_{k, \delta} ( y + \sigma t + x_{0})
\eeq
with $\sigma$, $0 \leq \sigma  < c_{0}/2$, $x_{0} \in \mathbb{R}$, $\delta$ sufficiently small, and $\chi_{k,\delta}$ is given by
\begin{align*}
\chi_{k, \delta}(y)= \big( A_{k} + (\delta y)^{2} \big)^k \big(1 + \tanh(\delta y) \big)
\end{align*}
We choose $A_k$ sufficiently big, so that the following  inequalities hold:
$$
\chi_{k,\delta} \sim w^2 \langle y \rangle^{2k} , \quad  \chi_{k,\delta}' \geq 0, \quad  \chi_{k,\delta}'' \lesssim \delta  \chi_{k,\delta}', \quad  \chi_{k,\delta}''' \lesssim \delta^2  \chi_{k,\delta}'.
$$
From  \eqref{freemKdV}, we first obtain
$$
\frac{d}{dt } { 1 \over 2} \int_{\mathbb{R}} \phi_k |v_{1}|^2
  + \frac{1}{2} (\tilde c- \sigma) \int \phi_k' |v_{1}|^2 + \frac{3}{2} \int \phi_k'  |\partial_{y} v_{1}|^2
  = \frac{1}{2} \int \phi_k'''  |v_{1}|^2 + \frac{1}{4}  \int \phi_k' |v_{1}|^4.
$$
Next, we observe that $|\phi_k'''| \lesssim \delta^2 \phi_k'$ and that
$ \|v_{1} \|_{L^\infty} \lesssim \|v_{1} \|_{H^1} \lesssim \eps_{0}$ thanks to Proposition \ref{vH1}.
We thus obtain that
\beq
\label{idviriel}
{d \over dt } { 1 \over 2} \int_{\mathbb{R}} \phi_k |v_{1}|^2
  + ( { 1 \over 2 } (\tilde c - \sigma)- C \eps_{0}^2 - C \delta^2) \int \phi_k' |v_{1}|^2  + { 3 \over 2} \int \phi_k'  |\partial_{y} v_{1}|^2 \leq 0.
\eeq

By setting
$$ e_{1}= { 1 \over 2}  |\partial_{y} v_{1}|^2  -   { 1 \over 4 } |v_{1}|^4, \quad d_{1}=  \partial_{y}^2 v_{1}  +  v_{1}^3, $$
we also  get from \eqref{freemKdV} that
$$ \partial_{t} e_{1}  - \tilde c  \partial_{y} e_{1}  -  d_{1} \partial_{y} d_{1}= -\partial_{y} (\partial_{y} d_{1} \partial_{y} v_{1}).$$
Note that this is the infinitesimal conservation law corresponding to the conservation of the Hamiltonian.
By integrating this identity against the weight $\phi_k$, we obtain after some integration by parts that
\begin{align*}
{d \over dt} \int \phi_k e_{1} +  ( \tilde{c} - \sigma) \int  \phi_k' e_{1} +  { 1\over 2} \int \phi_k' |d_{1}|^2
  =  \int \phi_k' \partial_{y} d_{1} \partial_{y} v_{1}.
\end{align*}
To control the last integral we can integrate by parts and use Proposition \ref{vH1}  and  $|\phi_k''| \lesssim \delta \phi_k'$ to get
$$ \int \phi_k' \partial_{y} d_{1} \partial_{y} v_{1}  + \int \phi_k' |d_{1}|^2
  \lesssim  \delta  \int \phi_k'( |d_{1}|^2  + | \partial_{y}v_{1}|^2)+  \eps_{0} \int \phi_k' (|d_1|^2 + |v_1|^2).
$$
We thus get that
$$ {d \over dt} \int \phi_k e_{1} +  ( \tilde c - \sigma) \int  \phi_k' e_{1} +  ({ 3\over 2} - C \delta  -  C\eps_{0}) \int \phi_k' |d_{1}|^2
  \lesssim  \delta  \int \phi_k' | \partial_{y} v_{1}|^2  +  \eps_{0} \int \phi_k' |v_{1}|^2.$$

By combining the last identity and \eqref{idviriel}, we thus obtain that
\begin{align}
\label{virielfinal}
\begin{split}
{d \over dt} \int \phi_k (e_{1}   + { 1 \over 2}  |v_{1}|^2)
  & + ( \tilde c  - \sigma - C \eps_{0}- C \delta) \int  \phi_k' (e_{1}  + {1 \over 2} |v_{1}|^2)
  \\
  & + ({3 \over 2} - C \delta - C \eps_{0})\int \phi_k' (|d_{1}|^2 + |\partial_{y} v_{1}|^2) \leq 0 .
\end{split}
\end{align}
Note that for $\eps_{0}$ sufficiently small, $e_{1} + { 1 \over 2}  |v_{1}|^2$ and $|d_{1}|^2 + | \partial_{y} v_{1}|^2 + |v_1|^2$ are positive quantities that control pointwise $| \partial_{y} v_{1}|^2 + |v_{1}|^2 $ and $| \partial_{y}^2 v_{1}|^2 + | \partial_{y} v_{1}|^2 + |v_{1}|^2$ respectively.

By using this identity with $k=0$,  $\sigma = 0$,  $x_{0}= 0$, we obtain  after integration in time  that
\beq
\label{viriel1}
\int_{0}^t  \int_{\mathbb{R}}(w')^2 (| \partial_{yy} v_{1}|^2 + |\partial_{y} v_{1}|^2 +   |v_{1}|^2 ) \, dy  dt \lesssim \|v_{0}\|_{H^1}^2.
\eeq

By taking $k=0$,  $\sigma >0$, small and $x_{0}=  - \sigma \tau$, we also get by integrating between $0$ and $\tau$  that for every $\tau>0$,
$$  \| v_{1}(\tau) \|_{H^1_{w}}^2 \lesssim  \int_{\mathbb{R}} \phi_0( y - \sigma \tau)( |v_{0}|^2  + | \partial_{y} v_{0}|^2 )\, dy.$$
Since  $ \phi_0( y - \sigma \tau)/ \langle y_{+} \rangle^{2m} \lesssim  1 / \langle \tau \rangle^{2m}$, we also obtain that
\beq
\label{viriel2}
\| v_{1}(\tau) \|_{H^1_{w}}  \lesssim  { 1 \over  \langle \tau \rangle^m} \| \langle y_{+} \rangle^m  v_{0}\|_{H^1}, \quad \forall \,\,\tau \in [0, T].
\eeq

Finally, by using \eqref{virielfinal} with $ \sigma = 0$ and $x_{0}=0$ but for  $k=m$, we get that
$$ \int_{\mathbb{R}}  \phi_{m}  \left(| \partial_{y} v_{1}|^2 + | v_{1}|^2 \right)  (t) \lesssim \| \langle y_{+} \rangle^m v_{0}\|_{H^1}^2, \quad \forall \,\, t \in [0, T].$$
Since  $\phi_{m}$ behaves like $y^{2m}$ for $y \geq 0$, we get, using also Proposition \ref{vH1}, that
$$ \| \langle y_{+} \rangle^m  v_{1}(t)\|_{H^1} \lesssim \eps_{0}.$$
This ends the proof of the proposition.
\end{proof}

\subsubsection*{Step 6: Estimate of $v_{2}$ }
We now estimate $v_{2}$ mainly using the semi-group estimates of Theorem \ref{theoPego}.

\begin{proposition}
\label{propv2}
For all $t \in [0,T]$ we have the estimates
\begin{align}
\label{propv2conc}
\langle t \rangle^m \| v_{2} (t)\|_{H^1_{a}} \lesssim \eps_{0},
  \qquad
  \int_{0}^t \| v_{2}(\tau) \|_{H^2_{a}}^2 \lesssim \eps_{0}^2.
\end{align}
\end{proposition}
\begin{proof}
We can first write the equation \eqref{KdVexp} for $v_{2}$ as
$$
\partial_{t} v_{2} + \mathcal{L}_{c_{0}} v_{2} = \partial_{y} \mathcal{N}(v) + e_{Q}  + \partial_{y} \mathfrak{e}_{Q},
$$
where $e_{Q}$ and $\mathcal{N}$ are defined in \eqref{eQ}, \eqref{defNdev} and $\mathfrak{e}_{Q}$ is given by
\beq
\label{tildeeQ}
\mathfrak{e}_Q =  - 3  (Q_{c(t)}^2 - Q_{c_{0}}^2) v_{2} + (\tilde{c} - c_0) v_2.
\eeq

%\comment{F: should it be $\mathfrak{e}_Q = \dot c \xi_{c(t)}^2(y) + (\dot h  - c)\xi_{c(t)}^1(y) + (\tilde{c} - c_0) v_2$?
%P: I find $\mathfrak{e}_Q = - 3  (Q_{c(t)}^2 - Q_{c_{0}}^2) v_{2} + (\widetilde{c}-c_0)$}

By using the semi-group $S_{c_{0}}$ we get   that $v_{2}$ is given by the following Duhamel formula
\beq
\label{duhamelv2}
v_{2}(t) = \int_{0}^t S_{c_{0}}(t- \tau) \left(  \partial_{y} \mathcal{N}(v) + e_{Q}  + \partial_{y} \mathfrak e_{Q}\right)(\tau)\, d\tau.
\eeq

We shall first estimate $\mathcal P_{c_{0}} v(t)$.
By using the definition of $\mathcal P_{c}$ and the fact that $v_{2}$ satisfies the constraint \eqref{contrainte2}, we get that
$$  \| \mathcal{P}_{c_{0}} v(t)\|_{H^2_{a}} \lesssim   \|v_{1}\|_{L^2_{w}}
  + |c(t) -c_{0}|   \| v_{2}\|_{L^2_{a}}.$$
Since
$$  \| \mathcal{P}_{c_{0}} v_{2}(t)\|_{H^2_{a}}  \leq \| \mathcal{P}_{c_{0}} v_{1}(t)\|_{H^2_{a}}
  + \| \mathcal{P}_{c_{0}} v(t)\|_{H^2_{a}}
\lesssim  \|v_{1}\|_{L^2_{w}} +   \| \mathcal{P}_{c_{0}} v(t)\|_{H^2_{a}},
$$
by using  Proposition \ref{propv1} and \eqref{estshift}, we get
\beq
\label{estPv2}
\langle t \rangle^m \|  \mathcal{P}_{c_{0}} v_{2}(t)\|_{H^2_{a}} \lesssim  \eps_{0} +  M(t)^3.
\eeq
 This also yields (since $m>1/2$)
 \beq
 \label{estPv2bis}
 \left( \int_{0}^t  \|  \mathcal{P}_{c_{0}} v_{2}(t)\|_{H^2_{a}} ^2 \right)^{1 \over 2}
  %\lesssim \left(\int_{0}^t \| v_{1}\|_{L^2_{w'}}^2 \right)^{1 \over 2}
  %+ \tilde{\eps_{1}}^2  \left( \int_{0}^t { 1 \over \langle \tau \rangle^{2m}}\right)^{1 \over 2}
  \lesssim \eps_{0} +  M(t)^3 .
\eeq
%In a similar way, by using Proposition \ref{propv1}, we obtain that
%and, therefore, %since $m>1/2$,
%\beq
%\label{estPv2bis}
%\left( \int_{0}^t  \|  \mathcal{P}_{c_{0}} v_{2}(t)\|_{H^2_{a}} ^2 \right)^{1 \over 2}
%  %\lesssim \left(\int_{0}^t \| v_{1}\|_{L^2_{w'}}^2 \right)^{1 \over 2}
%  %+ \tilde{\eps_{1}}^2  \left( \int_{0}^t { 1 \over \langle \tau \rangle^{2m}}\right)^{1 \over 2}
%  \lesssim \eps_{0} + \tilde \eps_{1}^2 .
%\eeq

Next, we apply $\mathcal{Q}_{c_{0}}$ to \eqref{duhamelv2}:
$$
\mathcal Q_{c_{0}}v_{2}(t) = \int_{0}^t S_{c_{0}}(t- \tau) \mathcal Q_{c_{0}}
  \left(  \partial_{y} \mathcal{N}(v) + e_{Q}  + \partial_{y} \mathfrak e_{Q}\right)(\tau)\, d\tau.
$$
Thanks to Theorem \ref{theoPego}, we obtain
\beq
\label{duhamel1}
\| \mathcal Q_{c_{0}} v_{2}(t) \|_{H^s_{a}} \lesssim \int_{0}^t  e^{-b(t- \tau)}
  \max\big( 1, (t - \tau)^{-1/2} \big)(\|e_{Q}\|_{H^2_{a}} + \|  \mathfrak{e}_{Q}\|_{H^s_{a}} +  \|\mathcal N (v)\|_{H^s_{a}})\, d\tau, \quad s= 1, \, 2.
\eeq
To estimate the right hand side above, we recall the definition \eqref{eQ} and observe that
$$
\|e_{Q}(t)\|_{H^2_{a}} \lesssim | \dot c | + | \dot h |.
$$
Therefore, using Proposition \ref{propshift}, we obtain that on $[0,T]$
\beq
\label{dotcdoth}
\|e_{Q}(t)\|_{H^2_{a}} \lesssim   \langle t \rangle^{-2m} M(t)^2.
\eeq
Next, we observe that
\beq
\label{esttildeeQ}
  \| \mathfrak e_{Q} (t) \|_{H^1_{a}} \lesssim (| c(t) - c_{0}|+|\dot h(t)|)  \|v_{2}(t) \|_{H^1_{a}}
    \lesssim  {\langle t \rangle}^{-m} M(t)^3,
    \eeq
and in a similar way, we obtain
\beq
\label{esttildeeQbis}
    \int_{0}^t \| \mathfrak e_{Q} (t) \|_{H^2_{a}}^2 \lesssim \tilde \eps_{1}^2 \int_{0}^t \|v_{2} \|_{H^2_{a}}^2.
\eeq
To estimate $\mathcal{N}(v)$, we recall its definition in \eqref{defNdev}, and write
\beq
\label{Nbisforme}
\mathcal{N}(v)= - \left(  3 Q_{c(t)}^2 v_{1} + 3Q_{c(t)} v_{1}^2 + 3 Q_{c(t)} v_2^2 + 6Q_{c(t)} v_{1} v_{2}  + 3 v_{1}^2 v_{2}+ 3v_{1} v_{2}^2 + v_{2}^3 \right).
\eeq
Using the localization provided by $Q_{c}$ we get that, on $[0, T]$,
\begin{align*}
\|\mathcal N(v)(t) \|_{H^1_{a}} \lesssim  ( 1 +  \|v_{1}\|_{H^1}) \|v_{1} \|_{H^1_{w}}
  + (\|v_{1} \|_{H^1} + \|v_{2}\|_{H^1})( 1+ \|v_{1}\|_{H^1} +  \|v_{2}\|_{H^1}) \|v_{2}\|_{H^1_{a}}.
  %+ ( 1+ \|v_{1}\|_{H^1} +  \|v_{2}\|_{H^1})^2 \|v_{2}\|_{H^1_{a}} .
\end{align*}
Therefore, by using Proposition \ref{propv1} and Proposition \ref{vH1}, we get that
\beq
\label{Nv1}
\|\mathcal N(v)(t) \|_{H^1_{a}} \lesssim  \eps_{0} {\langle t \rangle}^{-m}
  + \tilde{\eps}_{1} {\langle t \rangle}^{-m} M(t).
\eeq
In a similar way, we obtain
    \begin{equation}
      \label{NvH2}
     \int_{0}^t \| \mathcal N(v) \|_{H^2_{a}}^2 \lesssim    (1 + \tilde\eps_{1}) \int_{0}^t \| v_{1}\|_{H^2_{w}}^2 +  \tilde{\eps}_{1}^2 \int_{0}^t \|v_{2}\|_{H^2_{a}}^2.
    \end{equation}
Putting together \eqref{estPv2}, \eqref{duhamel1}, \eqref{dotcdoth}, \eqref{esttildeeQ}  and \eqref{Nv1},
we see that for all $t \in [0, T]$
$$
\langle t \rangle^m \| v_{2}(t) \|_{H^1_{a}} \lesssim \eps_{0} + \tilde \epsilon_1 M(t)  + (\epsilon_0 +   \tilde{\eps}_{1}M(t))
\int_0^t \langle t \rangle^m e^{- b(t- \tau)} \max \big( 1 , (t - \tau)^{-1/2} \big)  {1 \over  \langle \tau \rangle^m} \, d \tau.
$$
Since
$$  \sup_{t\geq 0}  \int_{0}^t   \langle t \rangle^m e^{- b(t- \tau)} \max \big( 1 , (t - \tau)^{-1/2} \big)
  {1 \over  \langle \tau \rangle^m} \, d \tau <+ \infty, $$
   we obtain, by using again  Proposition \ref{vH1}, that
$$  \langle t \rangle^m \| v_{2}(t) \|_{H^1_{a}}  \lesssim  \eps_{0} + \tilde{\eps}_{1} \sup_{[0, t]} ( \langle s \rangle^m   \| v_{2}(s) \|_{H^1_{a}}).$$
 Taking $\tilde \eps_{1} $ sufficiently small, we get
 \beq
 \label{firstpartoftheproof}
 \langle t \rangle^m \| v_{2}(t) \|_{H^1_{a}}  \lesssim  \eps_{0} , \quad \forall t \in [0, T].
\eeq
 Consequently,  the first part of \eqref{propv2conc} is proven.

To get the second part, we use Young's inequality and  \eqref{estPv2bis},  \eqref{duhamel1} for $s=2$, \eqref{dotcdoth}, \eqref{esttildeeQ},
 \eqref{NvH2} to obtain
$$
\left( \int_{0}^t \| v_{2}(\tau) \|_{H^2_{a}}^2 \, d\tau \right)^{1/2}  \lesssim \eps_{0}  +  M(t) + \Big( \int_{0}^t \| v_{1}\|_{H^2_{w}}^2\, d\tau \Big)^{1 \over 2 }  +
 \tilde{\eps}_{1}\left( \int_{0}^t \| v_{2}(\tau) \|_{H^2_{a}}^2 \, d\tau \right)^{1/2} .
$$
By using Proposition \ref{propv1} and \eqref{firstpartoftheproof}, this yields
$$ \left( \int_{0}^t \| v_{2}(\tau) \|_{H^2_{a}}^2 \, d\tau \right)^{1/2}  \lesssim \eps_{0}  +  \tilde{\eps}_{1}\left( \int_{0}^t \| v_{2}(\tau) \|_{H^2_{a}}^2 \, d\tau \right)^{1/2},$$
 and we conclude again the proof by  choosing $\tilde{\eps}_{1}$ sufficiently small.
\end{proof}

\subsubsection*{Step 7: Conclusion}
 By combining Propositions \ref{propv1} and \ref{propv2}, we have already proven that
  $ M(t) \lesssim \eps_{0}$  on $[0, T]$. From \eqref{estshift}, we also obtain
 that   $ |c(t) - c_{0}| + |h(t) - h_{0}| \leq \eps_{0}$ (actually we even have $\eps_{0}^2$).
  Finally from Proposition \ref{vH1}, we get
  $ \|v(t)\|_{H^1} \lesssim \eps_{0}^{1 \over 2}$ on $[0, T]$.
  Since
$$  \langle t \rangle^m \|v(t) \|_{H^1_{w}} + \| \langle y_{+} \rangle^m v(t)\|_{H^1}
  \lesssim  \langle t \rangle^m\| v_{1}\|_{H^1_{w}} + \langle t \rangle^m\| v_{2}\|_{H^1_{a}}+ \| \langle y_{+} \rangle^m v_{1}\|_{H^1} + \|v_{2} \|_{H^1_{a}} ,$$
 we obtain, by using again  Proposition \ref{propv1}, that
 $$ N(t) \lesssim \eps_{0}^{1 \over 2}, \quad \forall \,\, t \in [0, T].$$
 By taking $\eps_1 \gg \eps_{0}^{1/2}$ and sufficiently small,
we see, by a standard bootstrap argument, that the estimate \eqref{hypbootstrapsol} holds true for all times.

Moreover, from Proposition \ref{propshift}, we have that
$$ |\dot h | + |\dot c| \lesssim \tilde{\eps}_{1}^2\langle t \rangle^{-2m} ,$$
and, since $m>1/2$, we  deduce that there exists $c_{+}$ and $h_{+}$ such that
$$
| c(t) - c_{+}| + |h(t) - h_{+}| \lesssim \tilde{\eps}_{1}^2 \langle t \rangle^{-(2m- 1)}.
$$
This gives \eqref{estsoliton}.
Finally, note that the estimate \eqref{estsoliton2} follows from Proposition \ref{propv1} and  Proposition \ref{propv2}.
$\hfill \Box$

%%%%%%%%%%%%%%%%%%%%%%%%%%%%%%%%%%%%%%%%%%%%%%
%%%%%%%%%%%%%%%%%%%%%%%%%%%%%%%%%%%%%%%%%%%%%%
%%%%%%%%%%%%%%%%%%%%%%%%%%%%%%%%%%%%%%%%%%%%%%
%%%%%%%%%%%%%%%%%%%%%%%%%%%%%%%%%%%%%%%%%%%%%%

\bigskip
\subsection{Proof of Theorem \ref{maintheo2}}\label{secsoliton2}
We now show how to obtain Theorem \ref{maintheo2} from Theorem \ref{theosoliton2} and the first part of the paper.
%\vskip5pt
%Frederic: Here we shall  again assume that the initial data is taken at $t=-1$ ?
%\vskip5pt
We start again  from the decompositions \eqref{usplit} and \eqref{mizudec}, so that $v_{1}(t,y)$, $v_{2}(t,y)$, $c$ and $h$
satisfy the estimates in the  proof of Theorem \ref{theosoliton2}. We  thus write
\beq
\label{udec2}
u(t,x) = Q_{c(t)}(y) +  \tilde v(t,x), \quad  \tilde v(t,x)= \tilde v_{1}(t,x)+ \tilde{v_{2}}(t,x), \quad \tilde v_{i}(t,x)= v_{i}(t,y)
\eeq
where again $ y= y(t,x)= x - \int_{0}^t c  + h(t)$.
By using the estimates in Theorem \ref{theosoliton2}, we already have that
\begin{multline}
\label{estdebase}
\|\langle x_{+} \rangle^m v(t) \|_{H^1} +  \langle t \rangle^m ( \|v_{1}(t) \|_{H^1_{w}} + \|v_{2}(t) \|_{H^1_{a}}) +
  \langle t \rangle^{2 m}(|\dot{c}(t)|  +  | \dot{h}(t)| )
\\
    + \langle t \rangle^{2m-1} \left( |c(t) - c_{+}|
     + |h(t) - h_{+}| \right) \lesssim \tilde{\eps}_{1} , \quad \forall \,\, t \geq 0 .
\end{multline}
We now use the approach of the first part of the paper to estimate $ \tilde v= \tilde v_{1}+ \tilde v_{2}$ behind the solitary wave.

\subsubsection*{Step 1: Estimates for $\tilde{v}_{1}$}

By definition, since $v_{1}(t,y)$ solves \eqref{freemKdV}, we get that $\tilde{v}_{1}(t,x)$ solves the mKdV equation
$$ \partial_{t} \tilde v_{1} + \partial_{x}^3 \tilde v_{1} + \partial_{x} \tilde v_{1}^3 = 0, \quad (\tilde v_{1})_{/t=0}= v_{0}(x).$$
Consequently, $\tilde{v}_{1}$ verifies the estimates of Theorem \ref{maintheo1}.
In particular if we denote $f_{1} :=  e^{t \partial_{x}^3} \tilde v_{1}$ we have
\begin{equation}
\label{estv1fin}
{\| \tilde v_{1} \|}_{X} := \sup_{t} \big( \langle t \rangle^{- \delta} \|  x f_{1} \|_{H^1}
  + \langle t \rangle^{{\alpha \over 3} - { 1 \over 6}} \| | \partial_{x} |^\alpha  x f_{1} \|_{L^2}
  + \| \what{f_{1}} \|_{L^\infty} \big) \lesssim \eps_0 ,
\end{equation}
see the definition of the $X$-norm in \eqref{apriori}.
In particular, this also gives the linear estimate \eqref{lemlinear21} and \eqref{lemlinear22}, the bilinear estimates \eqref{bilinear}
and the trilinear estimate \eqref{oriole} for $\tilde{v}_{1}$:
\begin{equation}
 \label{estv1fin2}
\begin{split}
& \langle t \rangle^{1/3+\beta/3-1/(3p)}  {\| |\partial_x|^\beta \tilde v_{1} \|}_{L^p}  \lesssim \epsilon_0,  
  \qquad \mbox{for} \quad 0\leq \beta <1/2, \quad p(1/4-\beta/2)>1 ,
\\
& \langle t \rangle {\|\tilde v_{1} \partial_{x} \tilde v_{1} \|}_{L^\infty} \lesssim \epsilon_0^2,
\\
& \langle t \rangle^{ {5/6}+{\alpha/3}} {\| | \partial_{x}|^\alpha \tilde v_{1}^3 \|}_{L^2} \lesssim {\eps}_{0}^3,
  \qquad \mbox{for} \quad 0\leq \alpha<1/2.
\end{split}
\end{equation}
Thanks to \eqref{boundISu} we also have\footnote{Note that the second part of this estimate,
  was not explicitly written down, but it is a direct consequence of \eqref{SmKdV} and Gronwall's inequality.}
\beq \label{estvefin3}
\| I S \tilde v_{1}\|_{L^2} + \| S \tilde v_{1}\|_{L^2}\lesssim \eps_{0} t^{ C \epsilon_0^2}.
\eeq

For later use, we improve these latter estimates in front of the solitary wave:
\begin{lemma} \label{lemtildev1}
Let us set $ H_{1}(t,y) =( I S \tilde v_{1}) (t, x)$, again with $ y= x - \int_{0}^t c(s) \, ds + h(t)$.
Then we have the estimates
\begin{equation*}
 \| H_{1}(t) \|_{L^2_{w}} \lesssim { \eps_{0} \over
  \langle t \rangle^{m-1 - { 1 \over 2} C \epsilon_0^2}}
  \, , \qquad
  \int_{0}^t \| H_{1}(\tau) \|_{H^1_{w'}}^2 \, d\tau \lesssim \eps_{0}^2.
\end{equation*}
\end{lemma}

\begin{proof}
We observe that since $v_1$ solves the free equation \eqref{freemKdV}, and $S$ commutes with the equation as in \eqref{SmKdV},
then
$$
\partial_{t} H_{1} - \tilde{c} \partial_{y} H_{1}  + \partial_{y}^3 H_{1} + 3 v_{1}^2 \partial_{y} H_{1} = 0 .
$$
Then we can use a virial type computation similar to \eqref{idviriel}, with the same $\phi_{k}$ defined in \eqref{defpoids}, to find
\begin{align}
\label{virielI}
\begin{split}
{ d\over dt } { 1 \over 2 } \int_{\mathbb{R}} \phi_{k} H_{1}^2 \, dy  +  { 1 \over 2 }(\tilde c- \sigma - C \delta^2 - C \epsilon_0^2)
  \int_{\mathbb{R}}\phi_{k}'  H_{1}^2 \, dy 
  + { 3 \over 2} \int_{\mathbb{R}} \phi_{k}'|\partial_{y} H_{1} |^2 \, dy 
\\ = 3 \int_{\mathbb{R}} |v_{1} \partial_{y}v_{1} | \phi_{k} H_{1}^2 \, dy 
\lesssim \epsilon_0^2  { 1 \over \langle t \rangle} \int_{\mathbb{R}} \phi_{k} H_{1}^2 \, dy ,
\end{split}
\end{align}
where we have used \eqref{estv1fin2} to obtain the above inequality.
Note that for $\eps_0$ sufficiently small, we can ensure that
$\tilde c- \sigma - C \delta^2 - C  \eps_{0}^2 \geq  { c_{0} \over 2}$.
Consequently, integrating in time we get
$$ \int_{\mathbb{R}} \phi_{k}(t,y) |H_{1}(t,y)|^2 \, dy  \lesssim
  \langle t \rangle^{C \eps_{0}^2} \int_{\mathbb{R}} \phi_{k}(0) |H_1(0,y)|^2 \, dy.$$
Moreover, by observing that $H_{1}(0,y) = y v_{0}$
and taking the parameters in the weight so that $\sigma>0$,  $x_{0}= - \sigma \tau$,  and $k \leq m-1$,
we get in particular that for every $\tau \geq 0$,
$$ \| \langle x_{+} \rangle^{k} H_{1}(\tau) \|_{L^2_{w}}^2 \lesssim { \eps_{0}^2 \langle \tau \rangle^{ C\eps_{0}^2}  \over
  \langle \tau \rangle^{2(m-1)}}.$$
This proves the first part of the estimates by taking $k=0$.

Next, by using \eqref{virielI} with $\sigma = 0$,  $x_{0}= 0$  and $ k= 0$ in the weight, and integrating in time we get, using that $w'^2 \lesssim \chi_{0,\delta}^2$,
$$ \int_{0}^t \| H_{1}\|_{H^1_{w'}}^2 \, ds
  \lesssim \eps_{0}^2
  + \eps_{0}^4  \int_{0}^t  { 1 \over \langle s \rangle^{ 1+  2(m - 1)- C{\eps}_{0}^2}} \, ds
  \lesssim  \eps_{0}^2 + \eps_{0}^4 , $$
for $\eps_{0}$ sufficiently small,since we assumed $m > 3/2$. \end{proof}

\subsubsection*{Step 2: Weighted estimates for $I S \tilde v_{2}$}
In this section we shall use in a crucial way that
\beq \label{miracle}
\partial_{c} Q_{c}= { 1 \over 2 c}  \partial_{y}(yQ_{c}).
\eeq
\begin{lemma}
\label{lemtildev2}
Set $ H_{2}(t,y) =( I S \tilde v_{2}) (t, x)$
again with $ y= x - \int_{0}^t c + h.$ Then
$$
\| H_2 \|_{L^2_a}^2 + \int_0^\infty \| H_2 \|_{H^1_a}^2 \,d\tau \lesssim \tilde \epsilon_1^2.
$$
\end{lemma}

\begin{proof}
We shall first estimate $ \partial_{y}  H_{2}  (t, y)=  (S \tilde  v_{2}) (t,y)$. Commuting the vector field $S$ with the equation $\partial_t v + \partial_x^3 v = F + \partial_x G$ gives $\partial_t S v + \partial_x^3 Sv = (S+3) F + \partial_x (S+2) G$. Applying this identity to \eqref{KdVexp}, we get that  $\partial_{y} H_{2}$ solves
\begin{multline}
\label{dyH2}
\partial_{t} \partial_{y} H_{2} + \mathcal{L}_{c_{+}} \partial_{y}H_{2}=
  \partial_{y}\big(S ( \tilde{\mathcal{ N}}(v))(t,y) + \mathcal{ N}(v)(t,y) + S \tilde e_{Q}(t,y) + 3 e_{Q}(t,y) \big)
\\
-  \partial_{y} \big( (S \tilde Q_{c}^2)  v_{2}\big) - \partial_{y} (Q_{c}^2 v_{2}) +  \partial_{y} \mathfrak e_{Q}(t,y)
 := F(t,y),
\end{multline}
where we have set
\begin{align}
\nonumber
& \tilde Q_{c}(t,x) = Q_{c}(t,y), \qquad \tilde e_{Q}(t,x)= e_{Q}(t,y),
\\
\label{tildeN}
& \tilde{\mathcal{N}}(v)(t,x) =
  \big( - ( \tilde Q_{c(t)} + \tilde v_{1} + \tilde v_{2})^3
  + \tilde Q_{c(t)}^3  + \tilde  v_{1}^3  +  3 \tilde Q_{c(t)}^2  \tilde v_{2} \big)(t,x),
\end{align}
and
$$  \mathfrak e_{Q}(t,y) =   ( \tilde c - c_{+}) \partial_{y} H_{2}  - 3 ( (Q_{c}^2 - Q_{c_{+}}^2)\partial_{y} H_{2}).$$

Let us first estimate
$\mathcal{P}_{c_{+}} \partial_{y} H_{2}= \mathcal P_{c_{+}} (S \tilde v_{2})(t,y)$.
By using the equation \eqref{KdVexp} to compute $\partial_{t} \tilde v_{2}$,
and by putting the space derivatives on the functions $\zeta_{c_{+}}^i$, $\xi_{c_{+}}^i$, we obtain that
\begin{align*}
\| \mathcal{P}_{c_{+}} \partial_{y} H_{2} \|_{H^1_{a}} \lesssim \langle t \rangle \left( \| v_{2}\|_{L^2_{a}}
  + \|v_{1}\|_{L^2_{w}} + |\dot{c}| + |\dot{h}| \right)
\end{align*}
and hence by using \eqref{estdebase}, we obtain
\beq
\label{PcdyH}
\|   \mathcal{P}_{c_{+}} \partial_{y} H_{2} \|_{H^1_{a}} \lesssim { \tilde \eps_{1} \over \langle t \rangle^{m-1} }, \qquad
\|   \mathcal{P}_{c_{+}} \partial_{y} H_{2} \|_{L^2_{t}H^1_{a}} \lesssim \tilde{\eps_{1}}.
\eeq
Note that the second estimate comes from the fact that $m>3/2$.

We can now estimate $\mathcal{Q}_{c_{+}} \partial_{y}H$. Applying $\mathcal{Q}_{c_{+}}$ to the integral formulation of the equation
$ \partial_{y} H_{2}(t)= \int_{0}^t   S_{c_{+}}(t-\tau) F(\tau, y)  \, d \tau$ gives
$$
\mathcal{Q}_{c_{+}}  \partial_{y} H_{2}(t)= \int_{0}^t  S_{c_{+}}(t-\tau) \mathcal{Q}_{c_{+}} F(\tau, y)\, d\tau.
$$
Using the smoothing estimates in Theorem \ref{theoPego}, the estimates \eqref{estdebase}, and Proposition \ref{propshift}, we get
\begin{align}
\label{boundQcH2}
\begin{split}
\| \mathcal{Q}_{c_{+}}  \partial_{y} H_{2}(t) \|_{L^2_{a}} \lesssim \int_{0}^t  \Big( e^{- b( t- \tau)}
  \max \big( 1, (t- \tau)^{-1/2} \big)
\big(  {\| S \tilde{ \mathcal{N}} (v)\|}_{L^2_{a}} + {\| \mathcal{N}( v) \|}_{L^2_{a}}
\\
  +  { \tilde \eps_{1} \over \langle \tau \rangle^{2m-1}} {\| \partial_{y} H_{2}\|}_{L^2_{a}} + { \tilde \eps_{1} \over  \langle \tau \rangle^{m}} \big)
  + e^{- b(t - \tau)} \langle \tau \rangle(  | \ddot{c}(\tau) | + | \ddot{h}(\tau)|  ) \Big)\, d\tau.
\end{split}
\end{align}

Next, we claim that from the definitions of the nonlinearities in \eqref{Nbisforme} and \eqref{tildeN},
and using \eqref{estdebase} and \eqref{estvefin3}, one has
\begin{align}
\label{StildeN1}
& \|  \mathcal{N}(v) \|_{L^2_{a}}  \lesssim \tilde{\eps}_{1} \langle t \rangle^{-m},
\\
\label{StildeN2}
& \|S \tilde{ \mathcal{N}} (v)\|_{L^2_{a}} \lesssim  \tilde \eps_{1} \langle t \rangle^{-(m-1)} + \| \partial_{y}H_{1}\|_{L^2_{w'}}
  +  \tilde \eps_{1} \| \partial_{y} H_{2}\|_{L^2_{a}}.
\end{align}
The first estimate is a direct consequence of \eqref{Nv1}.
Most of the estimates involved in proving \eqref{StildeN2} are straightforward, so we only
give details about one of the most complicated terms:
  $$ \| v_{1} (S \tilde v_{1})(t,y) v_{2} \|_{L^2_{a}} \lesssim  \|v_{1}\|_{L^\infty} \| S \tilde v_{1}\|_{L^2} \|v_{2}\|_{H^1_{a}}
   \lesssim  {\tilde{\eps_{1}}^3 \over  \langle t \rangle^{m - C \eps_{0}^2}}.$$
 We also have to estimate $\ddot c$ and $\ddot h$.  By differentiating in time the equation \eqref{moduleq},  using the equations for $v_{1}$ and
  $v_{2}$ to express $\partial_{t} v_{1}$ and $\partial_{t} v_{2}$ and always putting the space derivatives on
   $\zeta_{c}^{i}$ in the scalar products using integration by parts, we obtain by using \eqref{estdebase} that
   \beq
  \nonumber | \ddot  h (t) | + |\ddot c (t) | \lesssim  \tilde  {{\eps}_{1} \over \langle t \rangle^m} + { \tilde \eps_{1} \over \langle t \rangle^m} \| \partial_{yy}v \|_{L^2_{w'}} .
   \eeq
   Note that the last term comes from the estimate of  cubic nonlinear term that yields, after integration by parts, terms of the form  $\int_{\mathbb{R}} v \partial_{v} v \partial_{yy} v \partial_{x} \zeta_{c}^{i}\, dx$, for $i=1, \, 2$.
    Consequently, by using \eqref{estsoliton2}, we obtain that
\beq
 \label{h"c"}  \Big( \int_{0}^t   \langle \tau \rangle^2 ( | \ddot  h (\tau) |^2 + |\ddot c (\tau) |^2 ) \, d\tau \Big)^{ 1 \over 2} \lesssim \tilde{\eps}_{1}.
 \eeq
% \comment{PG $\to$ FR: For the above bound, I do not see how to bound the term coming from the time derivative of $\mathcal{F}(v)$, say for instance the one containing $v^2 \partial_x^3 v$. It seems to me we need to put back the bound in $L^2_t H^2_{w'}$... What do you think?
%
%Fred: Yes and also using the $H^2$ estimate, we only get  an integrated  in time estimate because of the worse term you pointed out }

By plugging the above estimates into \eqref{boundQcH2}, we thus get  that
\begin{multline*}
\| \mathcal{Q}_{c_{+}}  \partial_{y} H_{2}(t) \|_{L^2_{a}} \lesssim  \int_{0}^t  e^{- b(t - \tau)} \langle \tau \rangle(  | \ddot{c}(\tau) | + | \ddot{h}(\tau)|) \, d\tau \\  +  \int_{0}^t  \Big( e^{- b( t- \tau)} \max( 1, (t- \tau)^{-1/2} )
  \Big(  {\tilde \eps_{1} \over \langle \tau \rangle^{m-1}} + \| \partial_{y}H_{1}\|_{L^2_{w'}}
  +  \tilde \eps_{1} \| \partial_{y} H_{2}\|_{L^2_{a}} \Big) \, d\tau.
\end{multline*}
     By using Young's inequality,  \eqref{PcdyH} and Lemma \ref{lemtildev1}, we thus get that
     $$ \| \partial_{y} H_{2}\|_{L^2_{t} L^2_{a}} \lesssim \tilde \eps_{1} +  \eps_{0} + \tilde \eps_{1} \| \partial_{y} H_{2} \|_{L^2_{t}L^2_{a}}$$
      and hence for $\tilde \eps_{1}$ sufficiently small that
     \beq
     \label{dyH2pf}
       \| \partial_{y} H_{2}\|_{L^2_{t} L^2_{a}} \lesssim \tilde \eps_{1}.
       \eeq

It remains to estimate $\|H_{2}\|_{L^2_{a}}$.
Since $H_{2}= \int_{- \infty}^y \partial_{y} H_{2}$, integrating in $y$ \eqref{dyH2}, we get
\begin{align*}
\partial_{t} H_{2} -  \tilde c \partial_{y} H_{2} + Q_{c}^2 \partial_{y}H_{2} + \partial_{y}^3 H_{2} & =
  S ( \tilde{\mathcal{ N}}(v))(t,y) +2 \mathcal N(v)(t,y)  \\
& \quad -  \int_{- \infty}^{y} [e_{Q}(t,y^\prime) +3 (S \tilde e_{Q})(t,y^\prime)] \, dy-  \big( S \tilde Q_{c}^2  v_{2}\big) - Q_c^2 v_2 \\
& = G(t,y).
\end{align*}
By a weighted energy estimate, we get that for some $b>0$, we have
$$ { d \over dt } { 1 \over 2} \int_{\mathbb{R}} e^{a y } | H_{2}|^2 \, dy
  + b \| H_{2}\|^2_{H^1_{a}} \lesssim \|G\|_{L^2_{a}} \| H_{2}\|_{L^2_{a}} + \| \partial_{y}H_{2}\|_{L^2_{a}} \|H_{2}\|_{L^2_{a}}.
$$
Next, we use \eqref{miracle} to write
$$
e_{Q}= \frac{\dot{c}}{2c} \partial_{y}(y Q_{c}) + \dot h  \partial_{y} Q_{c},
$$
and see that
$$
\int_{- \infty}^y \big( e_{Q}(t,y^\prime) + (S \tilde e_{Q})(t,y^\prime) \big) \, dy^\prime
$$
is an exponentially decreasing function at $\pm \infty$.
Therefore, thanks to \eqref{estdebase} and \eqref{h"c"}, we obtain
$$
\left\|   \int_{- \infty}^{y} \big( e_{Q}(t,y^\prime) + (S \tilde e_{Q})(t,y^\prime) \big) \, dy^\prime  \right\|_{L^2_{a}}
        \lesssim { \tilde \eps_{1} \over \langle t \rangle^{m-1}} .
$$
Consequently, by using \eqref{StildeN1}, \eqref{StildeN2}, we obtain
$$
{ d \over dt } \int_{\mathbb{R}} e^{a y } | H_{2}|^2 \, dy     + b \| H_{2}\|_{H^1_{a}}^2 \lesssim
  \Big( { \tilde \eps_{1} \over \langle t \rangle^{m-1} } + \| \partial_{y} H_{1} \|_{L^2_{w'}}
 +  \| \partial_{y} H_{2} \|_{L^2_{a}} \Big) \| H_{2}\|_{L^2_{a}}.
$$
By integrating in time and using Young's inequality, we obtain
$$ \| H_{2}(t) \|_{L^2_{a}}^2 +  b \int_{0}^t \|H_{2}\|_{H^1_{a}}^2 \lesssim \tilde \eps_{1}^2 + \int_{0}^t  \big( \| \partial_{y} H_{1} \|_{L^2_{w'}}^2 +  \| \partial_{y} H_{2} \|_{L^2_{a}}^2 \big)\, d\tau.
$$
By using \eqref{dyH2pf} and Lemma \ref{lemtildev1}, we finally get that
$$
\int_{0}^t   \|H_{2}\|_{H^1_{a}}^2 \lesssim \tilde \eps_{1}^2.
$$
This completes the proof of the Lemma.
\end{proof}

\subsubsection*{Step 3: Estimates of $\widetilde{v}$}
Let us recall that $\tilde v (t,x)$ is defined in \eqref{udec2}. We can write the equation for $\tilde v$ under the form
  \beq
  \label{eqtildev}
  \partial_{t} \tilde v + \partial_{x}^3 \tilde v + \partial_{x}( \tilde v^3)=- \partial_{x}\big( ( \tilde Q_{c} + \tilde v)^3 - \tilde Q_{c}^3
   - \tilde v^3\big)  + \dot h  \partial_{x} \tilde Q_{c} +  \dot c \partial_{c}  \tilde Q_{c} = \partial_{x} K.
   \eeq
   where
   \beq
   \label{defK}
   K :=  - ( \tilde Q_{c} + \tilde v)^3 + \tilde Q_{c}^3
   +  \tilde v^3 + \dot h \tilde Q_{c} + \frac{ \dot c}{2c} \,   \widetilde {(y Q_{c})}(t,x).
   \eeq
Above we have used the notation
$\widetilde {(y Q_{c})}(t,x)=  (x- \int_{0}^t c + h) Q_{c}(x- \int_{0}^t c + h)$.
Note that in order to write the right hand side of \eqref{eqtildev} as the derivative of a localized function,
we have used \eqref{miracle}.

We now study the profile $f$ of $\tilde v$ defined by $f=  e^{ t \partial_{x}^3} \tilde v$.
By taking $\eps_{1}$ sufficiently small but such that $ \eps_{0} \ll \tilde \eps_{1} \ll \eps_{1}$, we will prove the following:
\begin{lemma}
\label{lemfinal}
We have the estimate
$$  {\|\tilde  v \|}_X := \sup_t \big( \| \tilde v \|_{H^1}
  + {\langle t \rangle}^{-\frac{1}{6}} {\big\| x f \big\|}_{H^1}
  + \langle t \rangle^{\frac{\alpha}{3}-\frac{1}{6}} {\big\||\partial_x|^\alpha xf \big\|}_{L^2}
  + {\big\| \widehat{f}(\xi) \big\|}_{L^\infty}  \big) \leq \eps_1.
$$
\end{lemma}

\begin{proof}
We follow the same steps as in the first part of the paper.
Note that the norm $\| \tilde v \|_{H^1}$ is already estimated in view of \eqref{estdebase}.
Moreover, by combining Lemma \ref{lemtildev1} and Lemma \ref{lemtildev2} we have that
\beq
\label{weightedtildev}
\int_{0}^t \|   (I S \tilde v)(\tau, y) \|_{H^1_{w'}}^2 \, d\tau \lesssim  \tilde \eps_{1}^2.
\eeq

Still following the steps of the first  part of the paper, we shall first estimate $ \| x f \|_{H^1}$, using again $I S \tilde v$.
We note that $I S \tilde v$ solves
  $$ \partial_{t} I S \tilde v  + \partial_{x}^3I S \tilde v + 3 \tilde v^2 \partial_{x} I S \tilde v=
   (S+2)K$$
    An energy estimate yields
  \beq
  \label{energiefinale1} {d \over dt } {1 \over 2 } \| I S \tilde v \|_{L^2}^2
   \lesssim { \eps_{1}^2 \over t} \| I S \tilde  v \|_{L^2}^2 + \int_{\mathbb{R}} \left( |S K| + |K| \right) | I S \tilde v| \, dx.
    \eeq
By using \eqref{estdebase}, we obtain
$$  \int_{\mathbb{R}} \left( |S K| + |K| \right) | I S \tilde v| \lesssim \| I S \tilde v (t,y) \|_{L^2_{w'}} \left(
{ \tilde \eps_{1} \over \langle  t \rangle^{m-1}}  + \langle t \rangle( |\ddot{c}(t)|  +  |\ddot{h}(t)|) +  \| S  \tilde v  (t,y)\|_{L^2_{w'}}\right)
$$
and thus by \eqref{weightedtildev}  and \eqref{h"c"}, we obtain
$$
\int_{0}^t  \int_{\mathbb{R}} \left( |S K| + |K| \right) | I S \tilde v (t,y)| \,dy\,dt \lesssim  \tilde \eps_{1}^2.
$$
Consequently, by integrating \eqref{energiefinale1} in time, we obtain
\beq
\label{energiefinale2}
\| I S \tilde v (t)  \|_{L^2}^2 \lesssim \langle t \rangle^{ C  \eps_{1}^2}  \left(  \eps_{0}^2 + \tilde \eps_{1}^2 \right).
\eeq
Next, we observe that
$$ x f =  e^{t \partial_{x}^3} I S \tilde v - 3 t  I \partial_{t} f=   e^{t \partial_{x}^3} I S \tilde v
  -  3 t e^{t \partial_{x}^3}\big(- \tilde v^3  + K\big).$$
Using \eqref{oriole}  we get
$$
t  \| \tilde v^3 \|_{L^2} \lesssim  { \eps_{1}^3 \langle t \rangle^{1 \over 6}},
$$
and thanks to \eqref{estdebase} we obtain
$$
t \| K \|_{L^2} \lesssim {\tilde \eps_{1} \over \langle t \rangle^{m-1} } .
$$
Combining these estimates gives
$$
\| x f \|_{L^2} \lesssim \tilde \eps_{1} +  \eps_{1}^{ 3} t ^{1 \over 6}.
$$

Finally, we can  estimate $S \tilde v$ in $L^2$: first, we note that $S \tilde {v}$ solves
\begin{align}
\label{eqStildev}
\partial_{t}  S \tilde v  + \partial_{x}^3  S \tilde v + 3  \partial_{x}(\tilde v^2  S \tilde v)
  =  \partial_{x}(S+2)K.
\end{align}
From an energy estimate, we find after integrating by parts
$$ {d \over dt } {1 \over 2 } \|  S \tilde v \|_{L^2}^2
  \lesssim { \eps_{1}^2 \over t} \|  S \tilde  v \|_{L^2}^2 + \|  S \tilde  v \|_{L^2_{w'}}
  \left( { \tilde \eps_{1} \over \langle  t \rangle^{m-1}}  +  \langle t \rangle( |\ddot{c}(t)|  +  |\ddot{h}(t)|)  +   \| S  \tilde v  (t,y)\|_{L^2_{w'}}\right) ,
$$
and hence \eqref{weightedtildev} and  \eqref{h"c"} yield
\beq
\label{energiefinale3}
\|  S \tilde v (t)  \|_{L^2}^2 \lesssim \langle t \rangle^{ C  \eps_{1}^2}  \left(  \eps_{0}^2 + \tilde \eps_{1}^2 \right).
\eeq
 By using
 $$  |\partial_{x}|^\alpha( x f)  =   e^{t \partial_{x}^3} |\partial_{x}|^\alpha ( I S \tilde v )
  -  3 t e^{t \partial_{x}^3}\big(- |\partial_{x}^\alpha | \tilde v^3  +  |\partial_{x}^\alpha |K\big), $$
  we get that
  $$ \| |\partial_{x}|^\alpha( x f) \|_{L^2} \lesssim  \|I  S \tilde v \|_{H^1} +  t \|  |\partial_{x}^\alpha | \tilde v^3 \|_{L^2}
   + t  \|K\|_{H^1}.$$ 
Then, by using \eqref{energiefinale3}, \eqref{energiefinale2},  \eqref{estv1fin2} and
    $$ \| K \|_{H^1} \lesssim  { \tilde \eps_{1} \over \langle t \rangle^{m}}
  + ( 1 + \tilde \eps_{1}) \| v \|_{H^1_{w}}
       \lesssim { \tilde \eps_{1} \over \langle t \rangle^{m}},
       $$
       we finally get
$$
\|| \partial_x|^\alpha (x f) \|_{L^2} \lesssim \tilde \eps_{1} +  \eps_{1}^{ 3} t ^{ {1 \over 6 } - { \alpha \over 3} }.
$$
Note that another way to get this estimate (that avoids using the  $H^1$ regularity of the initial data in this step)  would be to start
 from \eqref{eqStildev} and to  follow the same steps as in the proof of \eqref{lemee3} in Lemma \ref{lemee}.

%\comment{F: Here I just said that we can follow the same steps in the proof of the analogous bound \eqref{lemee3} for $u$.
%However, we also need a bound for $|\partial_x |^\alpha S K$, like the one for $|\partial_x |^\alpha K$ above.
%
%PG: I agree, a bound for $|\partial_x|^\alpha S K$ is missing here. A little above, to estimate $S \tilde v$ in $L^2$, a bound for $\partial_x S K$ in $L^2$ or weighted $L^2$ is needed. I do not see an easy way out. Maybe we need to reintroduce the $H^2_{w'}$ bounds that we had. What do you guys think?}

It remains to estimate $\| \hat f \|_{L^\infty}.$ We can follow the proof of Proposition \ref{prohatf}.
The equation for $\hat f$ is now
$$
\partial_t \widehat{f}(t,\xi)  = \frac{i}{2\pi} \iint e^{-it\phi(\xi,\eta,\sigma)} \xi
  \widehat{f}(\xi-\eta-\sigma) \widehat{f}(\eta) \widehat{f}(\sigma)\,d\eta\,d\sigma +  G(t, \xi),
$$
where
$$ G(t, \xi) = \mathcal{F} \big(  e^{t \partial_{x}^3} \partial_{x} K \big)(\xi).$$
We can estimate
$$ \| G(t) \|_{L^\infty} \leq  \| \mathcal{F} (\partial_{x} K)\|_{L^\infty}
  \lesssim  \| \partial_{x} K \|_{L^1} \lesssim  | \dot c |  +  | \dot h | +  ( 1 + \tilde \eps_{1}) \| v (t,y) \|_{H^1_{w}}
$$
by using the exponential decay provided by  $Q_{c}$, and hence deduce
$$
\| G(t) \|_{L^\infty} \lesssim  { \tilde \eps_{1}\over \langle t \rangle^{m}}.
$$
This term is thus integrable in time for $m>1$ and does not affect the arguments in the proof of Proposition \ref{prohatf}.
This completes the proof of the statement, and gives Theorem \ref{maintheo2}.
\end{proof}

\appendix
\section{Auxiliary Lemmas}

In this appendix we gather several lemmas that are used throughout the paper.
First, a lemma about stationary phase.

\begin{lemma}[Stationary phase in dimension 2]
\label{pinguin2}
Consider $\chi \in \mathcal{C}_0^\infty$ such that $\chi = 0$ in $B(0,2)^c$, and $|\nabla \chi|+ |\nabla^2 \chi|\lesssim 1$;
and $\psi \in \mathcal{C}^\infty$ such that $| \det \operatorname{Hess} \psi| \geq 1$,
and $|\nabla \psi| + |\nabla^2 \psi| + |\nabla^3 \psi| \lesssim 1$. Let
$$
I = \iint e^{i\lambda \psi(\eta,\sigma)} F(\eta,\sigma) \chi(\eta,\sigma)\,d\eta \,d\sigma.
$$
Then, for any $\alpha \in [0,1]$,

\setlength{\leftmargini}{2em}
\begin{itemize}
\item[(i)] If $\nabla \psi$ only vanishes at $(\eta_0,\sigma_0)$,
$$
I = \frac{2\pi e^{i\frac{\pi}{4}s}}{\sqrt{\Delta}}\frac{e^{i\lambda \psi(\eta_0,\sigma_0)}}{\lambda}  F(\eta_0,\sigma_0) \chi(\eta_0,\sigma_0)
+ O \left( \frac{\left\| \langle (x,y) \rangle^{2 \alpha} \widehat{F} \right\|_{L^1}}{\lambda^{1+\alpha}} \right),
$$
where $s = \operatorname{sign} \operatorname{Hess} \psi$
and $\Delta = |\operatorname{det} \operatorname{Hess} \psi|$.
\item[(ii)] If $|\nabla \psi|\geq 1$,
$$
I =  O \left( \frac{\left\| \langle (x,y) \rangle^{\alpha} \widehat{F} \right\|_{L^1}}{\lambda^{1+\alpha}} \right).
$$
\end{itemize}
\end{lemma}

\begin{proof}
$(i)$ We assume for simplicity that $\eta_0 = \sigma_0 = \psi(\eta_0,\sigma_0) = 0$. If necessary, it is possible to restrict the support of $\chi$ to an arbitrarily small neighborhood of $0$, since the remainder can be treated by $(ii)$.
By Plancherel's theorem,
\begin{equation}
\label{robin}
I = \frac{1}{2\pi} \iint \underbrace{ \left[ \iint e^{i[\lambda \psi(\eta,\sigma) - x\eta - y \sigma ]}
\chi(\eta,\sigma)\,d\eta \,d\sigma \right]}_{K(x,y)} \widehat{F} (x,y) \,dx\,dy.
\end{equation}
The function $K$ can be written
$$
K(x,y) =  \iint e^{i\lambda \Phi_{X,Y}(\eta,\sigma)}
\chi(\eta,\sigma)\,d\eta \,d\sigma
$$
with $\Phi_{X,Y}(\eta,\sigma) = \psi(\eta,\sigma) - X \eta - Y \sigma$ and $X = x/\lambda$, $Y = y/\lambda$.
If $X$ or $Y$ is larger than $2 \max_{\operatorname{Supp} \chi}|\nabla \psi|$, it is easy to bound $K$ by integrating by parts
in the above integral.
Thus we can assume that $X$, $Y$ are less than $2 \max_{\operatorname{Supp} \chi} |\nabla \psi|$.
Since $\operatorname{Supp} \chi$ can be chosen to be a small neighborhood of $0$, we can assume that $X$ and $Y$ are small.

By assumption, $\operatorname{Hess} \Phi$ is non-degenerate, thus by the implicit function theorem there exists
$(\bar{\eta}(X,Y),\bar{\sigma}(X,Y))$ such that
$$
\nabla_{\eta,\sigma} \Phi_{X,Y} (\bar{\eta}(X,Y),\bar{\sigma}(X,Y)) = 0
$$
and furthermore
\begin{equation}
\label{cardinal}
|\bar{\eta}(X,Y)| + |\bar{\sigma}(X,Y)| \lesssim |X| + |Y| \quad \mbox{and} \quad
|\Phi_{X,Y}(\bar{\eta}(X,Y),\bar{\eta}(X,Y))| \lesssim |X|^2 + |Y|^2.
\end{equation}
By the stationary phase lemma,
$$
K(x,y) = \frac{2\pi e^{i\frac{\pi}{4}s}}{\sqrt{\bar \Delta}}
\frac{e^{i\lambda \Phi(\bar{\eta},\bar{\sigma})}}{\lambda}  \chi(\bar{\eta},\bar{\sigma})
+ O\left( \frac{1}{\lambda^2} \right),
$$
where $\bar \Delta = |\det \operatorname{Hess} \psi|(\bar \eta,\bar \sigma)$
Coming back to~\eqref{robin},
\begin{equation*}
\begin{split}
I & =2\pi e^{i\frac{\pi}{4}s} \frac{1}{\lambda} \iint \widehat{F}(x,y) e^{i\lambda \Phi(\bar{\eta},\bar{\sigma})}
\frac{1}{\bar \Delta} \chi(\bar{\eta},\bar{\sigma})\,dx\,dy + O\left( \frac{\|\widehat{F}\|_{L^1}}{\lambda^2} \right)
\\
& = \frac{2\pi e^{i\frac{\pi}{4}s}}{\sqrt{\Delta}} \frac{1}{\lambda}F(0,0) \chi(0,0)
\\
& + \frac{e^{i\frac{\pi}{4}s}}{\sqrt{\Delta}} \frac{1}{\lambda} \iint \widehat{F}(x,y)
\left[ e^{i\lambda \Phi(\bar{\eta},\bar{\sigma})} \frac{\chi(\bar{\eta},\bar{\sigma})}{\bar \Delta} -\frac{\chi(0,0)}{\Delta} \right] \,dx\,dy
+ O\left( \frac{\|\widehat{F}\|_{L^1}}{\lambda^2} \right)
\end{split}
\end{equation*}
Therefore, using~\eqref{cardinal}, we find for any $\alpha \in [0,1]$
\begin{equation*}
\begin{split}
& \left| I - \frac{2\pi e^{i\frac{\pi}{4}s}}{\sqrt{\Delta}}\frac{1}{\lambda} F(0,0)
+ O\left( \frac{\|\widehat{F}\|_{L^1}}{\lambda^2} \right) \right|
\\
& \lesssim \frac{1}{\lambda}  \iint | \widehat{F}(x,y) |
  \left[ |e^{i\lambda \Phi(\bar{\eta},\bar{\sigma})} - 1| + |\chi(\bar{\eta},\bar{\sigma}) - \chi(0,0)| \right]\,dx dy
\\
& \lesssim \frac{1}{\lambda} \iint | \widehat{F}(x,y) | \left[ \frac{|(x,y)|^{2 \alpha}}{\lambda^{\alpha}}
  + \frac{|(x,y)|^\alpha}{\lambda^\alpha} \right] \,dx dy
  \lesssim \frac{1}{\lambda^{1+\alpha}} \| {\langle (x,y) \rangle}^{2\alpha} \widehat{F} \|_{L^1},
\end{split}
\end{equation*}
which is the desired result.

\bigskip

\noindent
$(ii)$ A direct stationary phase estimate as above, using only the non-degeneracy of $\operatorname{Hess} \psi$, gives the bound
$$
|I| \lesssim \frac{1}{\lambda} \| \widehat{F} \|_{L^1}.
$$
Furthermore, since we are now assuming that $\psi$ does not have stationary points, it is possible to integrate by parts in $I$ before
applying this stationary phase estimate. This gives
$$
|I| \lesssim \frac{1}{\lambda^2} \| \langle (x,y) \rangle \widehat{F} \|_{L^1}.
$$
Interpolating between these two inequalities gives the desired estimate.
\end{proof}

The following lemma gives some bounds on pseudo-product operators satisfying certain strong integrability conditions:
%which is the bilinear generalization of the Mihlin-Hormander condition
\begin{lemma}[Bounds on pseudo-product operators]\label{penguin3}
Assume that $m \in L^1(\mathbb{R}\times\mathbb{R})$ satisfies
\begin{equation}\label{touse1}
\Big\|\,\int_{\mathbb{R}\times\mathbb{R}}m(\eta,\sigma)e^{ix\eta}e^{iy\sigma}\,d\eta d\sigma\Big\|_{L^1_{x,y}}\leq A \, ,
\end{equation}
for some $A > 0$.
Then, for all $p,q,r \in [1,\infty]$ such that $1/p+1/q=1/r$ one has
\begin{equation}
\label{touse1.5}
{\| T_m(f,g) \|}_{L^r} \lesssim A {\|f\|}_{L^p} {\|g\|}_{L^q}.
\end{equation}

Moreover, if $1/p+1/q+1/r = 1$
\begin{equation}\label{touse2}
 \Big|\int_{\mathbb{R}\times\mathbb{R}}\widehat{f}(\eta)\widehat{g}(\sigma)\widehat{h}(-\eta-\sigma)m(\eta,\sigma)\,d\eta d\sigma\Big|
\lesssim A\|f\|_{L^p}\|g\|_{L^q}\|h\|_{L^r}.
\end{equation}
\end{lemma}

\begin{proof}
We rewrite
\begin{equation*}
\begin{split}
\Big|\int_{\mathbb{R}\times\mathbb{R}}\widehat{f}(\eta)\widehat{g}(\sigma)\widehat{h}(-\eta-\sigma)m(\eta,\sigma)\,d\eta d\sigma\Big|
&=C\Big|\int_{\mathbb{R}^3}f(x)g(y)h(z)K(z-x,z-y)\,dxdydz\Big|,\\
&\lesssim \int_{\mathbb{R}^3}|f(z-x)g(z-y)h(z)|\,|K(x,y)|\,dxdydz,
\end{split}
\end{equation*}
where
\begin{equation*}
 K(x,y):=\int_{\mathbb{R}\times\mathbb{R}}m(\eta,\sigma)e^{ix\eta}e^{iy\sigma}\,d\eta d\sigma.
\end{equation*}
The desired bound \eqref{touse2} follows easily from \eqref{touse1}.
The bilinear estimate \eqref{touse1.5} can be proven similarly using a duality argument.
\end{proof}

Finally, for the energy estimates, we need  the following lemma.

\begin{lemma}\label{lemmacomm}
The following commutator estimate holds:
\begin{align*}
 {\| \left[ |\partial_x|^{\alpha-1} \partial_x , P_{\ll j} w \right] P_j f \|}_{L^2}
  \lesssim 2^{(\alpha-1)j} {\| \partial_x w\|}_{L^\infty} {\| P_j f \|}_{L^2}.
\end{align*}
\end{lemma}

\begin{proof}
Denoting $\widetilde P_j$ for the Fourier multiplier with symbol
$$\widetilde \chi \left( \frac{\xi}{2^j} \right) = \frac{\xi |\xi|^{\alpha-1}}{2^{\alpha j}} \left[  \chi \left( \frac{\xi}{2^{j+10}} \right) -  \chi \left( \frac{\xi}{2^{j-10}} \right) \right],$$
observe that
\begin{align*}
\big[ |\partial_x|^{\alpha-1} \partial_x, P_{\ll j} w \big] P_j f (x)
  & = 2^{\alpha j} \big[ \widetilde{P}_j, P_{\ll j} w \big] P_j f
  =  2^{\alpha j} \big[ 2^j \widehat{\widetilde{\chi}}(2^j \cdot) * , P_{\ll j} w \big] P_j f
\\
  & =2^{\alpha j} \int 2^j \widehat{\widetilde{\chi}}(2^j(x-y)) \big[ P_{\ll j} w(x) - P_{\ll j} w(y) \big] P_j f(y) \,dy .
\end{align*}
Thus
$$
\left| \big[ |\partial_x|^{\alpha-1} \partial_x , P_{\ll j} w \big] P_j f (y) \right|
  \lesssim 2^{\alpha j} \| \partial_x w\|_\infty \int 2^j |\widehat{\widetilde{\chi}}(2^j(x-y))| |x-y| |P_j f(y)|\,dy,
$$
and the desired result follows by Young's inequality.
\end{proof}

\end{document}